\documentclass[11pt,a4paper]{article}

\usepackage{natbib}
\usepackage{apalike}
\usepackage{comment}
\usepackage{lipsum}

\usepackage[labelfont=bf]{caption}
\usepackage
[
        a4paper,
        left=3cm,
        right=3cm,
        top=3cm,
        bottom=3cm,
]
{geometry}

\usepackage{setspace}
\usepackage{multicol,latexsym, array}
\usepackage{textcomp}
\usepackage[ansinew]{inputenc}
\usepackage{amsmath,amssymb, amsthm}
\usepackage[OT1]{fontenc}
\usepackage{hyperref}

\hypersetup{
    colorlinks,
    linkcolor={blue},
    citecolor={blue},
    urlcolor={blue}
}
\usepackage{enumitem}
\usepackage{MnSymbol} 
\usepackage{empheq}
\usepackage{rotating}
\usepackage{titletoc}
\usepackage{framed}
\usepackage{multirow,bigdelim}
\usepackage{subfigure}
\usepackage{bbm}
\usepackage[noend]{algpseudocode}
\usepackage{algorithm}
\usepackage{pgf}
\usepackage{pgfplots}
\usepackage{tikz}
\usepackage{shadethm}
\usepackage{thmtools}
\usepackage{mdframed}
\usepackage{lipsum}
\usepackage{fixmath}
\usepackage{multicol}

\usepackage{caption}
\usepackage{xargs}                      
\usepackage{xcolor}  
\usepackage[colorinlistoftodos,prependcaption,textsize=tiny]{todonotes}

\providecommand{\keywords}[1]{\textbf{\textit{Keywords: }} #1}

\date{December 25, 2022}

\newcommand\acknow{Acknowledgements}
    {\begin{center}%
    \large{\bfseries{\acknow}} \end{center}}

\usepackage{caption}
\usepackage{mathrsfs}
\usetikzlibrary{arrows}

\usepackage{xpatch}

\makeatletter
\xpatchcmd{\endmdframed}
  {\aftergroup\endmdf@trivlist\color@endgroup}
  {\endmdf@trivlist\color@endgroup\@doendpe}
  {}{}
\makeatother

\usepackage{nicefrac}
\usepackage{dsfont}

\usepackage{bigints}
\usepackage{mathtools}
\usepackage{fixltx2e}

\usepackage{mathtools}
\usepackage{color}
\usepackage{fancyhdr}
\usepackage{lastpage}

\usepackage{ifthen}

\newcommand{\OT}{{OT}}
\newcommand{\EROT}{{E\!ROT}^{\lambda}}
\newcommand{\EROTLambda}[1]{{E\!ROT}^{#1}}
\newcommand{\SOT}{{S}^{\lambda}}
\newcommand{\SOTLambda}[1]{{S}^{#1}}

\newcommand{\R}{\mathbb{R}}
\newcommand{\RR}{\mathbb{R}}

\newcommand{\N}{\mathbb{N}}
\newcommand{\NN}{\mathbb{N}}

\newcommand{\Z}{\mathbb{Z}}
\newcommand{\Indicator}[1]{\mathds{1}_{#1}}
\newcommand{\XF}{\mathcal{X}}

\newcommand{\XC}{\mathcal{X}}
\newcommand{\RC}{\mathcal{R}}
\newcommand{\tXC}{\hat{\mathcal{X}}}
\newcommand{\YC}{\mathcal{Y}}
\newcommand{\YCC}{\mathcal{Y}\backslash\{y_1\}}
\newcommand{\tYC}{\hat{\mathcal{Y}} }
\newcommand{\tYCC}{\hat{\mathcal{Y}}\backslash\{y_1\}}

\newcommand{\FC}{\mathcal{F}}
\newcommand{\HC}{\mathcal{H}}

\newcommand{\Gb}{\boldsymbol{G}}

\newcommand{\ACX}{\mathcal{A}^\XC_{\rb, \sb}}
\newcommand{\ACY}{\mathcal{A}^\YC_{\rb, \sb}}
\newcommand{\AC}{\mathcal{A}_{\rb, \sb}}
\newcommand{\BCX}{\mathcal{B}^\XC_{\rb, \sb}}
\newcommand{\BCY}{\mathcal{B}^\YC_{\rb, \sb}}
\newcommand{\BC}{\mathcal{B}_{\rb, \sb}}
\newcommand{\tACX}{\tilde{\mathcal{A}}^\XC_{\rb, \sb}}
\newcommand{\tACY}{\tilde{\mathcal{A}}^\YC_{\rb, \sb}}
\newcommand{\tAC}{\tilde{\mathcal{A}}_{\rb, \sb}}

\newcommand{\hAC}{\hat{\mathcal{A}}_{\hat\rb, \hat\sb}}
\newcommand{\hBC}{\hat{\mathcal{B}}_{\hat\rb, \hat\sb}}

\renewcommand{\Gb}{G}

\newcommand{\lamdba}{\lambda}
\newcommand{\supp}{\textup{supp}}

\newcommand{\landau}{\mathcal{O}}

\newcommand{\norm}[1]{\left\lVert#1\right\rVert}
\newcommand{\normal}{\mathcal{N}}

\newcommand{\Cov}{\operatorname{Cov}}

\newcommand{\EV}[1]{\mathbb{E}\left[#1 \right]}

\newcommand{\PC}{\mathcal{P}}
\newcommand{\probset}[1]{\mathcal{P}\left(#1\right)}

\newcommand{\sign}{\textup{sign}}
\newcommand{\iid}{\stackrel{\!\textup{i.i.d.}}{\sim}\!}

\newcommand{\DF}{\mathcal{D}^F}
\newcommand{\DCW}{\mathcal{D}^{CW}}
\renewcommand{\DCW}{\mathcal{D}^{}}
\renewcommand{\DH}{\mathcal{D}^H}
\newcommand{\ind}{{k}}

\newcommand{\rb}{{r}}
\renewcommand{\sb}{{s}}

\newcommand{\hb}{{h}}
\newcommand{\cb}{{c}}

\newcommand{\ab}{{a}}
\newcommand{\Ab}{{A}}
\newcommand{\bb}{{b}}

\newcommand{\pib}{{\pi}}

\newcommand{\xib}{{\xi}}
\newcommand{\varthetab}{{\vartheta}}
\newcommand{\zetab}{{\zeta}}
\newcommand{\alphab}{{\alpha}}
\newcommand{\betab}{{\beta}}

\newcommand{\coloneqq}{:=}
\newcommand{\eqqcolon}{=:}
\newcommand{\Id}{\textup{Id}}
\newcommand{\lX}{{\ell^1(\XC)}}
\newcommand{\lY}{{\ell^1(\YC)}}
\newcommand{\lYY}{{\ell^1(\YCC)}}
\newcommand{\lInfX}{{\ell^\infty(\XC)}}
\newcommand{\lInfY}{{\ell^\infty(\YC)}}
\newcommand{\lInfYY}{{\ell^\infty(\YCC)}}
\newcommand{\lXY}{{\ell^1(\XC\times\YC)}}
\newcommand{\lr}{{\ell^1_\rb(\XC)}}
\newcommand{\ls}{{\ell^1_\sb(\YC)}}
\newcommand{\lInfXY}{{\ell^\infty(\XC\times \YC)}}

\newcommand{\lXSub}[1]{{\ell^1_{#1}(\XC)}}
\newcommand{\lYSub}[1]{{\ell^1_{#1}(\YC)}}
\newcommand{\lYYSub}[1]{{\ell^1_{#1}(\YCC)}}
\newcommand{\lInfXSub}[1]{{\ell^\infty_{#1}(\XC)}}
\newcommand{\lInfYSub}[1]{{\ell^\infty_{#1}(\YC)}}
\newcommand{\lInfYYSub}[1]{{\ell^\infty_{#1}(\YCC)}}
\newcommand{\lXYSub}[1]{{\ell^1_{#1}(\XC\times\YC)}}

\newcommand{\lInfXYSub}[1]{{\ell^\infty_{#1}(\XC\times \YC)}}

\newcommand{\cxp}{{c_\XC^+}}
\newcommand{\cxm}{{c_\XC^-}}
\newcommand{\cyp}{{c_\YC^+}}
\newcommand{\cym}{{c_\YC^-}}

\newcommand{\tcxp}{{\tilde c_\XC^+}}
\newcommand{\tcxm}{{\tilde c_\XC^-}}
\newcommand{\tcyp}{{\tilde c_\YC^+}}
\newcommand{\tcym}{{\tilde c_\YC^-}}

\newcommand{\cx}{{\mathcal{C}_\XC}}
\newcommand{\cy}{{\mathcal{C}_\YC}}
\newcommand{\tcx}{{\tilde{\mathcal{C}}_\XC}}
\newcommand{\tcy}{{\tilde{\mathcal{C}}_\YC}}

\newcommand{\cxpb}{{{c}_\XC^+}}
\newcommand{\cxmb}{{{c}_\XC^-}}

\newcommand{\cypb}{{{c}_\YC^+}}
\newcommand{\cymb}{{{c}_\YC^-}}

\newcommand{\kxSup}[1]{{\Phi_{\XC}^{#1}}}
\newcommand{\kySup}[1]{{\Phi_{\YC}^{#1}}}

\newcommand{\exSup}[1]{{\phi_{\XC}^{#1}}}
\newcommand{\eySup}[1]{{\phi_{\YC}^{#1}}}

\newcommand{\texSup}[1]{{\tilde \phi_{\XC}^{#1}}}
\newcommand{\teySup}[1]{{\tilde \phi_{\YC}^{#1}}}

\newcommand{\equalDistr}{\stackrel{\;\mathcal{D}\;}{=}}

\newcommand*{\D}{\mathcal D} 
\newcommand{\konvD}{\xrightarrow{\;\;\mathcal{D}\;\;}}
\newcommand{\konvW}{\xrightarrow{\;\;\mathcal{D}\;\;}}

 \newcommand{\konvP}{\xrightarrow{\;\;\mathbb{P}^*\;\;}}

\renewcommand*{\epsilon}{\varepsilon}

\newcommand{\closure}{\overline{\textup{Cl}}}

\newcommand{\BL}[1]{\textup{BL}_{1}(#1)}

\DeclareMathOperator*{\argmin}{\textrm{argmin}}

\theoremstyle{plain}
\newtheorem{lemma}{Lemma}[section]
\newtheorem{theorem}[lemma]{Theorem}
\newtheorem{proposition}[lemma]{Proposition}
\newtheorem{corollary}[lemma]{Corollary}
\theoremstyle{remark}
\newtheorem{remark}[lemma]{Remark}

\newtheorem{definition}[lemma]{Definition}

\usepackage{titlesec}

\allowdisplaybreaks

\titleformat{\chapter}[display]
  {\filleft\bfseries\huge}
  {\centering{\Large\chaptertitlename~\thechapter}}
  {1ex}
  {\titlerule\vspace{1.5ex}\filright}
  [\vspace{1ex}\titlerule]

\numberwithin{equation}{section}

\newcommand{\footremember}[2]{
	\footnote{#2}
	\newcounter{#1}
	\setcounter{#1}{\value{footnote}}
}

\newcommand{\footrecall}[1]{
	\footnotemark[\value{#1}]
}

\providecommand{\keywords}[1]{\textbf{\textit{Keywords}}#1}
\providecommand{\keywordsMSC}[1]{\textbf{\textit{MSC 2020 subject classification}} #1}

\begin{document}

\author{Shayan Hundrieser \footremember{ims}{\scriptsize Institute for Mathematical
		Stochastics, University of G\"ottingen,
		Goldschmidtstra{\ss}e 7, 37077 G\"ottingen} \!\!\!\footremember{mbexc} {\scriptsize Cluster of Excellence "Multiscale Bioimaging: from Molecular Machines to Networks of Excitable Cells" (MBExC), University of G\"ottingen, Germany}
	\and 
	Marcel Klatt \footrecall{ims}{}
	\and 
	Axel Munk \footrecall{ims} \!\footrecall{mbexc} \footnote{\scriptsize Max Planck Institute for Multidisciplinary Sciences, Am Fa{\ss}berg 11, 37077 G\"ottingen}}

\title{Limit Distributions and Sensitivity Analysis for Empirical Entropic Optimal Transport on Countable  Spaces}
 
\pagenumbering{arabic}

\maketitle

\begin{abstract}
\noindent For probability measures on countable spaces we derive distributional limits for empirical entropic optimal transport quantities. More precisely, we show that the empirical optimal transport plan weakly converges to a centered Gaussian process and that the empirical entropic optimal transport value is asymptotically normal. The results are valid for a large class of cost functions and generalize distributional limits for empirical entropic optimal transport quantities on finite spaces. Our proofs are based on a sensitivity analysis with respect to norms induced by suitable function classes, which arise from novel quantitative bounds for primal and dual optimizers, that are related to the exponential penalty term in the dual formulation. The distributional limits then follow from the functional delta method together with weak convergence of the empirical process in that respective norm, for which we provide sharp conditions on the underlying measures. As a byproduct of our proof technique, consistency of the bootstrap for statistical applications is shown. 
\end{abstract}

\noindent\keywords{\!\!\!Optimal transport, entropy regularization, central limit theorem, bootstrap, sensitivity analysis}

\noindent\keywordsMSC{Primary: 60B12, 60F05, 62E20\\ Secondary: 90C06, 90C25, 90C31}

\section{Introduction}
Over the last decades, the theory of optimal transport (OT), originating in the seminal work by \cite{monge} and later by \cite{kant58},  has gradually established itself as an active area of modern mathematical research and related areas (see \cite{rachev1998massTheory, rachev1998massApplications}, \cite{vil03, villani2008optimal}, \cite{santambrogio2015optimal}, or \cite{galichon2016optimal} for comprehensive monographs). Recently, OT and variants thereof have also been recognized as an important tool for statistical data analysis, e.g., in genetics \citep{evans2012phylogenetic}, fingerprint identification \citep{sommerfeld2018}, computational biology \citep{Schiebinger19, tameling2021Colocalization}, deformation analysis \citep{zemel2019}, and medical imaging \citep{Chen2020}, among others. However, for routine data analysis the computational speed to solve the underlying linear program is still a bottleneck and the development of algorithms for fast computation is a highly active area of research. 
Up to polylogarithmic terms, common linear OT solvers (for bounded costs) such as the Auction algorithm \citep{bertsekas1981new,bertsekas1989} or Orlin's algorithm \citep{Orlin88} have a worst-case complexity ${\landau}(N^{3})$ for $N$ denoting the size of the data. The best known (theoretical) worst case complexity to solve OT as a linear program is given by ${\mathcal{O}}(N^{2.5})$ \citep{Lee2014} for which, however, no practical implementation is known. 
As an alternative approach \cite{cuturi13} proposed to replace the original OT optimization problem by an entropy regularized surrogate, a well-known approach in linear programming with efficient algorithms whose convergence was already derived by \cite{sinkhorn1964relationship,sinkhorn67} and \cite{cominetti1994asymptotic}. 
In particular, for a given precision  it solves the corresponding optimization problem in about $\landau(N^2)$ elementary operations \citep{altschuler2017near, dvurechensky2018computational}. 
 Since then, entropy regularized OT (EROT) has become a frequently used computational scheme for the approximation of OT \citep{cuturi18, Amari19, Clason21,Tong21}. 

In this paper, we are concerned with EROT on countable spaces $\XC=\{x_1, x_2, \dots \}$ and $\YC = \{ y_1, y_2, \dots \}$ (possibly $\XC = \YC$). 
A probability measure $\rb$ on $\XC$ ($\sb$ on $\YC$) is represented as an element in $\lX$ (resp. $\lY$), the space of absolutely summable sequences indexed over $\XC$ (resp. $\YC)$, such that $\sum_{x \in \XC}r_x = 1$ and $\rb \geq 0$. 
 The set of couplings between $\rb$ and~$\sb$, also known as transport plans, on the product space $\XC\times \YC$  is defined by 
$$\Pi(\rb, \sb) \coloneqq \left\{ \pi \in \lXY \colon  \Ab(\pib) = \begin{pmatrix}
		\rb\\
		\sb
	\end{pmatrix}
, \pi \geq 0 \right\},$$
	where $\Ab$ is the marginalization operator$$\Ab \colon \lXY \rightarrow \lX\times \lY,\quad  \pib \mapsto \begin{pmatrix}(\sum_{y \in \YC}\pi_{xy})_{x \in \XC}\\(\sum_{x \in \XC}\pi_{xy})_{y \in \YC}\end{pmatrix}.$$
	For a cost function $c\,\colon \XC\times \YC \rightarrow \R$ and a regularization parameter $\lambda>0$ the \emph{EROT value} between probability measures $\rb$ and $\sb$ is defined as
	\begin{equation}\label{eq:EntropicOptimalTransport}
\EROT(\rb, \sb) \coloneqq \inf_{\pib \in \Pi(\rb, \sb)} \langle \cb, \pib \rangle  + \lambda M(\pib).\tag{EROT}
\end{equation}
	The quantity $\langle \cb, \pib\rangle=\sum_{(x,y)\in \XC\times \YC}c(x,y) \pi_{xy}$ denotes the total costs associated to a transport plan $\pib\in \Pi(\rb, \sb)$ and $M(\pi)$ represents the mutual information	$$\begin{aligned}%
	 M(\pib) \coloneqq   \sum_{\substack{x\in\XC\\ y \in \YC}}\pi_{xy}\log\left(\frac{\pi_{xy}}{\big(\sum_{y' \in \YC}\pi_{xy'}\big)\big(\sum_{x' \in \XC}\pi_{x'y} \big)}\right) = \sum_{\substack{x\in\XC\\ y \in \YC}}\pi_{xy}\log\left(\frac{\pi_{xy}}{r_x s_y}\right)  \in [0,\infty],  \end{aligned}$$
	 where by convention $0\log(0) = 0$. 
If the cost function $c$ is integrable with respect to $\rb\otimes \sb$, i.e., if it fulfills $c\in \ell^1_{\rb\otimes \sb}(\XC\times \YC)$, then there exists a unique minimizer (Proposition \ref{prop:DualEROT})\begin{equation}
  \pib^\lambda(\rb, \sb) \coloneqq \argmin_{\pib \in \Pi(\rb, \sb)} \;\langle \cb, \pib\rangle + \lambda M(\pib),\label{eq:DefinitionEntropicOptimalTransportPlan}
\end{equation}
 known as \emph{EROT plan}. Plugging $\pib^\lambda(\rb, \sb)$ into the functional for total costs $\langle \cb, \cdot \rangle$ yields the \emph{Sinkhorn costs} %
  $$\SOT(\rb,\sb) \coloneqq \langle \cb, \pib^\lambda(\rb, \sb) \rangle.$$ 
Moreover, \eqref{eq:EntropicOptimalTransport} is a convex optimization problem and hence exhibits a dual formulation \begin{equation}\label{eq:DualEntropicOptimalTransportProblem}
		\sup_{\substack{\alphab\in \lr \\ \betab\in \ls}} \langle\alphab,\rb\rangle + \langle\betab,\sb\rangle - \lambda \bigg[ \sum_{\substack{x \in \XF\\ y \in \YC}}{{\exp\bigg(\frac{\alpha_x + \beta_y - c(x,y)}{\lambda} \bigg)r_x s_y - r_xs_y} } \bigg], \tag{DEROT}
	\end{equation}
	where $\lr$ and $\ls$ denote the spaces of functions on $\XC$, $\YC$ with finite expectation under $\rb$ and $\sb$, respectively (for a general dual formulation on Polish spaces see \cite{chizat16}). 
Optimizers $\alphab \in \lr, \betab\in \ls$ of \eqref{eq:DualEntropicOptimalTransportProblem} are called \emph{EROT potentials} and the quantities $\langle\alphab,\rb\rangle=\sum_{x \in \XC} \alpha_x r_x$, $\langle\betab,\sb\rangle=\sum_{y \in \YC} \beta_y s_y$ denote their expectation with respect to $\rb$ and $\sb$.

Statistical questions arise as soon as the probability measures $\rb$ and $\sb$ are estimated by (discrete) empirical measures 
\begin{equation}
\hat\rb_n = \frac{1}{n}\sum_{i = 1}^{n} \delta_{X_i} \quad \text{ and }\quad \hat \sb_m = \frac{1}{m} \sum_{j = 1}^{m}\delta_{Y_j}
\label{eq:DefEmpiricalMeasures}
\end{equation}
for samples $X_1, \dots, X_n \iid \rb$ and independent  $Y_1, \dots, Y_m \iid \sb$ where $\delta_x$ is the Dirac measure at $x$. This scenario occurs, e.g., if the underlying probability measures are unknown \citep{Genevay2018SampleCO}, when subsampling methods are applied for randomized approximations \citep{sommerfeld19FastProb}, or if statistical inference based on EROT is aimed for \citep{bigot2019CentralLT, klatt2018empirical}. 
 Fundamental to such tasks are asymptotic limit laws of the empirical EROT value and its corresponding plan. 

 For probability measures supported on finitely many points \cite{bigot2019CentralLT} proved that the limit law for the empirical EROT \emph{value} centered by its population counterpart is normal and \cite{klatt2018empirical} showed asymptotic normality for the empirical EROT \emph{plan} and empirical Sinkhorn costs. For the Euclidean space $\R^d$ and squared Euclidean costs \cite{mena2019} derived a normal limit law of the empirical EROT value when sampling from sub-Gaussian probability measures. In contrast to the finite case, the centering constant is given by the expected value of the empirical estimator rather than the population quantity. A refinement for squared Euclidean costs was recently obtained by \cite{del2022improved, goldfeld2022statistical} who showed that the expectation can indeed be replaced by the population quantity if the measures are sub-Gaussian. 
 Nonetheless, an analysis beyond the finite setting or Euclidean spaces with squared Euclidean costs remains largely open. 

 Only recently, parallel and independently to us, \cite{Harchaoui2020} analyzed statistical properties of a \emph{modified} EROT problem that enables an explicit formula for the unique optimal plan between empirical measures. Then, a central limit theorem for the associated modified empirical Sinkhorn cost and related quantities is derived. Their setting subsumes probability measures absolutely continuous with respect to Lebesgue measure and the distributional limit laws are valid for various cost functions. However, this approach seems to be not suitable for an analysis of limit distributions for (non-modified) EROT quantities as the proof crucially depends on the explicit solution for the respective empirical optimal plan. Such an explicit solution is generally not available for the (non-modified) EROT plan between empirical measures.  Still, \cite{Harchaoui2020} conjecture asymptotic normality of empirical non-modified Sinkhorn costs. During the time this work was under review, this conjecture was verified by \cite{gonzalez2022weak} for compactly supported probability measures with squared Euclidean costs, leaving open the question if the conjecture is also fulfilled on non-compact spaces and more general costs.  
 
 Indeed, this is true, at least on \emph{countable} spaces, as we will show in this work. Our theory provides distributional limit both for the one- and two-sample case. For the sake of presentation, we focus for now only on the one-sample case.
 Under suitable assumptions on the cost function and the probability measures $\rb$, $\sb$, with countable support, we prove that the empirical EROT value is asymptotically normal 
 \begin{equation}
	\sqrt{n}\left(\EROT(\hat \rb_n, \sb) - 	\EROT(\rb, \sb)\right) \konvW \mathcal{N}\big(0,\sigma^2_{\lambda}(\rb|\sb)\big),\label{eq:LimitLawEROTC}
\end{equation}
as $n\rightarrow \infty$  (Theorem \ref{them:LimitLawEmpEROT_Value}) with $\konvW$ denoting weak convergence \citep{billingsley1999convergence}. 
The asymptotic variance $\sigma^2_{\lambda}(\rb|\sb)$ depends on the population EROT potential $\alphab^\lambda$ for \eqref{eq:DualEntropicOptimalTransportProblem} and is equal to$$ \sigma^2_{\lambda}(\rb|\sb) = \text{Var}_{X\sim \rb}\left[\alpha^\lambda_X\right].$$  
 Concerning the EROT \emph{plan}, we show for the empirical counterpart centered by its population quantity, as $n \rightarrow \infty$, that  \begin{equation}
	\sqrt{n}\left(\pib^\lambda(\hat \rb_n, \sb) - 	\pib^\lambda(\rb, \sb)\right) \konvW G\big(0,\Sigma_{\lambda,\pib^\lambda}(\rb|\sb)\big), 
	\label{eq:LimitLawEROTP}
\end{equation}
for a centered Gaussian process $G$ with covariance $\Sigma_{\lambda,\pib^\lambda}(\rb|\sb)$ 
(Theorem \ref{them:LimitLawEmpEROT_Plan}).
  The covariance  $\Sigma_{\lambda,\pib^\lambda}(\rb|\sb)$ can be stated explicitly and depends on the regularization parameter $\lambda$ and the EROT plan $\pib^\lambda$ between $\rb$ and $ \sb$. Notably, weak convergence in \eqref{eq:LimitLawEROTP} takes place in a suitable weighted $\ell^1$-space over $\XC\times \YC$. This enables us to derive the limit distributions of suitably bounded (classes of) functions integrated with respect to the empirical EROT plan centered by its population version and verifies a conjecture by \cite{Harchaoui2020} on countable spaces. 
  Notably, their assumptions are in line with our conditions; whereas they demand that a certain operator is invertible, we impose sufficient conditions directly on the cost function, which are more easily verifiable. 
   As a corollary, we derive 
   the limit distribution for the empirical Sinkhorn cost on countable spaces (Corollary \ref{cor:LimitLawEmpSinkhornCosts})	\begin{equation}
	\sqrt{n}\left(\SOT(\hat \rb_n, \sb) - 	\SOT(\rb, \sb)\right) \konvD \mathcal{N}\big(0,\tilde\sigma^2_{\lambda,\pib^\lambda}(\rb|\sb)\big),\label{eq:LimitLawSC}
\end{equation}
as $n\rightarrow \infty$. 
The asymptotic variance $\tilde\sigma^2_{\lambda,\pib^\lambda}(\rb|\sb)$ is given by $$\tilde\sigma^2_{\lambda,\pib^\lambda}(\rb|\sb)= \sum_{\substack{x, x' \in \XC\\y, y' \in \YC}} c(x,y) c(x',y')\big(\Sigma_{\lambda,\pib^\lambda}(\rb|\sb)\big)_{(x,y),(x',y') }.$$
Our  limit laws are generically Gaussian for $\rb \neq \sb$ \emph{and} $\rb = \sb$. This is in strict contrast to limit results obtained for the empirical unregularized ($\lambda = 0$) OT value \citep{sommerfeld2018, tameling18,hundrieser2022unifying} and for the OT plan \citep{klatt2020limit} on discrete spaces.
Heuristically speaking, whereas the asymptotic law of the unregularized OT quantities depends on the geometry of the boundary of the underlying transport simplex, the entropy regularization smooths such quantities (unique solutions are attained in the interior of the simplex). Consequently, Gaussian fluctuations in the marginals translate to Gaussian fluctuations of EROT quantities.

 Additionally, %
 our method of proof implies consistency of the na\" ive $n$-out-of-$n$ bootstrap (Theorem \ref{them:BootstrapEROT}), which is also in contrast to unregularized OT \citep{sommerfeld2018,hundrieser2022unifying}. The latter being useful as the limit distributions in \eqref{eq:LimitLawEROTP} and \eqref{eq:LimitLawSC} are in general not accessible and its estimation from data is challenging. 

For our proof technique we rely on a general functional delta method for Hadamard differentiable functionals with respect to norms induced by certain appropriate function classes.  
Our analysis reveals an interesting interplay between the cost function and the probability measures $\rb$ and $\sb$ in order to guarantee weak convergence of empirical EROT quantities, an issue which does not arise for finite ground spaces \citep{bigot2019CentralLT, klatt2018empirical}. 
 To suitably control the growth of the cost function we assume existence of  functions 
$ \cxm, \cxp\colon \XC \rightarrow \R$ and $\cym, \cyp \colon \YC \rightarrow \R$ with $\cxm \leq \cxp$ and $\cym \leq \cyp$ such that 
\begin{equation}
  \cxm(x) + \cym(y) \leq c(x,y) \leq \cxp(x) + \cyp(y)\label{eq:LowerUpperBoundsC}
\end{equation}
for all $(x,y) \in \XC\times \YC$. Our assumptions on the underlying measures to guarantee the distributional limits for the empirical EROT value and plan are determined by these dominating functions. Hence, there is a particular interest in choosing the dominating functions as sharp as possible.
For most cost functions of interest such dominating functions exist. Various examples and their respective dominating functions are detailed in Section~\ref{sec:Examples}. 
The dominating functions on the cost function enable us to state lower and upper bounds for EROT potentials and plan (Proposition \ref{prop:BoundsOptimalPotentials}) which require $\rb$- and $\sb$-integrability of the functions $\cx, \exSup{}\colon \XC\rightarrow [1,\infty)$ and $\cy, \eySup{}\colon \YC\rightarrow [1,\infty)$, respectively, defined by
\begin{equation}
	\begin{aligned}
		\cx(x)&\coloneqq1 + |\cxp(x)| + |\cxm(x)|,&& \cy(y)\coloneqq1 + |\cyp(y)| + |\cym(y)|,\\
		  \exSup{}(x)&\coloneqq\exp\bigg(\,\frac{\cxp(x)- \cxm(x)}{\lambda}\bigg),&&  \eySup{}(y)\coloneqq\exp\bigg(\,\frac{\cyp(y)- \cym(y)}{\lambda}\bigg).
  \end{aligned}\label{eq:DefWeightingFunctions}
  \end{equation}

\begin{table}[t]
	\centering
	\caption{Interplay between cost functions and summability conditions on probability measures $\rb$, $\sb$ for one-sample limit distribution of empirical EROT value.}
	\label{tab:SummaryResultsAndConstrainsEROTValue}
	\begin{tabular}{cll}
	\hline	
	Type & Cost                                     &  Summability Condition                                                            \\ \hline  
	Bounded costs                  & $\norm{c}_{\lInfXY}< \infty$                              &  $\sum_{x \in \XC}\sqrt{r_x}< \infty$                                              \\[0.4cm]
	\multirow{2}{*}{\begin{tabular}{@{}c@{}}\\[-0.2cm]Costs with \\bounded variation\end{tabular}}   & $\norm{\tcxp - \tcxm}_{\lInfX}< \infty$                  &  $\sum_{x \in \XF}\cx(x)\sqrt{r_x} +\tcx^2(x)r_x  < \infty$                                       \\
			 & $\norm{\cypb - \cymb}_{\lInfY}< \infty$                  &  $\sum_{y \in \YC}(\cy +  \tcy)(y)s_y < \infty$ \;\;\;                                             \\[0.4cm]
	\multirow{2}{*}{\begin{tabular}{@{}c@{}}\\[-0.2cm]Costs without \\bounded variation\end{tabular}}   & $\norm{\tcxp - \tcxm}_{\lInfX}= \infty$                  &  $\sum_{x \in \XF}\cx(x)\sqrt{r_x} +(\tcx^2+\texSup{})(x)r_x  < \infty$                                       \\
	 & $\norm{\cypb - \cymb}_{\lInfY}=\infty$                  &  $\sum_{y \in \YC}(\cy +  \tcy + \eySup{})(y)s_y < \infty$ \;\;\;                                            \\[0.2cm]
	\hline
	\end{tabular}
	\end{table}

For the Hadamard differentiability of the EROT value (Theorem \ref{them:EntropicOTCostUnboundedIsHadamardDifferentiable}), we then consider 
two (possibly identical) collections of dominating functions $\cxm, \cxp, \cym, \cyp$ and $\tcxm, \tcxp,\tcym,\tcyp$ of~$c$ (one collection yields bounds for the EROT potential $\alphab^\lambda$, the other for $\beta^\lambda$) with corresponding functions $\tcx, \texSup{},  \tcy, \teySup{}$ as above and define the function classes 
\begin{equation}\label{eq:FunctionClassFC}
	\begin{aligned}
		\FC_{\XC}&\coloneqq \{f\colon \XC\rightarrow \RR \colon |f|\leq \cx\}\cup\{\tcx\}\cup\{\texSup{}\}\cup\{ \texSup{} \Indicator{\XC\backslash\{x_1, \dots, x_n\}} \colon n \in \NN\},\\
		\FC_{\YC}&\coloneqq \{f\colon \YC\rightarrow \RR \colon |f|\leq \tcy\,\}\cup\{\cy\,\}\cup\{\eySup{}\}\cup\{\eySup{} \Indicator{\YC\backslash\{y_1, \dots, y_n\}} \,\,\colon n \in \NN\}.%
	\end{aligned}
	\end{equation}
In particular, to verify Hadamard differentiability with respect to the resulting norm we use the dual formulation \eqref{eq:DualEntropicOptimalTransportProblem}, exploit strong duality, and rely on our quantitative bounds for EROT potentials. 
Hence, for our  distributional limit for the empirical EROT value it is necessary that the function class $\FC_\XC$ is $\rb$-Donsker and that $\sb$ defines a bounded functional on $\FC_\YC$. The latter two conditions are met if and only if (Lemma \ref{lem:WeakConvergence}) 
\begin{align*}
	\sum_{x\in \XC} \cx(x)\sqrt{r_x} + \big(\tcx^2+\texSup{2}\big)(x)r_x <\infty \quad \text{and} \quad \sum_{y\in \YC} \big(\cy+\tcy+\eySup{}\big)(y) s_y <\infty. 
\end{align*}
Multiple special cases and their corresponding assumptions to guarantee the one-sample distributional limits are detailed in Table \ref{tab:SummaryResultsAndConstrainsEROTValue}.
Note, that the different nature of the above conditions is a consequence of the fact that the measure $\rb$ is randomly perturbed (sampled) while $\sb$ is assumed to be fixed. 
For the related two-sample results for which both $\rb$ and $\sb$ are randomly perturbed we require symmetric conditions, i.e,  
\begin{align*}
	\sum_{x\in \XC} \cx(x)\sqrt{r_x} + \big(\tcx^2+\texSup{2}\big)(x)r_x <\infty \quad \text{and} \quad \sum_{y\in \YC} \tcy(y)\sqrt{s_y} + \big(\mathcal{C}^2_{\YC}+\eySup{2}\big)(y)s_y <\infty. 
\end{align*}

For our sensitivity analysis of the EROT plan as an element in the weighted $\ell^1$-space $\lXYSub{\cx\oplus\cy}$ as well as the Sinkhorn cost we only consider a single collection of dominating functions $\cxm, \cxp, \cym, \cyp$ for $c$ and additionally require that 
$$\norm{\cxp - \cxm}_\lInfX<\infty.$$ 
This condition is clearly met for uniformly bounded costs (Section \ref{subsec:Expl:Bounded}) but also for various metric based costs when $\XC$ is bounded while $\YC$ is not (Section \ref{subsec:Expl:BoundedUnbounded}). It is also fulfilled under  metric costs for unbounded spaces $\XC$ and $\YC$ if they fulfill certain geometrical properties (Section \ref{subsec:Expl:Separability}).
It enables us to verify that the Hadamard derivative is well-defined and ensures that a certain operator has a strictly positive eigenvalue gap (Remark \ref{rem:NoveltyOfProofAndDifficulty}). 
We then choose the function classes 
\begin{align}\label{eq:FunctionClassHC}
	\HC_{\XC}&\coloneqq \{f\colon \XC\rightarrow \RR \colon |f|\leq \cx\}, \quad \quad 
	\HC_{\YC}\coloneqq \{f\colon \YC\rightarrow\RR \colon |f|\leq \cy\eySup{4}\,\}
\end{align}
which induce on $\PC(\XC)$ and $\PC(\YC)$ weighted $\ell^1$-norms with weights given by $\cx$ and $\cy\eySup{4}$ (see Appendix \ref{app:BanachSpacesNorms}). 
The sensitivity analysis relies on an 
optimality criterion for primal and dual optimizers of \eqref{eq:EntropicOptimalTransport} and is based on an implicit function approach. 
A notable challenge in our analysis for the plan is caused by an operator that is only invertible 
on a rather small domain. In particular, this avoids an ad-hoc application of a standard implicit function theorem for Hadamard differentiable functionals \cite[Proposition 4]{Roemisch04}. Instead, we carefully assess the individual error terms that are caused by the perturbation and exploit our bounds for primal and dual optimizers (Remark \ref{rem:NoveltyOfProofAndDifficulty}). 

Invoking the functional delta method we then require for our distributional limits to hold that $\HC_\XC$ is $\rb$-Donsker while $\sb$ defines a bounded functional $\HC_\YC$. These assumptions are satisfied if and only if (Lemma \ref{lem:WeakConvergence})
\begin{align*}
	\sum_{x\in \XC} \cx(x)\sqrt{r_x}<\infty \quad \text{and} \quad \sum_{y\in \YC} \cy(y)\eySup{4}(y) s_y <\infty. 
\end{align*}
An overview of the resulting assumptions to guarantee our distributional limits for the empirical EROT plan and Sinkhorn cost is detailed in Table \ref{tab:SummaryResultsAndConstrains}.

\begin{table}[t]
	\centering
	\caption{Interplay between cost functions and summability conditions on probability measures $\rb$, $\sb$ for one-sample limit distribution of empirical EROT plan and Sinkhorn~cost.}
	\label{tab:SummaryResultsAndConstrains}
	\begin{tabular}{cll}
	\hline	
	Type & Cost                                     &  Summability Condition                                                            \\ \hline  
	Bounded costs                  & $\norm{c}_{\lInfXY}< \infty$                              &  $\sum_{x \in \XC}\sqrt{r_x}< \infty$                                              \\[0.4cm]
	\multirow{2}{*}{\begin{tabular}{@{}c@{}}\\[-0.2cm]Costs with \\bounded variation\end{tabular}}   & $\norm{\cxpb - \cxmb}_{\lInfX}< \infty$                  &  $\sum_{x \in \XF}\cx(x)\sqrt{r_x} < \infty$                                       \\
			 & $\norm{\cypb - \cymb}_{\lInfY}< \infty$                  &  $\sum_{y \in \YC}\cy(y)s_y < \infty$ \;\;\;                                             \\[0.4cm]
	\multirow{2}{*}{\begin{tabular}{@{}c@{}}\\[-0.2cm]Costs with  bounded\\ variation in $\XC$-component\end{tabular}}  & $\norm{\cxpb - \cxmb}_{\lInfX}< \infty$ &  $\sum_{x \in \XF}\cx(x)\sqrt{r_x} < \infty$  \,\,\,\;\;\;                                     \\
							&                                                        $\norm{\cypb - \cymb}_{\lInfY}= \infty$  &  $\sum_{y \in \YC}\cy(y)\eySup{4}(y)s_y<\infty$\\[0.2cm]
	\hline
	\end{tabular}
	\end{table}

The paper is structured as follows. Section~\ref{sec:Preliminaries} introduces basic notation, shows existence and uniqueness of optimizers, and states an optimality criterion for primal and dual optimizers of \eqref{eq:EntropicOptimalTransport}. Furthermore, we derive explicit bounds for optimal solutions. %
Section~\ref{sec:LimitDistributions} presents the main results on limit distributions for the empirical EROT value and plan (Sections \ref{subsec:EROTValue} and \ref{subsec:EROTPlan}). Furthermore, we verify that the na\" ive $n$-out-of-$n$ bootstrap is consistent for the empirical EROT value and plan (Section \ref{sec:Bootstrap}) and analyze the behavior of the limit distributions as the regularization parameter tends to zero (Section \ref{sec:RelationUnregularizedOT}). 
Section \ref{sec:Examples} provides various examples where our theory provides novel distributional limits; we consider uniformly bounded costs, costs induced by powers of a metric, and costs on Euclidean spaces. 
 Section~\ref{sec:SensitivityAnalysis} is concerned with the sensitivity analysis of the EROT value and plan with respect to suitable norms.  
Section \ref{sec:Discussion} discusses our results and concludes with open questions for future research.
 For the sake of readability, all the proofs, except for our main results in Section~\ref{sec:LimitDistributions}, are deferred to the appendix. 

We finally stress that our results have immediate statistical applications, which will be detailed in subsequent work. We briefly mention here biological colocalization of protein networks recorded with super-resolution microscopy which can be defined as a certain functional of the regularized OT plan. This has been done by \cite{klatt2018empirical} in the context of finite ground spaces, i.e., for the setting of a finite number of pixels (corresponding to the support of the intensity profile of protein structures). Our conditions in Table \ref{tab:SummaryResultsAndConstrains} can be viewed as a guarantee for robustness of sub-sampling methods in large scale images with many pixels. Moreover, our results imply in the context of high-resolution imaging that sub-sampling is not only computationally feasible but also a statistically viable approach.

\paragraph*{Notation}

Throughout this work we denote by $\lX$ the space of summable sequences indexed over $\XC$ and equipped with total variation norm $\norm{\ab}_{\lX}= \sum_{x \in \XC}|a_x|$. Note that $\lX$ can be interpreted as the space of finite signed measures on $\XC$.  
Its dual space can be identified by $\lInfX$ with norm $\norm{\bb}_\lInfX=\sup_{x \in \XC}|b_x|$. Given a positive function $f\colon \XC \rightarrow (0, \infty)$ we introduce the weighted $\ell^1$-space $\lXSub{f}$ and its corresponding dual space $\lInfXSub{f}$ with respective norms for $\ab \in \lXSub{f}$ and $\bb\in \lInfXSub{f}$ given by 
$$\norm{\ab}_{\lXSub{f}}\coloneqq \sum_{x \in \XC}f(x)|a_x| \quad \text{ and } \quad \norm{\bb}_\lInfXSub{f} \coloneqq \sup_{x \in \XC}f(x)^{-1}|b_x|.$$ Their dual pairing is given by $\langle \bb,\ab \rangle \coloneqq \sum_{x \in \XC} a_x$. The  set of probability measures on $\XC$ is denoted by $\mathcal{P}(\XC)=\{ \rb \in \lX \colon \sum_{x\in \XC} r_x = 1, \rb \geq 0\}$. We emphasize that we equip $\XC$ with the discrete topology and do not embed it, e.g., in $\R^d$. Hence, for any $\rb\in \mathcal{P}(\XC)$ its (topological) support is equal to $\supp(\rb) = \{ x \in \XC \colon r_x >0\}$ and  $\rb$ is of full support if and only if $r_x>0$ for all $x\in \XC$. For a probability measure $\rb\in\mathcal{P}(\XC)$ we define with slight abuse of notation $\lr$ as the space of functions on $\XC$ with finite expectation with respect to~$\rb$.  Moreover, for a pointwise bounded function class $\FC$ on $\XC$ we define $\ell^\infty(\FC)$ as the Banach space of uniformly bounded functionals on $\FC$ equipped with norm $$\norm{\Phi}_{\ell^\infty(\FC)} \coloneqq \sup_{f\in \FC}|\Phi(f)|.$$
 For an envelope\footnote{The function $F\colon \XC\rightarrow [0,\infty)$ is an envelope function for a function class $\FC$ on $\XC$ if $|f|\leq F$ for any $f\in \FC$.}  function $F$ of the class $\FC$ that fulfills $F\in \lr$ for $\rb\in \PC(\XC)$, it follows that $\rb$ defines a uniformly bounded functional on $\FC$ via  $\rb(f) \coloneqq \langle f, \rb\rangle$. Hence, elements of $\PC(\XC)$ that fulfill suitable integrability conditions can be considered as element of $\ell^\infty(\FC)$. 
For weighted $\ell^1$- and $\ell^\infty$-spaces (of function classes) on $\YC$ and $\XC\times \YC$ we use the same notation.

\section{Preliminaries on Entropic Optimal Transport}
\label{sec:Preliminaries}
\label{subsec:PrimalAndDualOptimizers}

In this section, we state some basic results for EROT. Most of which are well known for Euclidean spaces and the purpose here is to generalize them to a broader class of cost functions. %

Without loss of generality, we assume that the cost function $c$ is non-negative. Otherwise, according to \eqref{eq:LowerUpperBoundsC} define the non-negative function $$\tilde c \colon \XC\times \YC \rightarrow [0, \infty), \quad \tilde c(x,y) \coloneqq c(x,y) - \cxm\oplus\cym(x,y)$$ with $\cxm\oplus\cym(x,y) \coloneqq \cxm(x) + \cym(y)$ which is bounded from above by $\tilde c(x,y) \leq (\cxp - \cxm)(x) + (\cyp - \cym)(y)$ for all $(x,y) \in \XC\times\YC$. Further, if $\cxmb \in \lr$ and $\cymb\in \ls$, any transport plan $\pi\in \Pi(\rb, \sb)$ satisfies $ \langle \cb, \pi\rangle =\langle \tilde \cb, \pi\rangle + \langle \cxmb, \rb\rangle + \langle\cymb, \sb\rangle$.
Hence, the objective functions of \eqref{eq:EntropicOptimalTransport} with cost function $\cb$ and with the shifted non-negative cost function $\tilde \cb$ only differ by a constant that only depends on $\rb, \sb$. This does not affect the set of minimizers, and it holds that 
$$\EROT_c(\rb, \sb) = \EROT_{\tilde c}(\rb, \sb) + \langle \cxmb, \rb\rangle + \langle\cymb, \sb\rangle.$$

We now proceed with a result for existence of optimizers for \eqref{eq:EntropicOptimalTransport} and  formulate a necessary and sufficient optimality criterion. Recall that the primal problem \eqref{eq:EntropicOptimalTransport} and its dual \eqref{eq:DualEntropicOptimalTransportProblem} are said to satisfy strong duality if their respective optimal values match.

\begin{proposition}[Structure and Duality of EROT]\label{prop:DualEROT}   
	Assume that the cost function $c\colon \XC\times \YC\rightarrow [0,\infty)$ fulfills $c\in \ell_{\rb\otimes \sb}^1(\XC\times \YC)$. Then, 
	strong duality holds. Further, for the primal problem  \eqref{eq:EntropicOptimalTransport} there is a unique optimizer $\pib^\lambda \in \Pi(\rb, \sb)$, the EROT plan,  and for \eqref{eq:DualEntropicOptimalTransportProblem} there is an $(\rb, \sb)$-a.s.\ unique pair (up to constant shift\footnote{This means for any pair of EROT potentials $(\tilde \alphab^\lambda, \tilde \betab^\lambda)\in \lr\times \ls$ there is a constant $\eta\in \R$ with $ \alpha_x^\lambda - \tilde\alpha_x^\lambda = \tilde \beta_y^\lambda - \beta_y^\lambda = \eta$ for all $x \in \supp(\rb), y \in \supp(\sb)$.}) of optimizers $(\alphab^\lambda, \betab^\lambda)\in \lr \times \ls$, the EROT potentials. 
Moreover, 
	  $\pib\in \lXY$ and $(\alpha, \betab)\in \lr\times \ls$ are optimal for \eqref{eq:EntropicOptimalTransport} and \eqref{eq:DualEntropicOptimalTransportProblem}, respectively, if and only if 	\begin{align}
		& \;\pi_{xy} = \exp\bigg(\frac{\alpha_x + \beta_y - c(x,y)}{\lambda}\bigg)r_xs_y \quad \text{ for all } (x,y) \in \XF\times\YC, \label{eq:optimalityCriterion_1_connectionSolutions} \\ 
		& \sum_{y\in \YC}\pi_{xy} = r_x \quad \text{ and } \quad\sum_{x\in \XF}\pi_{xy} = s_y \quad \text{ for all }x \in \XF \text{ and } y \in \YC\label{eq:optimalityCriterion_2_marginalConstraints}. 
	\end{align}
\end{proposition}
The assertion on strong duality in conjunction with the existence and uniqueness of  optimizers for \eqref{eq:EntropicOptimalTransport} follows by \citet[Theorem 3.2]{chizat16}. The existence and uniqueness (up to a constant shift) of optimizers for \eqref{eq:DualEntropicOptimalTransportProblem}, as well as the optimality criterion is verified by \citet[Theorems 4.2 and 4.7]{nutz2021introduction}. %

\begin{remark}[Relation between EROT potentials and extension]\label{rem:NonConstantDualPotentials}
	Combining conditions \eqref{eq:optimalityCriterion_1_connectionSolutions} and \eqref{eq:optimalityCriterion_2_marginalConstraints} from Proposition \ref{prop:DualEROT} it follows that $(\alphab^\lambda, \betab^\lambda)\in \lr\times \ls$ denotes a pair of EROT potentials if and only if 
	\begin{equation}
		\label{eq:OptimalityCriterionForDualOptimizers}
		\begin{aligned}
			\alpha_x^\lambda &= - \lambda\log\left[\,\sum_{y \in \YC}\exp\bigg(\frac{ \beta_y^\lambda - c(x,y)}{\lambda}\bigg)s_y\right] \quad \text{for all } x \in \supp(\rb),\\
			\beta_y^\lambda &= - \lambda\log\left[\,\sum_{x \in \XF}\exp\bigg(\frac{ \alpha_x^\lambda - c(x,y)}{\lambda}\bigg)r_x\right]  \quad \text{for all } y \in \supp(\sb).
		\end{aligned}
	\end{equation}
			
		Notably, outside the support of $\rb$ and $\sb$, the EROT potentials can be selected arbitrarily. A natural extension outside the support of the underlying measure is to define them according to the right-hand side. %
\end{remark}

As a consequence of \eqref{prop:DualEROT} it follows that the respective EROT value is equal to
\begin{align}\label{eq:EROTValueRepresentation}
	\EROT(\rb, \sb) = \langle \alphab^\lambda, \rb\rangle  + \langle \betab^\lambda, \sb\rangle + \lambda \langle \Indicator{\XC\times \YC}, \pi^\lambda - \rb\otimes \sb\rangle  = \langle \alphab^\lambda, \rb\rangle  + \langle \betab^\lambda, \sb\rangle.
\end{align}
To assess  the sensitivity of the EROT value and the plan suitable bounds for the entropic optimal potentials are of interest.   
For squared Euclidean costs such bounds are readily available by \cite{mena2019}. In the following we generalize their argument to more general costs which fulfill the bounds \eqref{eq:LowerUpperBoundsC}.%
	\begin{proposition}[Bounds for EROT quantities]\label{prop:BoundsOptimalPotentials}
		Assume that $c\colon \XC\times \YC\rightarrow \RR$ satisfies \eqref{eq:LowerUpperBoundsC} for dominating functions $\cxm, \cxp, \cym, \cyp$ and define $\cx,\cy,\exSup{}, \eySup{}$ as in \eqref{eq:DefWeightingFunctions}.
	Let $\rb \in \PC(\XC)$ and $\sb \in \PC(\YC)$ be probability measures with  
	$\cx, \exSup{}\in \lr $ and $\cy, \eySup{}\in \ls$. Then there exists a unique pair of EROT potentials
		$(\alphab^\lambda, \betab^\lambda)\in \lr\times \ls$ that fulfill $$\langle\alphab^\lambda,\rb\rangle = \langle\betab^\lambda,\sb\rangle = \EROT(\rb, \sb)/2\geq (\langle\cxmb,\rb \rangle + \langle\cymb,\sb \rangle)/2$$ and satisfy for all $x\in \XC, y\in \YC$ the bounds
		 \begin{align*}
			&\cxm(x)- \langle \cxpb, \rb\rangle  + (\langle\cxmb,\rb \rangle + \langle\cymb,\sb \rangle)/2 - \lambda \log\langle \eySup{}, \sb\rangle \\
			\leq  \alpha^{\lambda}_x \leq &  \cxp(x) +\langle \cypb, \sb\rangle -  (\langle\cxmb,\rb \rangle + \langle\cymb,\sb \rangle)/2 , \\[0.2cm]
			&\cym(y)- \langle \cypb, \sb\rangle +(\langle\cxmb,\rb \rangle + \langle\cymb,\sb \rangle)/2 - \lambda \log\langle \exSup{}, \rb\rangle\\
			\leq \beta^{\lambda}_y \leq &  \cyp(y) +\langle \cxpb, \rb\rangle - (\langle\cxmb,\rb \rangle + \langle\cymb,\sb \rangle)/2.
		\end{align*} 
		For the EROT plan $\pib^\lambda$ in \eqref{eq:DefinitionEntropicOptimalTransportPlan} it holds for all $(x,y) \in \XC\times \YC$ that $$\begin{aligned}
			 &r_xs_{y}\Big(\exSup{}(x) \eySup{}(y)\langle \exSup{}, \rb\rangle^{2} \langle \eySup{}, \sb\rangle^{2}\Big)^{\!-1}\!\!\!\leq \pi^\lambda_{xy}
			 \leq r_xs_{y}\Big(\exSup{}(x)  \eySup{}(y)\,\langle \exSup{}, \rb\rangle \langle \eySup{}, \sb\rangle\Big).
		\end{aligned}$$
		In particular, if $\cx\exSup{}\in \lr$ and $\cy\eySup{}\in \ls$, it holds that $\pi^\lambda \in \ell^1_{\cx\oplus\cy}(\XC\times \YC)$. 
	\end{proposition}
	The proof is deferred to Appendix \ref{app:ProofsPreliminaries} and relies on the correspondence between EROT potentials from Remark \ref{rem:NonConstantDualPotentials} in conjunction with Jensen's inequality.

\section{Main Results}\label{sec:LimitDistributions}
We derive the limit distributions of the empirical EROT value and plan. 
More precisely, we estimate the EROT quantities $\EROT(\rb, \sb)$ and $\pib^\lambda(\rb, \sb)$ by plug-in estimators $\EROT(\hat\rb_n, \hat\sb_m)$ and $\pib^\lambda(\hat\rb_n, \hat\sb_m)$  based on empirical counterparts $\hat\rb_n$ and $\hat \sb_m$ \eqref{eq:DefEmpiricalMeasures} of the probability measures $\rb$ and $\sb$. %
We then characterize the statistical fluctuation of these empirical plug-in estimators around their respective population version in terms of weak convergence. %
The underlying metric space in which weak convergence of the empirical EROT plan occurs is a suitable weighted $\ell^1$-space. This is of particular interest in order to conclude uniform distributional limits for the empirical EROT plan evaluated over a suitably bounded function class.
Our results are based on an application of the functional delta method for  tangentially Hadamard differentiable functionals \citep[Chapter 3.6]{van1996weak}. More precisely, for our sensitivity analysis (Section \ref{sec:SensitivityAnalysis})  of the EROT value and plan we consider a suitable norm $\norm{\cdot}_{\ell^\infty(\FC)}$ induced by a function class $\FC$ which is motivated by our bounds for the EROT potentials (Proposition \ref{prop:BoundsOptimalPotentials}) and depends on the choices of the dominating functions from \eqref{eq:LowerUpperBoundsC} for the cost $c$. To employ the functional delta method it is necessary that the empirical process $\sqrt{n}(\hat \rb_n - \rb)$ indexed over $\FC$ weakly converges to a tight limiting process. This is formalized using the notion of Donsker classes \citep[Chapter~2]{van1996weak}. In particular, based on the finite dimensional central limit theorem it follows that the weak limit for the empirical process is given by a Gaussian process $G_\rb$ with covariance, 
	\begin{equation}\label{eq:DefintionCovariance}
		\Cov[G_\rb(f), G_\rb(f')] = \langle ff', \rb\rangle - \langle f, \rb\rangle\langle f', \rb\rangle \quad \text{ for } f,f'\in \FC.
	\end{equation}
Necessary and sufficient conditions for important function classes of this work to fulfill the Donsker property  are given in the following lemma. Its proof is given in  Appendix \ref{app:ProofSectionExamples}. 

\begin{lemma}\label{lem:WeakConvergence}
	Let  $\rb\in \PC(\XC)$  be a probability measure and consider a  positive function $f \colon \XC\rightarrow (0, \infty)$. Define the function classes $\FC_1 \coloneqq\{ \tilde f \colon \XC\rightarrow \RR \,\colon |\tilde f|\leq f\}$ and  $\FC_2\coloneqq \{f\}\cup\{f \Indicator{\XC\backslash\{x_1, \dots, x_n\}} \colon n \in \NN\}$ and let $\FC_3,\FC_4$ be two arbitrary function classes on $\XC$. Then, the following assertions hold. 
	\begin{enumerate}[label=$(\roman*)$]
		\item \;\;The function class $\FC_1 $ is $\rb$-Donsker if and only if $\sum_{x\in \XC}f(x)\sqrt{r_x}<\infty$.
		\item \;The function class $\FC_2$ is $\rb$-Donsker if and only if $\sum_{x\in \XC} (f(x))^2\,r_x<\infty$.
		\item The union $\FC_3\cup \FC_4$ is $\rb$ Donsker if and only if $\FC_3$ and $\FC_4$ are $\rb$-Donsker.
	\end{enumerate}

\end{lemma}

We like to stress that assertion $(i)$ is a weighted version of the celebrated Borisov-Dudley-Durst theorem \citep{durst1980, borisov1981some,  Borisov1983}. 
The summability condition in $(ii)$ is weaker than the weighted Borisov-Dudley-Durst condition in $(i)$. 
Indeed, if $\sum_{x\in \XC}f(x)\sqrt{r_x}<\infty$, then the subset $\tilde \XC\subseteq \XC$ of elements $x\in \XC$ for which $f(x)\sqrt{r_x}\geq 1$ admits finite cardinality; 
hence, since $t \leq \sqrt{t}$ for any $t\in [0,1]$, we obtain
\begin{align*}
	\sum_{x\in \XC} (f(x))^2\,r_x %
	\leq \sum_{x\in \tilde \XC} (f(x))^2\,r_x + \sum_{x\in\XC\backslash \tilde \XC}f(x)\sqrt{r_x}<\infty. 
\end{align*}
Moreover, since the Banach space $\lXSub{f}$ isometrically embeds into the space $\ell^\infty(\FC)$ (Lemma \ref{lem:equivalenceNorms}) it follows by \citet[Theorem 1.3.10]{van1996weak} that the condition $\sum_{x\in \XC}f(x)\sqrt{r_x}<\infty$ is necessary and sufficient for the empirical process $\sqrt{n}(\hat \rb_n - \rb)$ to weakly converge in $\lXSub{f}$.

\subsection{Distributional Limits for Empirical Entropic Optimal Transport  Value}\label{subsec:EROTValue}

In this subsection we state our main result on the distributional limits for the empirical EROT value under general cost functions.

\begin{theorem}[Distributional limit for EROT value]
	Assume that $c\colon \XC\times \YC\rightarrow \RR$  satisfies \eqref{eq:LowerUpperBoundsC} for two collections of dominating functions $\cxm, \cxp, \cym, \cyp$ and $\tcxm, \tcxp,\tcym,\tcyp$. For these two collections define $\cx,\cy,\eySup{}$ and $\tcx,\tcy,\texSup{}$ as in \eqref{eq:DefWeightingFunctions}. 
	Further, consider probability measures $\rb\in \PC(\XC)$, $\sb\in \PC(\YC)$. 	

\label{them:LimitLawEmpEROT_Value}
	\begin{enumerate}
		\item[$(i)$] (One sample from $\rb$) Denote by $\hat \rb_n \coloneqq \frac{1}{n}\sum_{i = 1}^n \delta_{X_i}$ the empirical measure for random variables $X_1, \dots, X_n \iid \rb$ and assume that  
		\begin{align*}
			&\sum_{x\in \XC} \cx(x)\sqrt{r_x} + \big(\tcx^2+\texSup{2}\big)(x)r_x <\infty \quad \text{and} \quad \sum_{y\in \YC} \big(\cy+\tcy+\eySup{}\big)(y) s_y <\infty.
		\end{align*} 
		Then, for $n\rightarrow \infty$, it follows that$$ \sqrt{n}\Big(\EROT(\hat \rb_n, \sb) - \EROT(\rb, \sb)\Big) \konvW \normal\big(0,\sigma_{\lambda}^2(\rb|\sb)\big),$$
		where $\sigma_\lambda^2(\rb|\sb)=
		\textup{Var}_{X\sim \rb}[\alpha^\lambda_X]$ for $(\alphab^\lambda, \betab^\lambda)\in \lr\times \ls$ denoting EROT potentials for $\rb$ and $\sb$ with cost $c$.
		\item[$(ii)$] (Two samples) Additionally, let $\hat \sb_m \coloneqq \frac{1}{m}\sum_{i = 1}^m \delta_{Y_i}$ be the empirical measure for random variables $Y_1, \dots, Y_m \iid \sb$ independent of $X_1, \dots, X_n \iid \rb$ and assume that 
		\begin{align*}
			&\sum_{x\in \XC} \cx(x)\sqrt{r_x}  + \big(\tcx^2+\texSup{2}\big)(x)r_x<\infty \!\!\quad \text{and} \quad\!\! \sum_{y\in \YC} \tcy(y)\sqrt{s_y} +\big(\cy^2+\eySup{2}\big)(y) s_y <\infty.
		\end{align*}
		Then, for $\min(m,n)\rightarrow \infty$ with $\frac{m}{m+n}\rightarrow \delta \in (0,1)$, it follows that \begin{align*}
			\sqrt{\frac{nm}{n+m}}\Big(\EROT(\hat \rb_n, \hat\sb_{m}) - \EROT(\rb, \sb)\Big) \konvW  \normal\big(0, \delta\sigma_{\lambda}^2(\rb|\sb) + (1-\delta)\sigma_{\lambda}^2(\sb|\rb)\big),
		\end{align*}%
		where $\sigma_{\lambda}^2(\rb|\sb) = \textup{Var}_{X\sim \rb}[\alpha^\lambda_X]$ and $\sigma_{\lambda}^2(\sb|\rb)=\textup{Var}_{Y\sim \sb}[\beta^\lambda_Y]$ for $(\alphab^\lambda,\beta^\lambda)\in \lr\times \ls$ denoting a pair of EROT potentials for $\rb$ and $\sb$ with cost~$c$.
	\end{enumerate}
\end{theorem}

\begin{proof}
	 Consider the function classes $\FC_\XC$ and $\FC_\YC$ from \eqref{eq:FunctionClassFC}. 
By the summability constraint for the one-sample and two-sample case we obtain from Lemma \ref{lem:WeakConvergence} weak convergence
\begin{align*}
	\sqrt{n} \Big((\hat \rb_n,s) - (\rb, s)\Big) &\konvW (\Gb_{\rb},0)\; &&\text{ in } \ell^\infty(\FC_\XC)\times \ell^\infty(\FC_\YC),\\
	\sqrt{\frac{nm}{n+m}} \Big((\hat \rb_n,\hat s_m) - (\rb, s)\Big) &\konvW (\sqrt{\delta}\Gb_{\rb},\sqrt{1-\delta}\Gb_\sb)\; &&\text{ in } \ell^\infty(\FC_\XC)\times \ell^\infty(\FC_\YC).
\end{align*}
The covariance of the Gaussian process is detailed in \eqref{eq:DefintionCovariance}. 
Further, by Theorem \ref{them:TransportPlanIsHadamardDiff} the  EROT value is Hadamard differentiable at $(\rb, \sb)$  tangentially to $\big(\probset{\XC}\cap \ell^\infty(\FC_\XC)\big)\times \big( \probset{\YC}\cap \ell^\infty(\FC_\YC)\big)$. The Hadamard derivative is defined on the contingent cone and given by $(h^\XC, h^\YC) \mapsto \langle \alpha^\lambda, h^\XC\rangle  + \langle \betab^\lambda, h^\YC\rangle$. In particular, by convexity of the domain it follows by \citet[Proposition 4.2.1]{aubin2009setValued} that the contingent cone at $(\rb, \sb)$ is characterized in terms of, 
\begin{align*}
	&T_{|(\rb, \sb)}\Big(\big(\probset{\XC}\cap \ell^\infty(\FC_\XC)\big)\times \big( \probset{\YC}\cap \ell^\infty(\FC_\YC)\big)\Big) \\
	=& \closure\{t(\tilde \rb - \rb, \tilde \sb- \sb) \colon t>0, \tilde \rb \in \PC(\XC)\cap \ell^\infty(\FC_\XC), \tilde \sb\in \PC(\YC)\cap \ell^\infty(\FC_\YC)\},
\end{align*}
where the closure is taken with respect to $\ell^\infty(\FC_\XC)\times \ell^\infty(\FC_\YC)$.  
Hence,  by Portmanteau's characterization of weak convergence in terms of closed sets \citep[Theorem 1.3.4(ii)]{van1996weak} it follows that the weak limit is almost surely contained in the contingent cone at $(\rb, \sb)$. 
Invoking the functional delta method \cite[Theorem 3.9.4]{van1996weak} thus asserts both distributional limits for the empirical EROT value. 
Note herein that the underlying topology for Hadamard differentiability and for the weak convergence coincide. 
By linearity of the tangential Hadamard derivative, the asymptotic distribution is given by centered normal where the limiting variance for the one-sample case is given by $$\sigma^2_{\lambda}(\rb|\sb) =  \sum_{x,x'\in \XC}\alphab^\lambda_x\alphab^\lambda_{x'} \Sigma(\rb)_{xx'} = \sum_{x \in \XC}(\alphab^\lambda_x)^2 r_x - \Big(\sum_{x\in \XC}\alphab^\lambda_x r_x\Big)^2 = \text{Var}_{X\sim \rb}[\alpha^\lambda_X].$$
The calculation for the two-sample case is analogous.
\end{proof}

\begin{remark}[Comparison to other works and degeneracy]\label{rem:NotesOnLimitLaws}
	A few comments concerning the obtained distributional limits and their relation to related results are in order. 
	\begin{enumerate}
		\item[$(i)$]  Implications to popular costs, e.g., $p$-th power metric based costs are derived Section \ref{sec:Examples}. We also assess uniformly bounded costs and powers of norms on Euclidean spaces.  %
		\item[$(ii)$] 
		The distributional limits for the EROT value are characterized by a centered normal where the asymptotic variance is given by the variance of the respective population EROT potentials. These observations are in line with previous findings for finite spaces \citep{bigot2019CentralLT,klatt2018empirical} and for continuous spaces with squared Euclidean costs \citep{mena2019,del2022improved, goldfeld2022statistical}.  In fact, for probability measures with finite support our required summability constraints are trivially fulfilled, and our findings coincide with results by \cite{bigot2019CentralLT} and  \cite{klatt2018empirical}. 
		\item[$(iii)$] For squared Euclidean costs and sufficiently spread countable spaces $\XC, \YC$ contained in a Euclidean spaces (e.g. $\XC = \YC= \Z$), our conditions require that the underlying measures are sub-Weibull of order $1+\epsilon$, instead of sub-Gaussian (i.e., sub-Weibull of order $2$) as required by \cite{mena2019,del2022improved, goldfeld2022statistical}. In this context, we deem it possible that the sub-Gaussianity condition can be slightly lifted for continuous Euclidean settings. 
		\item[$(iv)$] For a common countable  metric ground space $(\XC,d)$ i.e., $\XC = \YC$ our results for the cost function $c(x,y) = d^p(x,y)$, $p\geq 1$ show no substantial difference in the limit law for empirical EROT quantities between the cases $\rb =  \sb$ and $\rb \neq \sb$. 
		Additionally, our derived limit distributions generally do not degenerate as EROT potentials are typically non-constant due to their relation via the optimality conditions (Remark \ref{rem:NonConstantDualPotentials}). %
		In contrast, limit results for the empirical unregularized OT value obtained by \cite{tameling18} show a clear distinction in the limit behavior between the cases $\rb =  \sb$ and $\rb \neq \sb$. Further, for $\XC = \YC$ and $c(x,y) = d^p(x,y)$ with $p>1$ the obtained limits by \cite{tameling18} degenerate for $\rb= \sb$ with $\supp(\rb) = \XC$ if and only if $\XC$ has no isolated point. This is a consequence of uniqueness and constancy of the underlying unregularized OT potentials. For more general costs, sharp conditions for the existence of constant OT potentials in terms of the cost and the underlying measures are derived by \citet[Section 4]{hundrieser2022unifying}. This illustrates again that limit laws for empirical unregularized OT fundamentally differ from their entropic counterparts. 
		\item[$(v)$] Given the distributional limits, it is a natural question to ask whether the empirical EROT value is asymptotically efficient in the semiparametric sense (see  \citealt[p. 367]{vaart_1998}). Intuitively, this means that any estimator for $EROT(\cdot,\cdot)$  attains, among all one-dimensional submodels passing through the population measures $(\mu, \nu),$ at least the same asymptotic variance. \citet[Proposition 2, Theorem 7]{goldfeld2022statistical} have shown that this is indeed the case for the EROT value due to its linear derivative, demonstrating asymptotic efficiency of the empirical EROT value.
	\end{enumerate}

\end{remark}

When considering $\XC =\YC$ and a positive definite cost function, i.e., $c(x,y) \geq 0$ for any $x, y\in \XC$ and $c(x,y) = 0$ if and only if $x= y$, then for any $\rb\in \PC(\XC)$ which is not equal to a Dirac measure it follows that $\EROT(\rb, \rb)>0$. This is a consequence of entropy regularization which leads to $\supp(\pi^\lambda(\rb,\rb))=\supp(\rb\otimes \rb)$ for any $\lambda>0$. 
	To resolve this drawback of the EROT value, \cite{Feydy2019Interpolating} proposed a debiased quantity. For $\rb, \sb\in \PC(\XC)$ the \emph{Sinkhorn divergence} is defined as
	\begin{align*}
		\overline {E\!ROT}^\lambda(\rb, \sb) &\coloneqq \EROT(\rb, \sb) - \frac{1}{2} \left(\EROT(\rb, \rb) - \EROT(\sb, \sb)\right).%
	\end{align*}
	Clearly, it holds that $\overline {E\!ROT}^\lambda(\rb, \rb) =0$. For this  quantity we also obtain distributional limits when estimating the underlying measures using empirical measures. 
	\begin{theorem}[Distributional limit for Sinkhorn divergence]
		Let $\XC= \YC$ and consider a symmetric, positive definite cost function with two collections of dominating functions $\cxm, \cxp, \cym, \cyp$ and $\tcxm, \tcxp,\tcym,\tcyp$ as in \eqref{eq:LowerUpperBoundsC} and $\cx,\cy,\eySup{}$ and $\tcx,\tcy,\texSup{}$ as in \eqref{eq:DefWeightingFunctions}. 
		Further, consider probability measures $\rb, \sb\in \PC(\XC)$. \begin{enumerate}
			\item[$(i)$] (One sample from $\rb$) Denote by $\hat \rb_n \coloneqq \frac{1}{n}\sum_{i = 1}^n \delta_{X_i}$ the empirical measure for random variables $X_1, \dots, X_n \iid \rb$ and assume that  
		\begin{align}
			&\sum_{x\in \XC} (\cx + \cy)(x)\sqrt{r_x} + \big(\tcx^2 + \tcy^2+\texSup{2} +\eySup{2} \big)(x)r_x <\infty \quad \text{and} \quad \label{eq:SummabilityConstraintDebiasedVersion}\\
			&\sum_{y\in \YC} \big(\cx + \cy + \tcx+ \tcy+\texSup{} + \eySup{}\big)(x) s_x <\infty.
		\end{align} 
		Then, for $n\rightarrow \infty$, it follows that
		$$ \sqrt{n}\Big(\overline{E\!ROT}^\lambda(\hat \rb_n, \sb) - \overline{E\!ROT}^\lambda(\rb, \sb)\Big) \konvW \normal\big(0,\overline\sigma_{\lambda}^2(\rb|\sb)\big),$$
		where $\overline\sigma_\lambda^2(\rb|\sb)=
		\textup{Var}_{X\sim \rb}[\alpha_X^\lambda(\rb,\sb) - \alpha_X^\lambda(\rb, \rb)]$ and  $\alphab^\lambda(r,s)$, $\alphab^\lambda(r,r)$ represent EROT potentials.

		\item[$(ii)$] (Two samples) Additionally, let $\hat \sb_m \coloneqq \frac{1}{m}\sum_{i = 1}^m \delta_{Y_i}$ be the empirical measure for random variables $Y_1, \dots, Y_m \iid \sb$ independent of $X_1, \dots, X_n \iid \rb$ and assume that $\rb$ and $\sb$ fulfill \eqref{eq:SummabilityConstraintDebiasedVersion}. 
		Then, for $\min(m,n)\rightarrow \infty$ with $\frac{m}{m+n}\rightarrow \delta \in (0,1)$, it follows that \begin{align*}
			\sqrt{\frac{nm}{n+m}}\Big(\overline{E\!ROT}^\lambda(\hat \rb_n, \hat\sb_{m}) - \overline{E\!ROT}^\lambda(\rb, \sb)\Big) \konvW  \normal\big(0, \delta\overline\sigma_{\lambda}^2(\rb|\sb) + (1-\delta)\overline\sigma_{\lambda}^2(\sb|\rb)\big),
		\end{align*}%
		where $\overline\sigma_\lambda^2(\rb|\sb)=
		\textup{Var}_{X\sim \rb}[\alpha_X^\lambda(\rb,\sb) - \alpha_X^\lambda(\rb, \rb)]$ and $\overline\sigma_\lambda^2(\sb|\rb)=
		\textup{Var}_{Y\sim \rb}[\betab_Y^\lambda(\rb,\sb) - \beta_Y^\lambda(\sb, \sb)]$ 
		and  $\alphab^\lambda(r,s)$, $\alphab^\lambda(r,r)$, $\betab^\lambda(r,s),\betab^\lambda(s,s)$ represent EROT potentials.
		\end{enumerate}
	\end{theorem}

	The proof is analogous to the proof of Theorem \ref{them:LimitLawEmpEROT_Value} and based on the functional delta method in conjunction with the Hadamard differentiability of the Sinkhorn divergence (Corollary \ref{cor:DebiasedEntropicOTCostUnboundedIsHadamardDifferentiable}). 
	Notably, inspecting the variance, it follows under $\rb = \sb$ that the distributional limit for the debiased EROT value always degenerates. To analyze the limit law of the empirical debiased EROT value for this setting a second order sensitivity result is required. For discrete spaces and on compact Euclidean spaces with squared Euclidean costs this has been recently done by \cite{bigot2019CentralLT} and \cite{gonzalez2022weak}, respectively. The corresponding limit law can then be characterized as a bilinear form of normally distributed random variables. %

\subsection{Distributional Limits for Empirical Entropic Optimal Transport Plan}
\label{subsec:EROTPlan}
For our main result on the empirical EROT plan we require additional assumptions on the dominating functions for the underlying cost function. 

\begin{theorem}[Distributional limit for EROT plan]\label{them:LimitLawEmpEROT_Plan}
	Assume that $c\colon \XC\times \YC\rightarrow \RR$ satisfies \eqref{eq:LowerUpperBoundsC} for dominating functions $\cxm, \cxp, \cym, \cyp$ such that $\norm{\cxpb - \cxmb}_{\lInfX}< \infty$ and define $\cx,\cy,\exSup{}, \eySup{}$ as in \eqref{eq:DefWeightingFunctions}. Further, consider probability measures $\rb\in \PC(\XC)$, $\sb\in \PC(\YC)$. 	
	\begin{enumerate}
		\item[$(i)$] (One sample from $\rb$) 
		Denote by $\hat \rb_n \coloneqq \frac{1}{n}\sum_{i = 1}^n \delta_{X_i}$ the empirical measure for random variables $X_1, \dots, X_n \iid \rb$ and assume that 
		\begin{align*}
			\sum_{x\in \XC} \cx(x)\sqrt{r_x} <\infty \quad \text{and}\quad  \sum_{y\in \YC}  \cy(y)\eySup{4}(y) s_y <\infty.
		\end{align*}
		Then, for $n\rightarrow \infty$, it follows that %
		$$ \sqrt{n}\Big(\pib^\lambda(\hat \rb_n, \sb) - \pib^\lambda(\rb, \sb)\Big) \konvW \DH_{|(\rb, \sb)}\pib^\lambda (\Gb_{\rb},0) \equalDistr \Gb(0, \Sigma_{\lambda, \pi^\lambda}(\rb|\sb))$$
		in $\lXYSub{\cx\oplus\cy}$, 
		where $\DH_{|(\rb, \sb)}\pib^\lambda$ denotes the Hadamard derivative of $\pi^\lambda$ at $(\rb, \sb)$  (Theorem \ref{them:TransportPlanIsHadamardDiff}) and $\Gb_{\rb}$ is a tight, centered Gaussian process in $\lXYSub{\cx}$ with covariance $\Sigma(\rb)$ characterized by 
		\begin{equation}\label{eq:DefintionCovariance}
				 \Sigma(\rb)_{xx'} = 
				 	\begin{cases}
					   r_x(1 - r_x) & \text{ if } x = x',\\
					   -r_x r_{x'} & \text{ if } x \neq x'.
				   	\end{cases}
		  \end{equation}
		  \item[$(ii)$] (One sample from $\sb$) Denote by $\hat \sb_m \coloneqq \frac{1}{m}\sum_{i = 1}^m \delta_{Y_i}$ the empirical measure for random variables $Y_1, \dots, Y_m \iid \sb$ and assume that 
		  \begin{align*}
			  \sum_{x\in \XC} \cx(x) r_x <\infty \quad \text{and}\quad  \sum_{y\in \YC} \cy(y) \eySup{4}(y) \sqrt{s_y} <\infty.
		  \end{align*}
		  Then, for $m\rightarrow \infty$, it follows that
		  $$ \sqrt{m}\Big(\pib^\lambda(\rb,\hat \sb_m) - \pib^\lambda(\rb, \sb)\Big) \konvW \DH_{|(\rb, \sb)}\pib^\lambda (0,\Gb_{\sb}) \equalDistr \Gb\left(0, \Sigma_{\lambda, \pi^\lambda}(\sb|\rb)\right)$$
		  in $\lXYSub{\cx\oplus\cy}$, 
		  where $\Gb_{\sb}$ is a tight, centered Gaussian process in $\lYSub{\cy \eySup{4}}$ with covariance $\Sigma(\sb)$ characterized as in \eqref{eq:DefintionCovariance}. 

		  \item[$(iii)$] (Two samples) Additionally, suppose that the empirical measures $\hat \rb_n$ and $\hat \sb_m$ from above are based on independent random variables and assume that 
		  \begin{align*}
			  \sum_{x\in \XC} \cx(x) \sqrt{r_x} <\infty \quad \text{and}\quad  \sum_{y\in \YC} \cy(y) \eySup{4}(y) \sqrt{s_y} <\infty.
		  \end{align*}
		  Then, for $\min(m,n)\rightarrow \infty$ with $\frac{m}{m+n}\rightarrow \delta \in (0,1)$, it follows that 
		  \begin{align*}
			\sqrt{\frac{nm}{n+m}}\Big(\pib^\lambda(\hat \rb_n, \hat\sb_{m}) - \pib^\lambda(\rb, \sb)\Big) &\konvW \DH_{|(\rb, \sb)}\pib^\lambda \left(\sqrt{\delta}\Gb_{\rb},\sqrt{1-\delta}\Gb_{\sb}\right) \\
			& \;\;\equalDistr \Gb\left(0,  \delta\Sigma_{\lambda, \pi^\lambda}(\rb|\sb)+(1-\delta)\Sigma_{\lambda, \pi^\lambda}(\sb|\rb)\right)
		  \end{align*}
		  in $\lXYSub{\cx \oplus \cy}$, where $\Gb_{\rb}$ and $\Gb_{\sb}$ are independent Gaussian processes in  $\lXSub{\cx}$ and  $\lYSub{\cy\eySup{4}}$ as in $(i)$ and $(ii)$, respectively.
	\end{enumerate}
\end{theorem}

\begin{proof}%
For the assertions on the empirical EROT plan we consider the function classes $\HC_\XC$ and $\HC_\YC$ from \eqref{eq:FunctionClassHC} 
which induce according to Lemma \ref{lem:equivalenceNorms} on $\PC(\XC)$ and $\PC(\YC)$ weighted $\ell^1$-norms $\norm{\cdot}_{\lXSub{\cx}}$ and $\norm{\cdot}_{\lYSub{\cy\eySup{4}}}$. Hence, we infer by Lemma \ref{lem:WeakConvergence} using the respective summability constraints for the one-sample case from $\rb$, for the one-sample case from $\sb$, and the two-sample case that 
\begin{align*}
	\sqrt{n} \Big((\hat \rb_n,s) - (\rb, s)\Big) &\konvW (\Gb_{\rb},0)\;&& \text{ in } \lXSub{\cx}\times \lYSub{\cy\eySup{4}},\\
	\sqrt{m} \Big((\rb,\hat \sb_m) - (\rb, s)\Big) &\konvW (0, \Gb_{\sb})\;&& \text{ in } \lXSub{\cx}\times \lYSub{\cy\eySup{4}},\\
	\sqrt{\frac{nm}{n+m}} \Big((\hat \rb_n,\hat \sb_m) - (\rb, s)\Big) &\konvW (\delta\Gb_{\rb},\sqrt{1-\delta}\Gb_\sb)\; &&\text{ in } \lXSub{\cx}\times \lYSub{\cy\eySup{4}}.
\end{align*}
	As the EROT plan is Hadamard differentiable at $(\rb, \sb)$ tangentially to 
	$\big(\probset{\XC}\cap \lXSub{\cx}\big)\times \big( \probset{\YC}\cap \lYSub{\cy\eySup{4}}\big)$ the different claims follow from the functional delta method. Again by linearity of the derivative the weak limit is given by a centered Gaussian process \qedhere

\end{proof}

\begin{remark}[Weak limit for EROT plan]
	If $\rb$ and $\sb$ are not supported on a single point it follows by Remark \ref{rem:NonDegDerivative} that the Hadamard derivative of $\pi^\lambda$ is different from the zero operator, which implies that the weak limit for the empirical EROT plan is non-degenerate. 

	The nature of the weak limits for the empirical EROT plan in the one- and two-sample case is complicated due to the highly non-trivial derivative for the EROT plan. An explicit expression of the Hadamard derivative is provided in Theorem \ref{them:TransportPlanIsHadamardDiff} which is, however, rather complicated to estimate. Already for finite spaces, an explicit representation of the asymptotic covariance for the limit distribution comprises  a major challenge, even if the underlying measures are assumed to be known \citep{klatt2018empirical}. In practice, one would therefore rather resort to bootstrap resampling techniques, whose consistency we show in Section \ref{sec:Bootstrap}.  
\end{remark}

As a corollary from our distributional limit for the EROT plan we obtain from the continuous mapping theorem \citep[Theorem 1.9.5]{van1996new} a (uniform) weak limit for the expectation of a suitably dominated function class with respect to the empirical EROT plan as well as for the empirical Sinkhorn cost. 
We only present the one sample case, the two-sample case is analogous.

 \begin{corollary}\label{cor:LimitLawEmpFunctionWithEROTPlan}
Assume the same setting as for assertion (i) in Theorem \ref{them:LimitLawEmpEROT_Plan} and consider a function class $\FC$ on $\XC\times \YC$ such that $\sup_{f\in \FC}\norm{f}_{\lInfXYSub{\cx\oplus \cy}}<\infty$.  Then, for $n \rightarrow\infty$, it follows that  $$ \sqrt{n}\Big(\langle f, \pi^\lambda(\hat \rb_n, \sb)\rangle - \langle f, \pi^\lambda(\rb, \sb) \rangle \Big)_{f\in \FC} \konvW G(0, \Sigma_{\lambda, \pi^\lambda}^\FC(\rb|\sb)),$$%
where the weak limit is given by a tight, centered Gaussian
with covariance characterized by 
\begin{align*}
	\Sigma_{\lambda, \pi^\lambda}^\FC(\rb|\sb)_{f,f'} = \sum_{\substack{x, x' \in \XC\\y, y' \in \YC}} f(x,y) f'(x',y')\big(\Sigma_{\lambda,\pib^\lambda}(\rb|\sb)\big)_{(x,y),(x',y') }.
\end{align*}
\end{corollary}

\begin{corollary}\label{cor:LimitLawEmpSinkhornCosts}
	Assume the same setting as for assertion $(i)$ in Theorem \ref{them:LimitLawEmpEROT_Plan}. 
	Then, for $n \rightarrow\infty$, it follows that $$ \sqrt{n}\Big(\SOT(\hat \rb_n, \sb) - \SOT(\rb, \sb)\Big) \konvW \normal(0, \tilde \sigma_{\lambda, \pi^\lambda}^2(\rb|s)),$$ %
	where the limiting covariance is given by $$\tilde \sigma_{\lambda, \pi^\lambda}^2(\rb|s) = \sum_{\substack{x, x' \in \XC\\y, y' \in \YC}} c(x,y) c(x',y')\big(\Sigma_{\lambda,\pib^\lambda}(\rb|\sb)\big)_{(x,y),(x',y') }.$$
\end{corollary}

\begin{remark}[Comparison to other works]\label{rem:ComparisonEROTplan}\begin{enumerate}
	\item[$(i)$]  If $\XC$ and $\YC$ are finite spaces, then the conditions for our distributional limits are fulfilled for \emph{any} pair of probability measures $\rb\in \PC(\XC),\sb\in\PC(\YC)$. Our limits then match with results by \cite{klatt2018empirical} when considering the mutual information as the penalization term. 
	\item[$(ii)$] 
	Similar to the weak limit for the empirical EROT plan, the  asymptotic variance for the evaluation of functionals with respect to the empirical EROT plan is difficult to express explicitly. Still, these limits are consistent with the results obtained by \cite{Harchaoui2020} and \cite{gonzalez2022weak}. We verify this observation in Appendix \ref{subsec:ConsistencyOtherContributions}. In particular, Corollary~\ref{cor:LimitLawEmpSinkhornCosts} proves a conjecture by \cite{Harchaoui2020} on the asymptotic normality of the empirical Sinkhorn costs for a general class of cost function in the setting of countable ground spaces.
\end{enumerate}

\end{remark}

\begin{remark}[Sharpness of conditions]\label{rem:SharpnessOfBothLimitLaws}
For cost functions with bounded variation, i.e., if  $\norm{\cxpb - \cxmb}_{\lInfX}+\norm{\cypb - \cymb}_{\lInfY}<\infty$, and in particular bounded costs $\norm{c}_{\lInfXY}<\infty$, our stated assumptions for the validity of our limit laws for the empirical EROT plan are \emph{sharp}. This is a simple consequence of the fact that weak convergence of the empirical EROT plan $\sqrt{n}(\pib^\lambda(\hat \rb_n, \hat \sb_n) - \pib^\lambda(\rb, \sb))$ in $\lXYSub{\cx\oplus \cy}$ requires that the marginal empirical processes $\sqrt{n}(\hat \rb_n - \rb)$ and $\sqrt{n}(\hat \sb_n - \sb)$ converge weakly in $\lXSub{\cx}$ and $\lYSub{\cy}$, respectively. 
Invoking Lemma \ref{lem:WeakConvergence}(i) and \ref{lem:equivalenceNorms} such weak convergence occurs if and only if weighted Borisov-Dudley-Durst conditions $\sum_{x \in \XC}\cx(x) \sqrt{r_x}+\sum_{y \in \YC}\cy(y) \sqrt{s_y}<\infty$ are fulfilled.
\end{remark}

\subsection{Bootstrap Consistency}\label{sec:Bootstrap}
Our findings from Theorems \ref{them:LimitLawEmpEROT_Value} and
\ref{them:LimitLawEmpEROT_Plan} on the distributions of the empirical EROT value and plan  are asymptotic results. In order to estimate the respective non-asymptotic distribution typically bootstrap methods are applied.
On finite and countable spaces \cite{sommerfeld2018} and \cite{tameling18} showed that the unregularized OT value is only \emph{directionally} Hadamard differentiable, i.e., that the Hadamard derivative with respect to $\rb$ and $\sb$ is \emph{non-linear}. As a consequence, the (na\"ive) $n$-out-of-$n$ bootstrap for the approximation of the distribution of the empirical unregularized OT value fails. However, the EROT value and plan on countable spaces are Hadamard differentiable with \emph{linear} derivatives. Therefore, it follows that the na\" ive $n$-out-of-$n$ bootstrap appears to be a consistent estimation method.  

To make this statement precise we follow \cite{van1996weak}. Denote the notion of convergence in outer probability by $\konvP$ and consider for a given Banach space $B$ the set of bounded Lipschitz functions with Lipschitz modulus at most one, $$\BL{B} \coloneqq \left\{ g \colon B \rightarrow \R\; \colon\; \sup_{x \in B}|g(x)| \leq 1,\; |g(x_1) - g(x_2)| \leq \norm{x_1 - x_2}_{B} \forall x_1, x_2 \in B\right\}.$$ 

With this notation we show the consistency of the bootstrap for the EROT value and plan as an application of the functional delta method in conjunction with consistency of the bootstrap empirical process. The proof is an application of the functional delta method for the bootstrap \citet[Theorem 3.6.13]{van1996weak}, details are deferred to Appendix \ref{app:ProofSectionExamples}.

\begin{theorem}[Bootstrap consistency for EROT quantities]\label{them:BootstrapEROT}
Consider an empirical measure $\hat\rb_n$ derived by a sample $X_1, \dots, X_n \iid \rb$ and denote by $\hat \rb_n^* = \frac{1}{n}\sum_{i = 1}^{n}\delta_{X_i^*}$ the empirical bootstrap estimator for $\hat \rb_n$ based on a sample $X_1^*, \dots, X_n^* \iid \hat \rb_n$.
 Under the same setting as in assertion $(i)$ of Theorem \ref{them:LimitLawEmpEROT_Value} the (na\" ive) bootstrap is consistent for the EROT value, i.e., as $n$ tends to infinity it holds that
 $$\begin{aligned}
 	\sup_{h \in \BL{\R}}\Big|\EV{h\big(\sqrt{n}(\EROT(\hat \rb_n^*, \sb) - \EROT(\hat \rb_n, \sb))\big) \Big| X_1, \dots, X_n } \\
 	-  \EV{h\big(\sqrt{n}(\EROT(\hat \rb_n, \sb) - \EROT(\rb, \sb))\big)}\Big|\konvP 0.
 \end{aligned}	$$ 
 Moreover, under the same setting as in assertion $(i)$ of Theorem \ref{them:LimitLawEmpEROT_Plan} the EROT plan is also consistent for the (na\" ive) bootstrap, i.e., as $n$ tends to infinity it holds that $$\begin{aligned}
 	\sup_{h \in \BL{\lXYSub{\cx\oplus\cy} }}\Big|\EV{h\big(\sqrt{n}(\pib^\lambda(\hat \rb_n^*, \sb) - \pib^\lambda(\hat \rb_n, \sb))\big) \Big| X_1, \dots, X_n } \\
 	-  \EV{h\big(\sqrt{n}(\pib^\lambda(\hat \rb_n, \sb) - \pib^\lambda(\rb, \sb))\big)}\Big|\konvP 0.
 \end{aligned}	$$
 \end{theorem}
 Analogous statements on bootstrap consistency are valid for the two-samples case, i.e., when $Y_1^*, \dots, Y_m^* \iid \sb_m^*$ are independent bootstrap realizations from $X_1^*, \dots, X_n^* \iid\rb_n^*$.

\subsection{Entropy Regularized Optimal Transport in the Regime of Vanishing Regularization}\label{sec:RelationUnregularizedOT}
In this section, we assess the limit behavior of the empirical EROT value and Sinkhorn costs in the regime $\lambda\searrow 0$. For this purpose, we introduce the unregularized optimal transport value between $\rb$ and $\sb$ as
\begin{equation}
\OT(\rb, \sb) \coloneqq \inf_{\pi\in \Pi(\rb, \sb)} \langle \cb,\pib \rangle .\tag{OT}\label{eq:OptimalTransportProblem}
\end{equation}
Provided that $\cx\in \lr$, $\cy \in \ls$ it follows by \citet[Theorems~4.1 and 5.10]{villani2008optimal} that $\OT(\rb, \sb)$ is finite and that there exists a (possibly non-unique) optimal solution $\pib^0\in \Pi(\rb, \sb)$ such that $OT(\rb, \sb) = \langle \cb, \pib^0\rangle$. 
Under this assumption it also follows that the OT value admits a dual formulation
\begin{align*}
	\OT(\rb, \sb) &= \sup_{(\alphab, \betab)\in \Phi_c(\rb, \sb) 
	} \langle\alphab,\rb\rangle + \langle\betab,\sb\rangle,
\end{align*}
where  $\Phi_c(\rb, \sb) \coloneqq \{(\alphab, \betab)\in \lr\times \ls \colon \alphab\oplus \betab\leq c\}$. Optimizing elements of the dual formulation are referred to as OT potentials. 

For our analysis of the EROT value we require an upper bound for the quantities $|\EROT(\rb, \sb) - \OT(\rb, \sb)|$ and $|\SOT(\rb, \sb) - \OT(\rb, \sb)|$. Let $\pib^0$ be an OT plan for the non-regularized problem \eqref{eq:OptimalTransportProblem} for $\rb, \sb$ and denote for $\lambda>0$ its entropy regularized counterpart by $\pib^\lambda$.
By optimality for their respective optimization problem it follows that 
$$ \langle\cb, \pib^0 \rangle \leq  \langle\cb, \pib^\lambda \rangle\leq \langle\cb, \pib^\lambda \rangle +  \lambda M(\pib^\lambda) \leq \langle\cb, \pib^0 \rangle +  \lambda M(\pib^0),$$ 
where $M(\,\cdot\,)$ represents the mutual information. This yields 
$$ 0 \leq \SOT(\rb, \sb) - \OT(\rb, \sb) \leq \EROT(\rb, \sb) - \OT(\rb, \sb) \leq \lambda M(\pib^0).$$
 Most notably, if one of the probability measures has finite entropy, i.e.,  
 \begin{align}\label{eq:DefinitionEntropyFunctional}
	H(\rb, \sb) \coloneqq \min\bigg(\sum_{x \in \XC}r_x\log\Big(\frac{1}{r_x}\Big), \sum_{y \in \YC}s_y\log\Big(\frac{1}{s_y}\Big)\bigg)< \infty,
 \end{align}
we obtain by the bound $0 \leq M(\pib) \leq H(\rb, \sb)$ for any $\pib\in \Pi(\rb,\sb)$ \citep[Theorem 2.4.1]{cover1991elements} that 
\begin{align}\label{eq:StabilityInTermsOfRegularization}
	|\SOT(\rb, \sb) - \OT(\rb, \sb)|\leq|\EROT(\rb, \sb) - \OT(\rb, \sb)|\leq \lambda H(\rb, \sb)= \landau(\lambda).
\end{align}
We like to point out that if a probability measure fulfills the Borisov-Dudley-Durst condition, then it also admits finite entropy. 

\begin{remark}For finitely supported probability measures it was shown by \cite{cominetti1994asymptotic} that  $|\SOT(\rb, \sb) - \OT(\rb, \sb)|= o(\exp(-\kappa/\lambda))$ for some constant $\kappa>0$ depending on $\rb$ and $\sb$ as $\lambda$ tends to zero. However, as noted by \cite{Weed18_ExplicitAnalysis} the upper bound $o(\exp(-\kappa/\lambda))$ appears to fail in general for countably supported probability measures  where instead the rate $|S^\lambda(\rb, \sb) - \OT(\rb, \sb)|= \landau(\lambda)$ seems to be tight.
\end{remark}

For a sequence of i.i.d. samples $X_1, \dots, X_n \iid \rb$ and its associated empirical measure $\hat \rb_n$ we obtain by  \citet[Corollary 1]{antos2001convergence} that $\lim_{n \rightarrow \infty}H(\hat \rb_n, \sb) = H(\rb, \sb)$, hence $\limsup_{n \in \N}H(\hat \rb_n, \sb) < \infty$. Performing a  decomposition gives $$ \begin{aligned}
	&\;\sqrt{n} \big(\EROTLambda{\lambda(n)}(\hat \rb_n, \sb) - \EROTLambda{\lambda(n)}(\rb, \sb)\big) \\
	=&\; 	\sqrt{n} \big(\EROTLambda{\lambda(n)}(\hat \rb_n, \sb) - \OT(\hat\rb_n, \sb)\big)+ \sqrt{n} \big(\OT(\rb, \sb) - \EROTLambda{\lambda(n)}(\rb, \sb)\big) \\
	 &\;+  \sqrt{n}\big(\OT(\hat \rb_n, \sb) - \OT(\rb, \sb)\big). 
\end{aligned} $$
For $\lambda(n) = o(1/\sqrt{n})$ the first and second term converge to zero.  Similar assertions remain valid if $\EROTLambda{\lambda(n)}$ is replaced by $\SOTLambda{\lambda(n)}$. 
The weak convergence of the third term has been analyzed by \cite{tameling18} for the setting where the cost function is selected as the power of a metric. More precisely, for a general metric space $(M, d)$ and countable subsets $\XC, \YC$ contained therein, they consider costs of the form $c_p(x,y) \coloneqq d^p(x,y)$ for $p \geq 1$. 
Then, for $n\rightarrow \infty$, the empirical non-regularized OT value
converges in distribution to the supremum of a Gaussian process, 
\begin{equation}
	\sqrt{n}\big(\OT(\hat \rb_n, \sb) - \OT(\rb, \sb)\big) \konvD\sup_{\alphab^* \in \mathcal{S}^*} \langle \alphab^*, \Gb_\rb \rangle,\label{eq:LimitLawNonRegOT}
  \end{equation}
 where $\mathcal{S}^*$ denotes the set of dual optimizers of \eqref{eq:OptimalTransportProblem} and $\Gb_\rb$ represents the Gaussian process with covariance $\Sigma(\rb)$.
 This distributional limit is valid if, for some $z\in (M,d)$, it holds that 
 \begin{align*}
	\sum_{x\in\XC}d^p(x,z)\sqrt{r_x}<\infty \quad \text{and}\quad \sum_{y\in\YC}d^p(y,z)s_y<\infty.
 \end{align*}

Recently, distributional limits for the empirical unregularized OT value on countable spaces have been extended by \cite{hundrieser2022unifying} to arbitrary uniformly bounded cost functions $c$. They  assert a similar asymptotic distribution as in \eqref{eq:LimitLawNonRegOT} if the estimated population measure $\rb$ fulfills the Borisov-Dudley-Durst condition $\sum_{x\in\XC}\sqrt{r_x}<\infty$. 

  Both limit laws, for non-regularized OT value and its entropy regularized counterpart, are essentially governed by their respective set of population dual optimizers. However, unlike the entropic variant where the limit law is always characterized by a centered normal due to uniqueness of dual optimizers, the set of dual optimizers for the non-regularized OT problem may not be unique resulting in a supremum of Gaussian distributions rather than a Gaussian distribution. 
  This suggests that  in general it is not sufficient to select a small but fixed regularization parameter $\lambda>0$ as such choice crucially affects the corresponding weak limit, leading to a  centered normal asymptotic law. Instead, the regularization parameter should decrease with increasing sample size.

Nevertheless, under unique OT potentials (up to a constant shift) the distributional limit for the empirical OT value is indeed centered normal. Such uniqueness is in fact often expected under different measures. According to \citet[Lemma 6 and Corollary 1]{staudt2022uniqueness} it is always guaranteed for $\rb\in \PC(\XC), \sb\in \PC(\YC)$ if  
\begin{align*}
	\sum_{x\in A}\rb_x \neq \sum_{y\in B}\sb_y
\end{align*} for every non-empty proper pair of subsets $A\subset \supp(\rb)$ and $B\subset \supp(\sb)$. 
Hence, given a uniformly bounded cost and measures $\rb\in\PC(\XC),\sb\in\PC(\YC)$ which satisfy the above condition, there exists a unique (up to a constant shift) pair of OT potentials $(\alpha^0,\beta^0)$. If the Borisov-Dudley-Durst condition is met by $\rb$, it  follows by \cite{hundrieser2022unifying}~that 
\begin{align*}
	\sqrt{n} \big(OT(\hat \rb_n, \sb) -OT(\rb, \sb)\big)&\konvW \mathcal{N}(0, \textup{Var}_{X\sim \rb}[\alpha^0_X]),
\end{align*}
while for the EROT value with $\lambda>0$ under  identical assumptions Theorem \ref{them:LimitLawEmpEROT_Value} yields that
\begin{align*}
	\sqrt{n} \big(\EROTLambda{\lambda}(\hat \rb_n, \sb) - \EROTLambda{\lambda}(\rb, \sb)\big)&\konvW \mathcal{N}(0, \textup{Var}_{X\sim \rb}[\alpha^\lambda_X]).
\end{align*}
In particular, based on Proposition \ref{prop:ContinuityEROTVanishingRegularization}$(ii)$, we obtain for suitably shifted EROT potentials $\alpha^\lambda$ pointwise convergence to $\alpha^0$, and hence using our quantitative bounds  in terms of  $\norm{c}_\infty$ for these EROT potentials (Proposition \ref{prop:BoundsOptimalPotentials}) it follows from dominated convergence that 
\begin{align*}
	\lim_{\lambda \searrow 0}\textup{Var}_{X\sim \rb}[\alpha^\lambda_X] =  \textup{Var}_{X\sim \rb}[\alpha^0_X].
\end{align*}
Moreover, as the Borisov-Dudley-Durst condition for $\rb$ also implies that $\norm{\hat \rb_n  - \rb}_{\ell^1(\XC)}\rightarrow 0$ a.s.\ (see, e.g., \citealt[Lemma 2.10.14]{van1996weak}),   it also follows for any sequence $\lambda_n = o(1)$ a.s.\ that
$$\lim_{n\rightarrow \infty}\textup{Var}_{X\sim \hat\rb_n}[\alpha^{\lambda_n}_X] =\textup{Var}_{X\sim \rb}[\alpha^0_X] .$$
As an appealing consequence, this showcases under bounded costs and unique OT potentials that the limit distribution for the empirical OT value can be consistently approximated using the empirical EROT quantities. 

\begin{remark}[Degeneracy of limit law]
	As noted by \cite{tameling18} the limit distributions provided in \eqref{eq:LimitLawNonRegOT} may degenerate in certain cases, namely when the set of dual optimal solutions $\mathcal{S}^*$ for the non-regularized OT problem only contains constant elements. This occurs e.g. under $\rb = \sb$ with $\supp(\rb) = \XC$ for a cost function $c(x,y) = d^p(x,y)$ with $p>1$ if and only if $\XC$ has no isolated points. In view of \citet[Section 4]{hundrieser2022unifying} this setting amounts to unique and $\rb$-a.s. constant dual OT potentials.  Hence, in this case, the limit distribution of the empirical EROT value and Sinkhorn cost degenerate for $\lambda(n) = o(1)$.
	In contrast, for fixed $\lambda>0$ the limit law generally does not degenerate since population EROT potentials are typically non-constant (Remark \ref{rem:NonConstantDualPotentials}) and as the Hadamard derivative of the EROT plan differs from zero (Remark \ref{rem:NonDegDerivative}). 
\end{remark}

\section{Examples}\label{sec:Examples}

In this section we discuss a number of examples where our theory asserts novel distributional limits for EROT quantities and compare it to related statistical contributions on EROT. We focus on three settings: (1) uniformly bounded costs on general countable spaces, (2) countable spaces embedded in a metric space $(M,d)$ and a cost function given by the $p$-th power of a metric with $p \geq 1$, and, as a special instance thereof, (3) countable spaces contained within Euclidean spaces  with a norm as the metric. Since our summation constraints for the distributional limits on the EROT value and plan solely depend on the choice of the dominating functions for the cost function we outline possible choices in each respective setting and omit the specific summation constraints. Further, unless we define different collections of dominating functions $\cxm, \cxp, \cym, \cyp$ and $\tcxm, \tcxp,\tcym,\tcyp$ they are to be selected identically.

\subsection{Uniformly Bounded Costs}\label{subsec:Expl:Bounded}
In case of a uniformly bounded cost function $c\colon \XC\times\YC\rightarrow \RR_{\geq 0}$ with $C \coloneqq \sup_{(x,y)\in \XC\times \YC}c(x,y)<\infty$ we consider for \eqref{eq:LowerUpperBoundsC} the dominating functions $\cxm, \cxp\colon\XC\rightarrow \RR$, $\cym, \cyp\colon \YC\rightarrow \RR$ given by
\begin{align*}
	\cxm(x) = \cym(y)= 0\quad \text{ and } \quad \cxp(x)= \cyp(y) = \frac{C}{2}.
\end{align*}
Then it follows for any $x\in \XC$ and $y\in \YC$ that 
\begin{align*}
	\cx(x) = \cy(y) = 1+\frac{C}{2} \asymp 1 \quad \text{ and } \exSup{}(x) = \eySup{}(y) = \exp\left(\frac{C}{2\lambda}\right)\asymp 1.
\end{align*}
Herein, the relation $a\asymp b$ mean that there exists a constant $\kappa>0$ such that $a/k\leq b \leq ak$. 
In consequence, for the distributional limits of the empirical EROT value and plan we require that the respective estimated measure fulfills the standard Borisov-Dudley-Durst condition \citep{durst1980,borisov1981some,Borisov1983}.

\subsection{Costs induced by Powers of a Metric}

For a general metric space $(M, d)$ and countable subsets $\XC, \YC$ contained therein, a suitable cost function is given by $c_p(x,y) \coloneqq d^p(x,y)$ for $p \geq 1$. This is motivated by the $p$-th Wasserstein distance $W_p$, which is defined for probability measures $\rb\in \PC(\XC)$, $\sb\in \PC(\YC)$ by  $$W_p^p(\mu, \nu) \coloneqq \OT_{\!c_p}(\rb, \sb) =  \inf_{\pi\in \Pi(\rb, \sb)}\langle c_p, \pi\rangle.$$
For this cost function, the dominating functions for our limit distribution theory on the empirical EROT value and plan are to be selected as follows. For fixed $z\in M$ we define $$\RC_\XC = \RC_\XC(z) \coloneqq\sup_{x\in \XC}d(x,z) \quad \text{ and } \quad \RC_\YC = \RC_\YC(z) \coloneqq \sup_{y\in \YC}d(y,z),$$
and differentiate between four settings: $(i)$ bounded spaces $\RC_\XC, \RC_\YC<\infty$, and $(ii)$ bounded and unbounded spaces $\RC_\XC<\infty$ and $\RC_\YC = \infty$ (we refer to this setting as semi-bounded), $(iii)$ unbounded spaces $\RC_\XC = \RC_\YC = \infty$, and $(iv)$ potentially unbounded spaces with a separability constraint and $p = 1$. 

\subsubsection{Bounded Setting}\label{subsec:Expl:Bounded2}
For setting $(i)$ it follows by triangle inequality that $c_p$ is uniformly bounded, $$0\leq c_p(x,y)= d^p(x,y) \leq (d(x,z) + d(z,y))^p\leq (\RC_\XC + \RC_\YC)^p<\infty,$$
and the dominating functions are to be selected according to Section \ref{subsec:Expl:Bounded}. In particular, for distributional limits of empirical EROT value and plan under estimated $\rb$ or $\sb$ we require the standard Borisov-Dudley-Durst condition to be satisfied.

\subsubsection{Semi-Bounded Setting}\label{subsec:Expl:BoundedUnbounded}
For setting $(ii)$ we obtain for $(x,y) \in  \XC\times \YC$, upon denoting $(t)_+^p\coloneqq  \max(0,t)^p$ for $t \in \RR$, 
\begin{align*}
	c_p(x,y) = d^p(x,y) \begin{cases}
		\leq (d(y,z) + d(z,x))^p \!\!\!\!\!\!\!&\leq (d(y,z) +\RC_\XC)^p, \\
		\geq (d(y,z) - d(z,x))_+^p\!\!\!\!\!\!\!&\geq (d(y,z) - \RC_\XC)_+^p.
	\end{cases}
\end{align*}

This gives rise to dominating functions  $\cxm, \cxp\colon\XC\rightarrow \RR$ and $\cym, \cyp\colon \YC\rightarrow \RR$ defined by
\begin{equation}\label{eq:choicesDominatingFunctionsSemiBounded}
\begin{aligned}
	\cxm(x) &=  0 && \cym(y) = (d(y,z) - \RC_\XC)_+^p, \\
	\cxp(x) &=  0, &&
	\cyp(y) = (d(y,z) + \RC_\XC)^p, 
\end{aligned}
\end{equation}
and hence 
\begin{align*}
	\cx(x) &=  1, && \cy(y) = 1 + |\cym(y)| + |\cyp(y)| \asymp 1 + d^p(y,z), \\
	\exSup{}(x) &= 1, &&
	\eySup{}(y) = \exp\left(\frac{\cyp(y) - \cym(y)}{\lambda} \right) \leq \exp\left(\frac{(d(y,z) + \RC_\XC)^p}{\lambda}\right).
\end{align*}
In particular, if $p \in \NN$, it holds for any $y\in \YC$ that 
\begin{align*}
	\cyp(y) - \cym(y)
	\begin{cases}
		\leq (d(y,z) + \RC_\XC)^p \leq 2^p\RC_\XC^p & \text{ if }d(y,z)\leq \RC_\XC,\\
		= \sum_{\substack{i = 1\\i \text{ odd}}}^p 2\binom{p}{i}\RC_\XC^{i}d^{p-i}(y,z)& \text{ if }d(y,z)> \RC_\XC,
	\end{cases}
\end{align*}
which yields for any entropy regularization $\lambda>0$ and arbitrary $\gamma, \epsilon>0$  the relations
\begin{align}\label{eq:BoundedUnbounded_EySupBounds}
	\eySup{}(y) \asymp  \exp\Big( 2\lambda^{-1}\textstyle \sum_{\substack{i = 1\\i \text{ odd}}}^p \binom{p}{i}\RC_\XC^{i}d^{p-i}(y,z)\Big)\begin{cases}\lesssim \exp\left( \gamma d^{p-1+\epsilon}(y,z)\right),\\
	\gtrsim \exp\left( \gamma d^{p-1-\epsilon}(y,z)\right).\end{cases}
\end{align}

These considerations assert the validity of distributional limits for the empirical EROT value and plan under estimated $\rb$ if the standard Borisov-Dudley-Durst condition is met, in conjunction with $\cy, \eySup{2}\in \ls$. For the distributional limit of the empirical EROT value if $\sb$ is (also) estimated we (additionally) require that
\begin{align}\label{eq:EROT_Value_BoundedUnbounded}
	\sum_{y\in\YC}\cy(y)\sqrt{s_y} +  \eySup{2}(y)s_y < \infty,%
\end{align}
whereas for the empirical EROT plan we require that 
\begin{align}\label{eq:EROT_Plan_BoundedUnbounded}
	\sum_{y\in\YC}\cy(y)\eySup{4}(y)\sqrt{s_y}<\infty.%
\end{align}
Sufficient conditions for \eqref{eq:EROT_Value_BoundedUnbounded} and \eqref{eq:EROT_Plan_BoundedUnbounded} are provided in the subsequent lemma. 
\begin{lemma}
\label{prop:SemiBounded_Sufficient}
Suppose $\YC\subseteq (M,d)$ and set $f_\YC\colon \NN\rightarrow \NN,  n\mapsto \#\{y\in \YC\colon n-1\leq d(y,z)< n\}.$ Let $\sb \in \PC(\YC)$ and 
assume that there exist $\overline \gamma,  \delta,  \epsilon>0$ such that \begin{align}
	&\sum_{y\in \YC}\exp\left( \overline \gamma d^{p-1+ \epsilon}(y,z)\right)\!s_y<\infty \quad \text{ and }\quad \label{eq:MeasureConditionForSimplicationOfAssumptions} %
	f_\YC(n) \lesssim \exp\left( \delta n^{p-1+ \epsilon/2}\right).
	\end{align}
	Then, for $\gamma \coloneqq \overline \gamma/10$ in \eqref{eq:BoundedUnbounded_EySupBounds},  %
	it follows that conditions \eqref{eq:EROT_Value_BoundedUnbounded} and \eqref{eq:EROT_Plan_BoundedUnbounded} are met. 
\end{lemma}

The proof of Proposition \ref{prop:SemiBounded_Sufficient} is given in Appendix \ref{app:ProofsExamples}. 
The summability constraint in \eqref{eq:MeasureConditionForSimplicationOfAssumptions} means that the measure $\sb$ needs to be \emph{sub-Weibull} for the metric $d$ with an order strictly greater than $p -1$. The growth condition in \eqref{eq:MeasureConditionForSimplicationOfAssumptions} requires that the elements of $\YC$ are sufficiently spread. It is always fulfilled if $f_\YC$ grows at most polynomially as $n$ tends to infinity.

\subsubsection{Unbounded Setting}\label{subsec:Expl:Unbounded}
For setting $(iii)$ it follows by triangle inequality and Jensen's inequality that \begin{align*}
	0 \leq c_p(x,y) 
	\leq (d(y,z) + d(z,x))^p
	\leq  2^{p-1} (d^p(y,z) + d^p(z,x)).
\end{align*}
Hence, for arbitrary $p\geq 1$ we may select dominating functions  $\cxm, \cxp\colon\XC\rightarrow \RR$ and $\cym, \cyp\colon \YC\rightarrow \RR$ defined by
\begin{align*}
	\cxm(x) &=  0, && \cym(y) =0, \\
	\cxp(x) &=  2^{p-1}d^p(x,z), &&
	\cyp(y) = 2^{p-1}d^p(y,z), 
\end{align*}
which yields for $\epsilon, \gamma>0$ the relation
\begin{align*}
	\cx(x) &=  1 + 2^{p-1}d^p(x,z) \asymp 1 + d^p(x,z),&& \cy(y) = 1 + 2^{p-1}d^p(y,z)\asymp 1 + d^p(y,z), \\
	\exSup{}(x) &= \exp\left(\frac{2^{p-1}d^p(x,z)}{\lambda}\right) &&
	\eySup{}(y) = \exp\left(\frac{2^{p-1}d^p(y,z)}{\lambda}\right)\\
	&\lesssim \exp\left(\gamma d^{p+\epsilon}(x,z)\right), && \color{white}{
	\eySup{}(y)}\color{black}\lesssim \exp\left(\gamma d^{p+\epsilon}(y,z)\right).
\end{align*}

Moreover, for $p\in \NN$ more refined dominating functions may be selected. Further we consider different collections of dominating functions $\cxm, \cxp, \cym, \cyp$ and $\tcxm, \tcxp,\tcym,\tcyp$. Given $\epsilon, \gamma>0$ we use triangle inequality and H\"older's inequality onto the $i$-th element of the subsequent sum such that 
\begin{align*}
	d^p(x,y) &\leq (d(y,z) + d(z,x))^p= \textstyle \sum_{i = 0}^{p}  \binom{p}{i} d^{p-i}(y,z)d^i(x,z) \\
	&\leq \textstyle d^{p}(y,z) + \frac{\lambda\gamma}{2} d^{p-1+\epsilon}(y,z) + \sum_{i = 1}^{p-1} K_i d^{q_i}(x,z)   +d^{p}(x,z), 
\end{align*}
where $q_i = q(i,p,\epsilon) \coloneqq \frac{p-1+\epsilon}{i-1+\epsilon}$ and $K_i = K(i,p,\epsilon, \gamma, \lambda) \coloneqq \binom{p}{i}^{\frac{p-1+\epsilon}{i-1+\epsilon}}\left(\frac{\lambda\gamma}{2(p-1)}\frac{p-i}{i-1+\epsilon}\right)^{-\frac{p-i}{p-1+\epsilon}}$ represent for each $i \in \{1, \dots, p-1\}$ the dual H\"older exponent and coefficient, respectively.%

Again invoking triangle inequality we find since $p \in \NN$ that 
\begin{align*}
	d^p(x,y) &\geq(d(y,z) - d(z,x))^p= \textstyle \sum_{i = 0}^{p} (-1)^{i} \binom{p}{i} d^{p-i}(y,z)d^i(x,z) \\
	&\geq \textstyle d^{p}(y,z) - \frac{\lambda\gamma}{2} d^{p-1+\epsilon}(y,z) - \sum_{i = 1}^{p-1} K_i d^{q_i}(x,z)  + (-1)^{p}d^{p}(x,z).
\end{align*}
Hence, under $p\in \NN$ we may select dominating functions given by 
\begin{equation}\label{eq:choicesDominatingFunctionsUnbounded1}
\begin{aligned}
	\cxm(x) &=  (-1)^{p}d^{p}(x,z)- \textstyle\sum_{i = 1}^{p-1} K_i d^{q_i}(x,z),  && \cym(y) =d^{p}(y,z) - \frac{\lambda\gamma}{2} d^{p-1+\epsilon}(y,z) , \\
	\cxp(x) &=  d^{p}(x,z)+\textstyle\sum_{i = 1}^{p-1} K_i d^{q_i}(x,z), &&
	\cyp(y) = d^{p}(y,z) + \frac{\lambda\gamma}{2} d^{p-1+\epsilon}(y,z).
\end{aligned}
\end{equation}
	By symmetry of the metric, we may also consider the dominating functions 
	\begin{equation}
		\begin{aligned}
		\label{eq:choicesDominatingFunctionsUnbounded2}
	\tcxm(x) &=  d^{p}(x,z) - \frac{\lambda\gamma}{2} d^{p-1+\epsilon}(x,z),  && \tcym(y) = (-1)^{p}d^{p}(y,z)- \textstyle\sum_{i = 1}^{p-1} K_i d^{q_i}(y,z), \\
	\tcxp(x) &= d^{p}(x,z) + \frac{\lambda\gamma}{2} d^{p-1+\epsilon}(x,z), &&
	\tcyp(y) = d^{p}(y,z)- \textstyle\sum_{i = 1}^{p-1} K_i d^{q_i}(y,z).
\end{aligned}
\end{equation}
As a consequence, we obtain 
\begin{align*}
	\cx(x) &\asymp 1 +  d^{p}(x,z) +d^{(p-1+\epsilon)/\epsilon}(x,z), &&\cy(y) \asymp 1 +  d^{p}(y,z), \\
	\tcx(x) &\asymp 1 +  d^{p}(x,z), &&\tcy(y) \asymp 1 +  d^{p}(y,z) +d^{(p-1+\epsilon)/\epsilon}(y,z),\\
	\texSup{}(x) &\asymp \exp(\gamma d^{p-1+\epsilon}(x,z)), && \eySup{}(y) \asymp \exp(\gamma d^{p-1+\epsilon}(y,z)).
\end{align*}
	
These functions provide the weights for the summability constraints of our distributional limits for the EROT value if $\rb$ or $\sb$ or both are estimated as outlined in Theorem \ref{them:LimitLawEmpEROT_Value}. 
In particular, based on Proposition \ref{prop:SemiBounded_Sufficient}, it follows that the summation constraints are fulfilled if the measures $\rb$ and $\sb$ are sub-Weibull for the metric $d$ of an order strictly greater than $p-1$ and if additionally the sets $\XC$ and $\YC$ are sufficiently spread out in $(M,d)$. 

Notably, for this setting our theory on distributional limits for the empirical EROT plan does not apply since $\norm{\cxp - \cxm}_\lInfX = \infty$ and $\norm{\cyp - \cym}_\lInfY= \infty$. Nevertheless, as the following subsection highlights, under an additional separability constraint we can formulate refined dominating functions, which enables us to apply our distributional limits for the empirical EROT plan even if both spaces are unbounded.

\subsubsection{Unbounded Setting with Separability Constraint}\label{subsec:Expl:Separability}
Assume that the countable spaces $\XC, \YC$ are unbounded subsets in the metric space $(M,d)$ that fulfill the condition
\begin{equation}
\kappa \coloneqq \sup_{(x,y) \in \XC\times \YC}\Big(d(x,z) + d(z,y) - d(x,y)\Big)< \infty.\label{eq:ConditionCostfunctionWithMetric}
\end{equation} 
This condition is to be interpreted as a separability constraint between the spaces $\XC$ and $\YC$, i.e., the larger the distances $d(x,z)$ and $d(z,y)$, the larger the distance $d(x,y)$. Under condition \eqref{eq:ConditionCostfunctionWithMetric} and by triangle inequality we can select the dominating functions for $c_1$ as 
\begin{equation}\label{eq:SeparabilityDominatingFunctions}
\begin{aligned}
	\cxm(x) &= d(x,z) - \kappa/2, \quad \quad && \cym(x) = d(y,z) - \kappa/2,\\
	\cxp(x) &= d(x,z), && \cyp(x) = d(y,z),
\end{aligned}
\end{equation}
which results in
\begin{align*}
	\cx(x) &\asymp 1+ d(x,z), && \cy(y) \asymp 1 + d(y,z),\\
	\exSup{}(x)&\asymp 1,&& \eySup{}(y) \asymp 1.
\end{align*}
	
This means that the summability constraints for the distributional limits of the empirical EROT value and plan only consist in weighted Borisov-Dudley-Durst conditions with weights given by $\cx$ and $\cy$. 

\subsection{Costs induced by Powers of a Norm on Euclidean Space}
\label{subsec:SquaredEuclideanCosts}

We now assume that $\XC$ and $\YC$ are contained in a Euclidean space and consider a cost function $c(x,y) = \norm{x-y}^p$ for some $p\in \NN$ and some arbitrary norm $\norm{\cdot}$. 
Based on the previous section we provide a single proposition for the two-sample case. 

\begin{proposition}[EROT limits laws for powers of a norm]\label{prop:EROT_EuclideanNorm}
	Let $\XC, \YC$ be countable spaces in $(\RR^d, \norm{\cdot})$, consider costs $c_p(x,y) \coloneqq \norm{x-y}^p$ for $p \in \NN$ and take probability measures $\rb \in \PC(\XC), \sb\in \PC(\YC)$. 
	Consider  one of the following four settings.
	\begin{enumerate}
		\item[$(i)$] Suppose that $\XC, \YC$ are bounded subsets of $\RR^d$ and assume that $$\sum_{x\in \XC}\sqrt{r_x}<\infty \quad \!\text{ and } 
		\quad  \sum_{y\in \YC}\sqrt{s_y}<\infty.$$
		\item[$(ii)$] Suppose that $\XC$ is a bounded subset of $\RR^d$ while $\YC$ is not. Assume that the function $f_\YC\colon \NN\rightarrow \NN,  n\mapsto \#\{y\in \YC\colon n-1\leq \norm{y}< n\}$ grows at most polynomial as $n \rightarrow \infty$. Further, assume there exist $\epsilon, \gamma >0$ such that, $$ \color{white}\exp( \norm{y}^{p-1+\epsilon})\color{black}\sum_{x\in \XC}\sqrt{r_x} <\infty \quad \text{ and } 
		\quad  \sum_{y\in \YC} \exp(\gamma \norm{y}^{p-1+\epsilon}) s_y<\infty.$$
		\item[$(iii)$]Suppose that $\XC, \YC$ are unbounded subsets of $\RR^d$.  Assume that the functions $f_\XC\colon \NN\rightarrow \NN,  n\mapsto \#\{x\in \XC\colon n-1\leq \norm{x}< n\}$ and $f_\YC$ grow at most polynomial as $n \rightarrow \infty$. Further, assume there exist $\epsilon, \gamma >0$ such that, $$\sum_{x\in \XC} \exp(\gamma \norm{x}^{p-1+\epsilon}) r_x <\infty \quad \text{ and } 
		\quad   \sum_{y\in \YC} \exp(\gamma \norm{y}^{p-1+\epsilon}) s_y<\infty.$$ 
		\item[$(iv)$]  Take $\norm{\cdot}= \norm{\cdot}_1$ and let $p = 1$. Assume  that $\XC\subseteq (-\infty, a]$, $\YC\subseteq [b, \infty)$ for elements $a,b\in \R^d$ and suppose $$ \sum_{x\in \XC} (1+\norm{x}_1)\sqrt{r_x} <\infty \quad \text{ and } 
		\quad   \sum_{y\in \YC}(1+\norm{y}_1)\sqrt{s_y}<\infty.$$
	\end{enumerate}
	Then, assertion (ii) of Theorem \ref{them:LimitLawEmpEROT_Value} for the empirical EROT value holds. Further, under setting $(i)$, $(ii)$, or $(iv)$,  assertion (iii) of Theorem \ref{them:LimitLawEmpEROT_Plan} for the empirical EROT plan holds. 
\end{proposition}

\begin{remark}[Comparison to other works on squared Euclidean costs]
	The work by \cite{mena2019} provide distributional limits for the EROT value under squared Euclidean costs for sub-Gaussian measures by centering with the expectation of the empirical quantities.  
	A refinement of this distributional limit has been recently achieved by \cite{del2022improved} and \cite{goldfeld2022statistical} during the time this work was under review, asserting that the centering indeed can be selected as the population quantity. Their limit law is in line with our findings and is also given by a centered normal whose variance is characterized by the respective variance of the EROT potentials. 
	
	Proposition \ref{prop:EROT_EuclideanNorm} complements these contributions as we do not focus only on the squared Euclidean norm but consider more general choices for $p\in \NN$ as well as other norms, while demanding the measures to be concentrated on countable sets. Further, for $p = 2$ and measures concentrated on sufficiently spread countable sets, e.g., $\XC= \YC= \delta\Z^d$ for $\delta>0$, we do not require sub-Gaussianity but only that the measures are sub-Weibull of order $(1+\epsilon)$ for some $\epsilon>0$. 
\end{remark}

	\section{Sensitivity Analysis}\label{sec:SensitivityAnalysis}

	To formalize the underlying norms in which we perform our sensitivity analysis, we assume throughout this section that $c\colon \XC\times \YC\rightarrow \RR$  satisfies \eqref{eq:LowerUpperBoundsC} for two collections of dominating functions $\cxm, \cxp, \cym, \cyp$ and $\tcxm, \tcxp,\tcym,\tcyp$ and define for these two collections the functions $\cx,\cy,\exSup{},\eySup{}$ and $\tcx,\tcy,\texSup{},\teySup{}$ as in \eqref{eq:DefWeightingFunctions}. Note that the functions $\exSup{},\eySup{}$ and $\texSup{},\teySup{}$ also depend on the regularization parameter $\lambda$. 
	Recall the definition of the function classes $\FC_\XC, \FC_\YC$ from \eqref{eq:FunctionClassFC} and $\HC_\XC, \HC_\YC$ from \eqref{eq:FunctionClassHC}. 
	According to Lemma \ref{lem:equivalenceNorms}, the norm induced by $\FC_\XC$ (resp. $\FC_\YC$) consists of a weighted $\ell^1$-component with $\cx$ (resp. $\cy$) in combination with another component that ensures convergence of integrals of certain individual functions. 
	Further, we define for $\delta \geq 0 $ the functions 
	\begin{equation}\label{eq:DefinitionKXSup}
	\begin{aligned}
		\kxSup{\delta}&\colon \XC\rightarrow [1,\infty), \quad x\mapsto  \cx(x)\exSup{\delta}(x),\\
		\kySup{\delta}&\colon \YC\rightarrow [1,\infty), \quad \,y\mapsto  \cy(y)\eySup{\delta}(y), 
	\end{aligned}
	\end{equation}
	which shall serve as weights for a weighted $\ell^1$-norm. 
	The key motivation for the use of these norms in our analysis arises from our quantitative bounds on the EROT potentials and the plan from Proposition \ref{prop:BoundsOptimalPotentials}.

	\subsection{Continuity of Entropic Optimal Transport Quantities}\label{subsec:ExplicitBoundsOnOptimizerts}

	As a first step for our sensitivity analysis we derive a stability result for the EROT value which relates the perturbation of measures in one component to the respective EROT potentials. All proof of this subsection are deferred to Appendix \ref{subsec:ProofContinuity}. 

	\begin{lemma}[Stability of EROT value]\label{lem:LipschitzEROTValue}
		Consider probability measures $\rb, \tilde \rb\in \PC(\XC), \sb\in \PC(\YC)$ such that $\cx\in \lr\cap \ell^1_{\tilde \rb}(\XC)$, $\cy, \eySup{} \in \ls$ and let $\lambda>0$. Then it follows for EROT potentials $(\alpha^\lambda, \betab^\lambda)$ and $(\tilde \alpha^\lambda, \tilde \betab^\lambda)$ (according to Remark \ref{rem:NonConstantDualPotentials}) for the pairs $(\rb,\sb)$ and $(\tilde \rb, \sb)$, respectively, that 
		\begin{align*}
			\langle \tilde \alphab^\lambda, \rb - \tilde \rb\rangle \leq 	\EROT(\rb, \sb) - \EROT(\tilde \rb, \sb) \leq \langle \alphab^\lambda, \rb - \tilde \rb\rangle 
		\end{align*}
		In particular, it holds that 
		\begin{align*}
			\left|\EROT(\rb, \sb) - \EROT(\tilde \rb, \sb)\right| &\leq \left(\langle \cx, \rb+\tilde \rb\rangle +2\langle \cy, \sb\rangle + 2\lambda\log\langle \eySup{}, \sb\rangle \right)\norm{\rb - \tilde \rb}_{\lXSub{\cx}}.%
		\end{align*}
	\end{lemma}
	
	Building upon this stability result for the EROT value we verify continuity properties for the EROT value as well as EROT potentials and plan. Notably, these results crucially rely on our quantitative bounds from Proposition \ref{prop:BoundsOptimalPotentials}. 

	\begin{proposition}[Continuity under fixed $\lambda>0$]\label{prop:ConvergenceOptimalDualSolutions}	\label{cor:ConvergenceEROTP}
		Take sequences of probability measures $\rb,(\rb_\ind)_{\ind\in \NN}\in \PC(\XC)$, $\sb, (\sb_\ind)_{\ind \in \NN}\in \PC(\YC)$ and let $\lambda>0$ be fixed. 
		\begin{enumerate}
			\item[$(i)$] 
			Assume that 
			$\cx, \tcx, \texSup{}\in \lr$, $\cy, \tcy, \eySup{}\in \ls$ and  for $\FC_\XC, \FC_\YC$ suppose that $$\lim_{k\rightarrow \infty}\norm{\rb_k-\rb}_{\ell^\infty(\FC_\XC)} + \norm{\sb_k-\sb}_{\ell^\infty(\FC_\YC)}= 0.$$ 
			Then it follows that $$\lim_{\ind\rightarrow \infty}\EROT(\rb_\ind, \sb_\ind) = \EROT(\rb, \sb).$$
			\item[$(ii)$] Under the assumptions of setting $(i)$, take EROT potentials $(\alphab^{\lambda}, \betab^\lambda), ((\alphab^{\lambda}_\ind,\betab^\lambda_\ind))_{\ind \in \N}\subset$ $\R^\XC\times \R^\YC$ of \eqref{eq:DualEntropicOptimalTransportProblem} for $(\rb_\ind, \sb_\ind)_{\ind \in \N}, (\rb,\sb)$ according to Remark \ref{rem:NonConstantDualPotentials} such that $\beta^\lambda_{\ind,y_1}=\beta^\lambda_{y_1} =0$ for some element $y_1\in \supp(\sb)$. Then it follows  for any $x\in\XC, y\in \YC$ that $$\lim_{\ind \rightarrow \infty}\alpha^{\lambda}_{\ind,x} = \alpha^{\lambda}_{x}, \quad \lim_{\ind \rightarrow \infty} \beta^\lambda_{\ind,y} = \beta^\lambda_y. $$
 			\item[$(iii)$] 
			Assume that $\kxSup{1} \in \lr$, $\kySup{1}\in \ls$  for $\kxSup{1}, \kySup{1}$ defined in \eqref{eq:DefinitionKXSup} and suppose that $$\lim_{k\rightarrow \infty}\norm{\rb_k-\rb}_{\lXSub{\kxSup{1}}} + \norm{\sb_k-\sb}_{\lYSub{\kySup{1}}}= 0.$$
			Then it follows for the EROT plans of $(\rb_\ind, \sb_\ind)_{\ind \in \N}, (\rb,\sb)$, respectively, that $$\lim_{k\rightarrow\infty} \norm{\pib^\lambda(\rb_\ind, \sb_\ind) - \pib^\lambda(\rb, \sb)}_{\ell^1_{\cx\oplus\cy}(\XC\times\YC)}=0.$$
		\end{enumerate}
	\end{proposition}

	We emphasize that for the continuity of the EROT plan in the Banach space $\lXYSub{\cx\oplus \cy}$ we consider for the underlying measures the norm of the space $\lXSub{\kxSup{1}}\times \lYSub{\kySup{1}}$ which is (under appropriate choice of the functions $\cxm, \cxp, \cym, \cyp$) at least as strong as the norm of $\ell^\infty(\FC_\XC)\times \ell^\infty(\FC_\YC)$. %
	Indeed, by selecting $\cxm=\tcxm,\cxp=\tcxp$ and $\cym=\tcym, \cyp=\tcyp$, 
	we have $$\norm{\cdot}_{\ell^\infty(\FC_\XC)}\leq \norm{\cdot}_{\lXSub{\kxSup{1}}} \text{ on }\PC(\XC) \quad \text{ and } \quad \norm{\cdot}_{\ell^\infty(\FC_\YC)}\leq \norm{\cdot}_{\lYSub{\kySup{1}}} \text{ on }\PC(\YC),$$ where the latter norms are strictly stronger if $\|\cxp-\cxm\|_{\lInfX}=\infty$ and $\|\cyp-\cym\|_{\lInfY}=\infty$.
	
	In addition to continuity results for a fixed regularization parameter, we also provide a novel insight on the EROT value and potential under suitably converging measures in the regime of vanishing regularization. This complements a convergence result  for %
	 the EROT potential by \citet[Theorem 1.1]{nutz2022entropic} on general Polish spaces but under fixed marginal probability distributions. 

	\begin{proposition}[Continuity under $\lambda\searrow 0$]\label{prop:ContinuityEROTVanishingRegularization}
		Let $c\colon \XC\times \YC\rightarrow \RR$ be a uniformly bounded cost function. 
		Take sequences of probability measures $\rb,(\rb_\ind)_{\ind\in \NN}\in \PC(\XC), \sb, (\sb_\ind)_{\ind \in \NN}\in \PC(\YC)$ such that $H(\rb, \sb) <\infty$ (defined in \eqref{eq:DefinitionEntropyFunctional}) and 
		$$\lim_{\ind\rightarrow \infty}\norm{\rb_k - \rb}_{\lX} + \norm{\sb_k - \sb}_{\lY} = 0.$$
		Further, let  $(\lambda_k)_{k\in \NN}$ be positive regularization parameters with $\lim_{k\rightarrow \infty}\lambda_k = 0$. 
		\begin{enumerate}
			\item[$(i)$]  Then it follows that 
			\begin{align*}
				\lim_{\ind\rightarrow \infty}\EROTLambda{\lambda_\ind} (\rb_\ind, \sb_\ind) = \OT(\rb, \sb).
			\end{align*}
			\item[$(ii)$]  Assume that the (unregularized) OT potential for $\rb, \sb$ is unique (up to a constant shift) on $\supp(\rb)\times \supp(\sb)$. Consider a pair of OT potentials $(\alphab^{0}, \betab^{0})$ for $(\rb,\sb)$ and take EROT potentials $ ((\alphab^{\lambda_\ind}_\ind,\betab^{\lambda_\ind}_\ind))_{\ind \in \N}\subset \R^\XC\times \R^\YC$ for $(\rb_\ind, \sb_\ind)_{\ind \in \N}$ with regularization parameter $(\lambda_\ind)_{\ind \in \NN}$, respectively, such that $\beta^{\lambda_k}_{\ind,y_1}=\beta^0_{y_1} =0$ for some element $y_1\in \supp(\sb)$. Then it follows for any $x\in \supp(\rb), y\in \supp(\sb)$ that $$\lim_{\ind \rightarrow \infty}\alpha^{\lambda_\ind}_{\ind,x} = \alpha^{0}_{x}, \quad \lim_{\ind \rightarrow \infty} \beta^{\lambda_\ind}_{\ind,y} = \beta^0_y. $$
		\end{enumerate}
	\end{proposition}

	In our proof we explicitly make use of the boundedness of the cost function in order to ensure existence of EROT potentials which then can be uniformly bounded independent of the regularization parameter.  Moreover, as part of our proof technique  we show that cluster points of sequences of EROT %
	 potentials with vanishing regularization constitute %
	 dual optimizers of the OT problem. For this reason, it is central in our argument to impose the uniqueness assumption. It remains open to what extent these requirements can be relaxed. Such stability results would be both of statistical as well as computational interest. %

	\subsection{Hadamard Differentiability of Entropic Optimal Transport Value}

	In this section, we prove Hadamard differentiability of the EROT value   with respect to the marginal probability measures for a suitable weighted norm induced by a function class. 
	\begin{definition}  A mapping $\Psi\colon U \to V$ between normed spaces $U, V$ is said to be \emph{Hadamard differentiable} at $u \in U$ if there exists a continuous linear map $\DH_{|u}\Psi \colon U \rightarrow V $ such that for any sequence $h_n\in U$ converging to $h$ and any positive sequence $(t_n)_{n \in \N}$ with $t_n \searrow 0$ such that $u +t_nh_n \in U$ it holds 
			\begin{equation}
			\norm{\frac{\Psi(u + t_nh_n) - \Psi(u)}{t_n} - \DH_{|u}\Psi(h)}_{V} \rightarrow 0
			\label{eq:dirHad}
			\end{equation}
			for $n \rightarrow \infty.$ 
			Further, $\Psi$ is Hadamard differentiable \emph{tangentially} to $U_0\subseteq U$ at $u$ if the limit \eqref{eq:dirHad} exists for all sequences $h_n= t_n^{-1}(k_n - u)$ converging to $h$ where $k_n \in U_0$ and $t_n \searrow 0$.
			The Hadamard derivative  is then defined on the contingent (Bouligand) cone to $U_0$ at $u$
			$$T_{u}(U_0) = \left\{h \in U: h = \lim_{n\to\infty} t_n^{-1}(k_n - u), (k_n)_{n \in \N}\subseteq U_0, (t_n)_{n \in \N} \in \R_{\geq 0}, t_n \searrow0 \right\}. $$
	\end{definition}
	 Concerning the derivative of the EROT value we recall by the relation between primal and dual optimizers $\pib^\lambda, \alphab^\lambda, \betab^\lambda$ for $\rb, \sb$ (Proposition \ref{prop:DualEROT}) that$$ \EROT(\rb, \sb)  = \langle \alphab^\lambda, \rb\rangle + \langle \betab^\lambda, \sb \rangle.$$ This indicates that the Hadamard derivative of $\EROT$ at $(\rb, \sb)$ is characterized by $\langle \alphab^\lambda, \cdot \rangle + \langle \betab^\lambda, \cdot \rangle$, an observation that is in line with findings on finite ground spaces by \cite{bigot2019CentralLT} and \cite{klatt2018empirical}. On countable spaces this is also valid under suitable conditions on the cost functional and the probability measures $\rb, \sb$.

	\begin{theorem}[Hadamard differentiability of EROT value]\label{them:EntropicOTCostUnboundedIsHadamardDifferentiable} Assume that $\cx, \tcx, \exSup{}\in \lr$ and $\cy, \tcy, \texSup{}\in \ls$, and from $\FC_\XC, \FC_\YC$ from \eqref{eq:FunctionClassFC}. 
		Then, for $\lambda>0$ the functional
		\begin{align*}\EROT \colon &\Big( \probset{\XF}\cap\ell^\infty(\FC_\XC)  \Big)\times \Big(\probset{\YC}\cap\ell^\infty(\FC_\YC) \Big) \rightarrow \R%
		\end{align*}
		is Hadamard differentiable at $(\rb, \sb)$ tangentially to $\big( \probset{\XF}\cap\ell^\infty(\FC_\XC) \big)\times \big(\probset{\YC}\cap\ell^\infty(\FC_\YC)  \big) $ with Hadamard derivative \begin{equation}\begin{aligned}
	  \DH_{|(\rb,\sb)}\EROT \colon \,T_{(\rb, \sb)}&\left(\big( \probset{\XF}\cap\ell^\infty(\FC_\XC) \big)\times\big( \probset{\YC}\cap\ell^\infty(\FC_\YC) \big)\right) \rightarrow \R,\\
	  &  (\hb^\XC, \hb^\YC) \mapsto \langle\alphab^\lambda, \hb^\XF \rangle + \langle\betab^\lambda, \hb^\YC \rangle,
	  \end{aligned} \notag\label{eq:ETCOSTSHadamardDerivativeProposed}
	\end{equation}
		 where $(\alphab^\lambda, \betab^\lambda)$ are EROT potentials for probability measures $(\rb,\sb)$. 
	\end{theorem}
	The proof is deferred to Appendix \ref{subsec:ProofsSensitivityEROTV}. Central to its proof is strong duality (Proposition \ref{prop:DualEROT}) \and the correspondence between the EROT potentials (Remark \ref{rem:NonConstantDualPotentials}) in conjunction with pointwise continuity (Proposition~\ref{prop:ConvergenceOptimalDualSolutions}). In particular, we make explicit use of the norm induced by the function classes $\FC_\XC$ and $\FC_\YC$ in order to suitably control the EROT potentials. 

	As a corollary, we obtain a sensitivity result for the Sinkhorn divergence (Section \ref{subsec:EROTValue}).
	
	\begin{corollary}[Hadamard differentiability of Sinkhorn divergence]\label{cor:HadDiffDebiasedEROTvalue}
		\label{cor:DebiasedEntropicOTCostUnboundedIsHadamardDifferentiable} Suppose that $\XC = \YC$ and take a symmetric, positive definite cost function.  Assume that  $\cx, \cy, \tcx,\tcy,$ $\exSup{}, \texSup{}\in \lr\cap\ls$, and consider $\FC_\XC, \FC_\YC$ from \eqref{eq:FunctionClassFC}. 
		Then, for $\lambda>0$ the functional
		\begin{align*}\overline{E\!ROT}^\lambda \colon &\Big( \probset{\XF}\cap\ell^\infty(\FC_\XC\cup\FC_\YC)  \Big)\times \Big(\probset{\YC}\cap\ell^\infty(\FC_\XC\cup\FC_\YC) \Big) \rightarrow \R%
		\end{align*}
		is Hadamard differentiable at $(\rb, \sb)$ tangentially to $\big( \probset{\XF}\cap\ell^\infty(\FC_\XC\cup\FC_\YC) \big)\times \big(\probset{\YC}\cap\ell^\infty(\FC_\XC\cup\FC_\YC)  \big) $ with Hadamard derivative \begin{equation*}\begin{aligned}
	  \DH_{|(\rb,\sb)}\overline{E\!ROT}^\lambda \colon &\,T_{(\rb, \sb)}\left(\big( \probset{\XF}\cap\ell^\infty(\FC_\XC\cup\FC_\YC) \big)\times\big( \probset{\YC}\cap\ell^\infty(\FC_\XC\cup\FC_\YC) \big)\right) \rightarrow \R,\\
	  &  (\hb^\XC, \hb^\YC) \mapsto \langle\alphab^\lambda( \rb,  \sb) - \alphab^\lambda( \rb,  \rb), \hb^\XF \rangle + \langle\betab^\lambda( \rb,  \sb) - \betab^\lambda( \rb,  \rb), \hb^\YC \rangle,
	  \end{aligned} %
	\end{equation*}
		 where $(\alphab^\lambda(\tilde \rb, \tilde \sb), \betab^\lambda(\tilde \rb, \tilde \sb))$ are EROT potentials for probability measures $\tilde \rb, \tilde \sb\in \PC(\XC)$.
	\end{corollary}

	\subsection{Hadamard Differentiability of Entropic Optimal Transport Plan}\label{subsec:HadDiffPlan}

	Deriving the Hadamard derivative of the EROT plan is more challenging and requires more careful considerations. Before we state this in a formal way, we give a heuristic derivation. For this purpose, we consider a fixed element $y_1 \in \YC$ (as in Proposition \ref{prop:ConvergenceOptimalDualSolutions}) and 
	  denote for any $\bb \in \R^\YC$ the element where we omit the entry at $y_1$ by $\bb_*\coloneqq (b_{y_2}, b_{y_3}, \dots)\in \R^{\YCC}$.  
	  We adhere to this notation in order to fix the choice for EROT potentials $(\alpha^\lambda, \beta^\lambda)$ via $\beta_{y_1}^\lambda = 0$, it also enables us to prove invertibility of a certain operator. %
	  
	  In particular, we set $$\probset{\YC}_* = \left\{ \sb_* \in \lYY \colon \sum_{y \in \YCC} s_{*,y} \in [0,1], \;s_{*} \geq 0 \right\}$$ as the set of probability vectors on $\YC$, where we omit the entry for $y_1$. 
	  Note that for a given element $ \sb_* \in \probset{\YC}_*$ we can recover the associated probability measure via the relation $ \sb = (1 - \sum_{y \in \YCC} s_{*,y},  s_{*,y_2},  s_{*,y_3}, \dots).$
	At the core of our approach is the reformulation of the optimality criterion (Proposition \ref{prop:DualEROT}) as an operator defined on suitable spaces.
	We introduce the marginalization operator omitting the entry at $y_1$, i.e., 
	\begin{equation*}
		\Ab_* \colon \ell^1(\XF\times \YC) \rightarrow \ell^1(\XF) \times \ell^1(\YCC), \quad \pib \mapsto \begin{pmatrix}
			\left( \sum_{y\in \YC}{\pi_{xy}} \right)_{x\in \XF}\,\,\,\,\,\,\,\,\,\,\textcolor{white}{.}\\
			\left( \sum_{x\in \XF}{\pi_{xy}} \right)_{y\in \YCC}
		\end{pmatrix}.
	\end{equation*}
	
	Its dual operator, formally a linear mapping $\lInfX\times \lInfYY\rightarrow\lInfXY$, will be extended as follows
	$$\Ab_*^T \colon   \R^\XF \times \R^{\YCC} \rightarrow \R^{\XF \times \YC}, \quad(a, b_*) \mapsto  (a\oplus (0,b_*)),$$
	where $\oplus$ denotes the tensor sum, i.e., $(a\oplus (0,b_*))_{x,y} = a_x + b_{*,y}$ for $(x,y) \in \XC\times \YC$ and $(a\oplus (0,b_*))_{x,y_1} = a_x$ for $x\in \XC$.
	These operators enable us to define %
	\label{text:DefinitionFC}
	$$\begin{aligned} \FC \colon \Big(\ell^1(\XF &\times \YC) \times \R^\XC \times \R^{\YCC} \Big)\times \Big( \ell^1(\XF) \times \ell^1(\YC \backslash \{y_1\}) \Big)  \rightarrow   \R^{\XF \times \YC} \times \R^\XF\times \R^{\YCC},\notag\\
	&\Big((\pib, \alphab, \betab_*), (\rb, \sb_*)\Big) \mapsto  \begin{pmatrix}
		\pib - \exp\left( \frac{1}{\lambda}\left[\Ab_*^T(\alphab, \betab_*) - \cb \right]\right)\odot (\rb \otimes \sb)\\
		\Ab_* (\pib) - \begin{pmatrix}\rb\\
		\sb_*
		\end{pmatrix}
	\end{pmatrix}, \label{text:DefinitionFC}\notag
	\end{aligned}$$
	where $\odot$ denotes the component-wise product and $\otimes$ is defined as the tensor product of the probability measures $(\rb \otimes \sb)_{xy} = r_xs_y$ for all $(x,y) \in \XC\times \YC$. 
	In terms of $\FC$ the optimality criterion from Proposition \ref{prop:DualEROT} is restated as follows.

	\begin{lemma}\label{lem:RewrittenOptimalityCriterionWithFunctionF}
		For given probability measures $(\rb,\sb)\in \probset{\XF}\times \probset{\YC}$ and $\lambda>0$ the element $\pib\in \ell^1(\XF \times \YC)$ and the potentials $(\alphab, (0,\betab_*))\in \lr\times \ls \subseteq \R^\XF \times \R^\YC$ are optimal for \eqref{eq:EntropicOptimalTransport} and \eqref{eq:DualEntropicOptimalTransportProblem}, respectively,  if and only if \begin{equation} \FC\Big((\pib,\alphab, \betab_*), (\rb, \sb_*)\Big) = 0.\label{eq:OptimalityCriterionImplicitlyDefinedFunction}\end{equation}
	\end{lemma}
	Equation \eqref{eq:OptimalityCriterionImplicitlyDefinedFunction} intuitively contains a proposal for the Hadamard derivative of $\pib^\lambda$. 
	Indeed, a na\" ive application of the usual calculus of partial derivatives  (denoted by $\D$) from finite-dimensional spaces with respect to $\rb, \sb_*$ yields the partial derivative of $\FC$ at $\tilde\rb, \tilde\sb_*\in \probset{\XC}\times \probset{\YC}_*$ for fixed elements  $\tilde{\pib}\in \ell^1(\XF\times \YC)$, $(\tilde\alphab, \tilde\betab_*)\in \R^{\XF}\times \R^{\YCC}$, given by 
	\begin{align*}
			& \DCW_{\rb, \sb_*|(\tilde\pib,\tilde\alphab,\tilde\betab_*, \tilde\rb, \tilde\sb_*) }\FC \colon  \; \ell^1(\XF) \times \ell^1(\YC \backslash \{y_1\}) \rightarrow   \R^{\XC\times \YC} \times \R^\XC\times \R^{\YCC},\\
	&\quad \quad (\hb^\XF , \hb_*^\YC) \mapsto  \begin{pmatrix}
		-\exp\left( \frac{1}{\lambda}\left[\Ab_*^T(\tilde\alphab, \tilde\betab_*) - \cb \right]\right)\odot (\tilde\rb \otimes \hb^\YC + \hb^\XF \otimes \tilde\sb)\\[0.1cm]
		- \begin{pmatrix}\hb^\XF\\
		\hb_*^\YC
		\end{pmatrix}
	\end{pmatrix},
		\end{align*}
		where $\hb^\YC \coloneqq (-\sum_{y \in \YCC} h_{*,y}^\YC, h_{*,y_2}^\YC, h_{*,y_3}^\YC, \dots)$. Likewise, the na\" ive partial derivative of $\FC$ with respect to $\pib, \alphab, \betab_*$ for the elements $(\tilde{\pib},\tilde\alphab, \tilde\betab_*,\tilde\rb, \tilde\sb_*)$ is equal to 
		\begin{align*}
			&\DCW_{\pib, \alphab, \betab_*|(\tilde\pib,\tilde\alphab,\tilde \betab_*, \tilde\rb, \tilde\sb_*) } \FC\colon  \; \ell^1(\XF\times \YC)\times \R^\XC\times \R^{\YCC} \rightarrow \R^{\XC\times \YC} \times \R^\XC\times \R^{\YCC},\\
	&\quad (\hb^{\XF \times \YC},\hb^{\XF, \infty}, \hb^{\YC, \infty}_*) \\
	&\quad \mapsto \begin{pmatrix}
			\hb^{\XF \times \YC}- \frac{1}{\lambda}\exp\left( \frac{1}{\lambda}\left[\Ab_*^T\big(\tilde\alphab, \tilde\betab_*\big) - \cb \right]\right)\odot (\tilde\rb\otimes \tilde\sb)\odot \Ab_*^T\big(\hb^{\XF, \infty} ,\hb^{\YC,\infty}_*\big)\\[0.1cm]
			\Ab_*(\hb^{\XF \times \YC})
		\end{pmatrix}.
		\end{align*} 
	Motivated by an implicit function approach a proposal for the  Hadamard derivative   of $\pib^\lambda$ at the pair of probability measures $(\rb, \sb)\in \probset{\XC}\times \probset{\YC}$ is %
	\begin{equation}\label{eq:ProposedDerivativeForPi}
	 - \left[\DCW_{\pib, \alphab, \betab_*|(\pib^\lambda, \alphab^\lambda, \betab^\lambda_*,  \rb, \sb_*)}\FC\right]^{-1}_{\pib} \circ \DCW_{\rb, \sb_*|(\pib^\lambda, \alphab^\lambda, \betab^\lambda_*, \rb, \sb_* ) }\FC.
	\end{equation}
	Much of our proof is devoted to making \eqref{eq:ProposedDerivativeForPi} mathematically precise. This requires assumptions on the cost function and the probability measures $\rb$ and $\sb$. 
	More precisely, for our subsequent differentiability result for the EROT plan we require that $\norm{\cxpb - \cxmb}_{\lInfX}< \infty$ which implies that $\cx \leq \kxSup{\delta} \leq \cx\norm{\exSup{\delta}}_{\ell^\infty(\XC)}$ for any $\delta\geq 0$. 
	 Further, for an explicit expression of the derivative for the EROT plan we introduce the operators 
	 \begin{align*}
		\ACX\colon& \lInfYYSub{\eySup{4}}\rightarrow \lInfX, && h^{\YC, \infty}_{\ast}\mapsto \bigg(\sum_{y\in\YCC}\bigg(\frac{\pi^\lambda_{xy}(\rb, \sb)}{r_x}\bigg)h^{\YC, \infty}_{\ast,y}\bigg)_{x\in \XC}, \\
		 \ACY\colon& \lInfX\rightarrow \lInfYYSub{\eySup{4}}&&h^{\XC, \infty}\mapsto \bigg(\sum_{x\in\XC}\bigg(\frac{\pi^\lambda_{xy}(\rb, \sb)}{s_y}\bigg) h^{\XC, \infty}_{x}\bigg)_{y\in \YCC},\\
		 \BCX\colon& \lYSub{\kySup{4}}\rightarrow \lInfX&& h^{\YC}\mapsto \bigg(\sum_{y\in\YCC}\bigg(\frac{\pi^\lambda_{xy}(\rb, \sb)}{r_xs_y}\bigg)h^{\YC}_{\ast,y}\bigg)_{x\in \XC},\\
		  \BCY\colon& \lXSub{\cx}\rightarrow \lInfYYSub{\eySup{4}}&&h^{\XC}\mapsto \bigg(\sum_{x\in\XC}\bigg(\frac{\pi^\lambda_{xy}(\rb, \sb)}{r_xs_y}\bigg)h^{\XC}_{x}\bigg)_{y\in \YCC}.
	 \end{align*}
	 Based on our quantitative bounds for the EROT plan from Proposition \ref{prop:BoundsOptimalPotentials} it follows that these operators are continuous. Upper bounds on their operator norms are obtained in the proof of the subsequent theorem. With these conventions we can state our main result on the Hadamard differentiability of the EROT plan. 

	\begin{theorem}[Hadamard differentiability of EROT plan]\label{them:TransportPlanIsHadamardDiff}
		Consider probability measures $\rb\in \PC(\XC), \sb\in \PC(\YC)$ which both have full support. 
		Assume that $\norm{\cxpb - \cxmb}_{\lInfX}< \infty$ and 
		$\cx\in\lr, \kySup{4}\in \ls$. 
		 Further, consider the EROT plan for $\lambda>0$  as a mapping $$ \pib^\lambda \colon \big(\probset{\XC}\cap \lXSub{\cx}\big)\times \big( \probset{\YC}\cap \lYSub{\kxSup{4}} \big)\rightarrow \lXYSub{\cx \oplus \cy}.$$ Then $\pib^\lambda$ is Hadamard differentiable at $(\rb, \sb)$ tangentially to $\big(\probset{\XC}\cap \lXSub{\cx}\big)\times \big( \probset{\YC}\cap \lYSub{\kxSup{4}} \big)$ with Hadamard derivative given by \begin{align*} \DH_{|( \rb,  \sb)}\pib^\lambda  \colon & T_{( \rb,  \sb)}\Big(\big(\probset{\XC}\cap \lXSub{\cx}\big)\times\big( \probset{\YC}\cap \lYSub{\kySup{4}} \big)\Big)  \rightarrow \lXYSub{\cx \oplus \cy},\\
		(\hb^\XC, \hb^\YC) &\mapsto \left(-\left[\DCW_{\pib, \alphab, \betab_*|(\pib^\lambda, \alphab^\lambda, \betab^\lambda_*,  \rb, \sb_* )}\FC\right]^{-1}_{\pib} \circ \DCW_{\rb, \sb_*|(\pib^\lambda, \alphab^\lambda, \betab^\lambda_*,  \rb, \sb_* ) }\FC\right) (\hb^\XC, \hb^\YC_*)\\
		&\!\!\!\!\!\!\!\!\!\!\!\!\!\!\!\!\!\!\!\!\!\!\!\!\!= \frac{\pib^\lambda}{ \rb \otimes  \sb }\odot \Big[   \rb \otimes \hb^\YC + \hb^\XF \otimes  \sb \Big]\\
		&\!\!\!\!\!\!\!\!\!\!\!\!\!\!\!\!\!\!\!\!\!\!\!\!\! -\pi^\lambda\odot A_*^T\begin{pmatrix}
			(\Id_{\XC} -\ACX\ACY)^{-1}&-(\Id_{\XC} -\ACX\ACY)^{-1}\ACX\\
			-\ACY(\Id_{\XC} -\ACX\ACY)^{-1}&(\Id_{\YCC} -\ACY\ACX)^{-1}\\
		\end{pmatrix}\begin{pmatrix}
			0 & \BCX\\
			\BCY&0
		\end{pmatrix}\begin{pmatrix}
			h^{\XC}\\
			h^{\YC}
		\end{pmatrix}.
		\end{align*} 
		The derivative is a well-defined and bounded operator and   the contingent cone at $( \rb,  \sb)$ with respect to $\lXSub{\cx}\times \lYSub{\kySup{4}}$ 
		is given by  $$\begin{aligned}
			&\;T_{( \rb,  \sb)}\Big(\big(\probset{\XC}\cap \lXSub{\cx}\big)\times\big( \probset{\YC}\cap \lYSub{\cy} \big)\Big) \\
			=& \;\left\{\hb^{\XC}\in \lXSub{\cx} \colon \sum_{x\in \XF} h^{\XC}_x = 0\right\} \times \left\{\hb^{\YC}\in \lYSub{\cy} \colon \sum_{y\in \YC} h^{\YC}_y = 0\right\}. 
		\end{aligned}$$ %
	\end{theorem}
	If there exists a collection of functions $\cxpb, \cxmb, \cypb,\cymb$ for  $c$ such that $\norm{\cxpb - \cxmb}_{\lInfX}< \infty$ and $\norm{\cypb - \cymb}_{\lInfX}< \infty$ it follows that $\cx \asymp \kxSup{\delta}$ and $\cy\asymp \kySup{\delta}$ for any $\delta\geq 0$. In this case, the norm for the Hadamard differentiability can be replaced by $\lXSub{\cx}\times \lYSub{\cy}$. 
	The proof of Theorem \ref{them:TransportPlanIsHadamardDiff} is deferred to Appendix \ref{subsec:ProofsSensitivityEROTP}, but we like to sketch its main arguments here. 
	
	\begin{proof}[Sketch of Proof]
		We first verify that the proposed operator for the derivative \eqref{eq:ProposedDerivativeForPi} is well-defined and bounded (Proposition \ref{prop:ProposedDerivativeIsBounded}). Herein, we require that  $\norm{\cxpb - \cxmb}_{\lInfX}< \infty$ to ensure the validity of the \emph{Neumann series calculus}.
		 We then proceed with the proof of Hadamard differentiability of the EROT plan.  
		By definition, this requires to take a sequence $(t_n)_{n \in \N}$ such that $t_n \searrow 0$ and a converging sequence $(\hb_n^\XF,\hb_{n}^\YC)_{n\in \N}\subseteq \lXSub{\cx} \times \lYSub{\kySup{4}}$ with limit $(\hb^\XF, \hb^\YC)$ and $(\rb + t_n \hb_n^\XF, \sb + t_n \hb^\YC_{n})\in \big(\probset{\XC}\cap \lXSub{\cx}\big)\times \big( \probset{\YC}\cap \lYSub{\kySup{4}} \big)$ for each $n\in \N$. 
		To this end, we show that the difference quotient of EROT plans for finitely supported perturbations can be approximated by the proposed derivative (Proposition \ref{prop:DerivativeFiniteSupportPerturbation}) and verify that the EROT plan is locally Lipschitz continuous (Proposition \ref{prop:DualOptimizerAreLipschitz}). 
		These results, in conjunction with a notion of finite support approximation (Appendix \ref{subsec:App:FSAProperties}), allow us to find for any $\epsilon>0$ an integer $N\in \N$ such that for all $n \geq N$ holds
				\begin{equation*}
			\norm{\frac{\pib^\lambda(\rb + t_n\hb^\XF_n,\sb + t_n\hb^\YC_{n}) - \pib(\rb,\sb)}{t_n} - \DH_{|( \rb,  \sb)}{\pib^\lambda}(\hb^\XF, \hb^\YC)}_\lXYSub{\cx\oplus \cy} < \epsilon. \notag \qedhere
		\end{equation*}
	\end{proof}

	\begin{remark}[On proof technique]\label{rem:GeneralizationPartiallyBoundedVariation} \label{rem:NoveltyOfProofAndDifficulty}\begin{enumerate}
		\item[$(i)$] Our proof for the sensitivity of the EROT plan does not rely on a standard implicit function theorem for Hadamard differentiable functions \cite[Proposition 4]{Roemisch04}. 
		The main issue in employing such a result lies in the selection of suitable normed spaces for the domain and range of $\mathcal{F}$. To ensure that the mapping $\mathcal{F}$ is well-defined the range space has to be chosen sufficiently large while at the same time the range space has to be sufficiently small such that the operator
		$[\DCW_{\pib, \alphab, \betab_*|(\pib^\lambda, \alphab^\lambda, \betab^\lambda_*,  \rb, \sb_*)}\FC]^{-1}$ is well-defined on a neighborhood around the origin in the range space. 
		As it turns out by Lemma \ref{lem:InverseOfDerivativeOfFIsWellDefined} has the operator $ [\DCW_{\pib, \alphab, \betab_*|(\pib^\lambda, \alphab^\lambda, \betab^\lambda_*,  \rb, \sb_*)}\FC]^{-1}$ only a fairly small domain (Remark \ref{rem:UnboundedDomainInverseOperator}). 
		Hence, to prove the claim on Hadamard differentiability we instead perform a careful analysis of the individual perturbation errors and show that they tend towards zero.
		\item[$(ii)$] Our assumption  $\smash{\norm{\cxp - \cxm}_{\lInfX}<\infty}$ asserts by Proposition \ref{prop:BoundsOptimalPotentials} a strictly positive lower bound for the quotient $\pi_{xy_1}^\lambda\!\!(\rb, \sb)/r_x$ uniformly in $x\in \supp(\rb)$ and may be interpreted as a curvature condition; these are commonly required in the context of $M$-estimators in order to control error bounds. In particular, the uniform lower bound implies that the operators $\ACX\ACY$ and $\ACY\ACX$ have operator norm strictly smaller than one, which asserts that $(\Id_\XC - \ACX\ACY)$ and $(\Id_{\YCC} - \ACY\ACX)$ are invertible. This property was imposed as a direct condition by \cite{Harchaoui2020} in their Assumption 2(3). It is also fulfilled for the setting of \cite{gonzalez2022weak} as they consider compactly supported probability measures in Euclidean spaces with a smooth cost function. It remains open if the analysis can be relaxed to account for general unbounded ground costs (Remark \ref{rem:IssuesUnboundedOperator}) on non-compact spaces.
		
	\end{enumerate}

	\end{remark}

	\begin{remark}[Non-vanishing derivative for EROT plan]\label{rem:NonDegDerivative}
		For any $\rb, \rb' \in \probset{\XC}\cap \lXSub{\cx}$, $\sb, \sb'\in \probset{\YC}\cap \lXSub{\cy}$ it holds that $$\norm{\pib^\lambda(\rb, \sb) - \pib^\lambda(\rb', \sb')}_{\lXYSub{\cx\oplus \cy}}\geq \max\left(\norm{\rb-\rb'}_{\lXSub{\cx}},\norm{\sb-\sb'}_{\lXSub{\cx}}\right).$$ 
		This implies under $\norm{\cxpb - \cxmb}_{\lInfX}< \infty$ and $\norm{\cypb - \cymb}_{\lInfY}< \infty$ for a sequence $(t_n)_{n \in \N}\subseteq (0,\infty)$ such that  $t_n \searrow 0$ and a converging sequence $(\hb_n^\XC,\hb_n^\YC)_{n \in \N}\in \lXSub{\cx}\times\lYSub{\cy}$ with limit $(\hb^\XC,\hb^\YC)$ and $(\rb + t_n \hb_n^\XC,\sb + t_n \hb_n^\YC)\in \left(\probset{\XC}\cap \lXSub{\cx}\right) \times \big(\probset{\YC}\cap \lYSub{\cy}\big)$ for all $n \in \N$ that $$ \begin{aligned}\!\!\norm{\DH_{|(\rb, \sb)}\pib^\lambda(\hb^\XC, \hb^\YC)}_{\lXYSub{\cx\oplus\cy\!}}\!\!\!\!\! = \lim_{n \rightarrow \infty} \norm{\frac{\pib^\lambda(\rb + t_n h_n^\XC, \sb+t_n h_n^\YC) - \pib^\lambda(\rb, \sb)}{t_n}}_{\lXYSub{\cx\oplus \cy}}\\
		 \geq \;\lim_{n \rightarrow \infty}\max\left(\norm{\hb^\XC_n}_{\lXSub{\cx}},\norm{\hb^\YC_n}_{\lXSub{\cy}} \right) =\max\left(\norm{\hb^\XC}_{\lXSub{\cx}}, \norm{\hb^\YC}_{\lYSub{\cy}}\right).
		 \end{aligned}$$ 
		This demonstrates for either $\rb$ or $\sb$ not supported on just a single point that $\DH_{|(\rb, \sb)}\pib^\lambda \neq 0$. 
		\end{remark}

\section{Discussion}\label{sec:Discussion}

For a general class of cost functions on countable spaces we have derived distributional limits for the empirical EROT value and plan. These results are in line with other statistical contributions on the EROT value \citep{bigot2019CentralLT,mena2019,del2022improved, goldfeld2022statistical} and the EROT plan \citep{klatt2018empirical,Harchaoui2020,gonzalez2022weak}, and extend the current state of research concerning empirical EROT.

It remains open to what extent the conditions on limit laws can be weakened, possibly to omit the exponential term in our summability conditions. 
In particular, it would be interesting to investigate the sensitivity of the EROT plan under generally unbounded ground costs. 
Notably, in our results for the EROT plan the condition $\norm{\cxpb - \cxmb}_{\lInfX}< \infty$ is of particular use to employ the Neumann series calculus of bounded operators. 
 In case of ground costs that are generally unbounded this proof technique does not generalize well due to a lack of a guarantee for an eigenvalue gap of the operators $\ACX\ACY$ and $\ACY\ACX$ (Remark \ref{rem:NoveltyOfProofAndDifficulty}). Therefore, weak limits for this setting remain unknown, although a similar structure is reasonable to conjecture. 

In addition to our limit results for fixed regularization parameter $\lambda >0$, we characterize in Section \ref{sec:RelationUnregularizedOT} the asymptotic behavior of the empirical EROT value and Sinkhorn costs for the regime of a  decreasing regularization parameter $\lambda(n) = o(1/\sqrt{n})$. 
We see that the resulting limit law is given by the respective limit law of the empirical \emph{unregularized} OT value which is fundamentally different when OT potentials are non-unique \citep{tameling18,hundrieser2022unifying}. 
In contrast, under bounded costs and unique OT potentials, our analysis shows that the empirical EROT value behaves in line with its unregularized counterpart, independent of the convergence speed of $\lambda(n)$ to zero. 

Naturally, in the regime of non-unique OT potentials the question arises for $\lambda(n)$ of slower order than $1/\sqrt{n}$ whether the empirical EROT value centered by its population version still converges weakly towards a suitable limit distribution and when a phase transition to the obtained normal limit occurs. 
Our bounds on EROT potentials (Proposition \ref{prop:BoundsOptimalPotentials}) combined with our stability result (Lemma \ref{lem:LipschitzEROTValue}) and arguments from empirical process theory \citep{mena2019} suggest at least for bounded costs and $\rb, \sb$ satisfying the Borisov-Dudley-Durst condition that  $\mathbb{E}\left[\left| \EROTLambda{\lambda}(\hat \rb_n, \hat \sb_n)-\EROTLambda{\lambda}(\rb, \sb)\right|\right]=\landau(\norm{c}_\infty/\sqrt{n})$, independent of the regularization parameter $\lambda$.
By Markov's inequality this causes for $\lambda(n)\rightarrow 0$ the sequence $\sqrt{n}\big(\EROTLambda{\lambda(n)}(\hat \rb_n, \hat \sb_n)-\EROTLambda{\lambda(n)}(\rb, \sb)\big)$ to be tight, asserting with Prokhorov's Theorem the existence of a weakly converging subsequence.

The analysis of the limit behavior of the empirical EROT \emph{plan} on countable spaces for the regime $\lambda\searrow 0$ is even more involved. Mimicking Section \ref{sec:RelationUnregularizedOT} two aspects appears crucial to us. First, it is necessary to obtain suitable bounds between the (empirical) EROT plan and a suitable unregularized (empirical) OT plan. Although it is known under mild conditions that the EROT plan $\pi^\lambda(\rb, \sb)$ converges for $\lambda\rightarrow 0$  in $\lXY$ to the OT plan $\pi^*\in \Pi(\rb, \sb)$ with minimal entropy \citep[Theorem 5.5]{nutz2021introduction}, quantitative bounds for this convergence are only available on finite spaces \citep{Weed18_ExplicitAnalysis}. %
Second, the limit distribution of the empirical unregularized OT plan on countable spaces has to be characterized. For finitely supported probability measures  with a unique unregularized OT plan \cite{klatt2020limit} recently obtained such a limit law explicitly relying on finite-dimensional linear programming. 
However, this approach does not generalize well beyond finite settings and hence similar statements for countable spaces remain open.

 \section*{Acknowledgements}
 S. Hundrieser and A. Munk acknowledge funding by the Deutsche Forschungsgemeinschaft (DFG, German Research
Foundation) under Germany's Excellence Strategy - EXC 2067/1- 390729940. Further, S. Hundrieser and M. Klatt acknowledge funding from the DFG Research Training Group 2088 \emph{Discovering structure in complex data: Statistics meets Optimization and Inverse Problems}.

\newpage

\begin{appendix}

	\newpage
	\section{Proofs for Section \ref{sec:Preliminaries}}\label{app:ProofsPreliminaries}
	\begin{proof}[Proof of Proposition \ref{prop:BoundsOptimalPotentials}]
		The proof is based on techniques by \cite{mena2019}. By relation \eqref{eq:optimalityCriterion_1_connectionSolutions} between primal and dual optimizers for \eqref{eq:EntropicOptimalTransport} it follows that any pair of EROT potentials $(\alphab^\lambda, \betab^\lambda)\in \lr\times \ls$ satisfies $\EROT(\rb, \sb) = \langle \alphab^\lambda, \rb\rangle + \langle \betab^\lambda, \sb\rangle.$ 
		Since $\cx\in \lr$ and $\cy\in \ls$,  there exists $(\rb,\sb)$-a.s. unique  EROT potentials up to a constant shift (Proposition \ref{prop:DualEROT}), allowing us to select potentials such that  $$\langle\alphab^\lambda,\rb\rangle = \langle\betab^\lambda,\sb\rangle = \EROT(\rb, \sb)/2 \geq (\langle \cxm, \rb\rangle + \langle \cym, \rb\rangle)/2.$$ Optimality of $(\alphab^\lambda, \betab^\lambda)$ implies by Proposition \ref{prop:DualEROT} for all $x \in \supp(\rb)$ and $y \in\supp(\sb)$ that 
		\begin{align*}
			\alpha_x^\lambda = - \lambda\log\left[\sum_{y \in \YC}\exp\bigg(\frac{ \beta_y^\lambda - c(x,y)}{\lambda}\bigg)s_y\right], &&
			 \beta_y^\lambda = - \lambda\log\left[\sum_{x \in \XF}\exp\bigg(\frac{ \alpha_x^\lambda - c(x,y)}{\lambda}\bigg)r_x\right] .
		\end{align*}
		As noted in Remark \ref{rem:NonConstantDualPotentials} we may extend the potentials to the whole spaces $\XC$ and $\YC$ according to the respective right-hand sides. 
		Applying Jensen's inequality for the convex function $-\log(\,\cdot \,) $ and by our choice of $(\alphab^\lambda,\betab^\lambda)$ it then follows for each $x\in \XC$ that 
		\begin{align*}
			\alpha_x^\lambda &= -\lambda\log\bigg[\sum_{y \in \YC}\exp\bigg(\frac{ \beta_y^\lambda - c(x,y)}{\lambda}\bigg)s_y\bigg] \\
			&\leq  \sum_{y \in \YC}c(x,y)s_y -\langle \betab^\lambda, \sb \rangle\\
			&\leq \cxp(x) + \langle \cypb, \sb\rangle - (\langle \cxm, \rb\rangle + \langle \cym, \rb\rangle)/2. 
		\end{align*}
		Likewise, it follows for all $y\in \YC$ that $$\beta_y^\lambda \leq \cyp(y) + \langle \cxp, \rb\rangle - (\langle \cxm, \rb\rangle + \langle \cym, \rb\rangle)/2.$$
				For the lower bound of $\alpha_x^{\lambda}$, take the upper bound of $\beta^{\lambda}_y$ and the lower bound on $c$ to infer 
			\begin{align*}
				\alpha_x^\lambda &= -\lambda\log\bigg(\sum_{y \in \YC}\exp\bigg(\frac{\beta^\lambda_y - c(x,y)}{\lambda}\bigg)s_y\bigg) \\
				&\geq -\lambda\log\bigg(\sum_{y \in \YC}\exp\bigg(\frac{\cyp(y) + \langle \cxpb, \rb\rangle -(\langle \cxm, \rb\rangle + \langle \cym, \rb\rangle)/2 - \cxm(x) - \cym(y)}{\lambda}\bigg)s_y\bigg) \\
				&\geq \cxm(x)- \langle \cxpb, \rb\rangle + (\langle \cxm, \rb\rangle + \langle \cym, \rb\rangle)/2 - \lambda \log\langle\eySup{},\sb\rangle.
			\end{align*}
		The lower bound for $\betab^{\lambda}$ follows analogously. 
		The bounds for the EROT plan follows from the correspondence to the EROT potentials (Proposition \ref{prop:DualEROT}) and the above derived bounds. More precisely, for the upper bound it holds for any $(x,y)\in \XC\times \YC$ that
		\begin{align*}
			\pi^{\lambda}_{xy} &= \exp\left(\frac{\alpha^\lambda_x + \beta^\lambda_y -c(x,y)}{\lambda}\right)r_xs_y \\
			&\leq \exSup{}(x)\eySup{}(y) \exp\left(\frac{ \langle \cxpb - \cxmb, \rb\rangle+ \langle \cypb - \cymb, \sb\rangle}{\lambda}\right)r_xs_y\\
			&\leq \exSup{}(x)\eySup{}(y)\langle \exSup{},\rb\rangle \langle \eySup{},\sb\rangle.
		\end{align*}
	The lower bound for the EROT plan follow by invoking the lower bounds on the potentials in combination with the lower bound on the cost function.
	\end{proof}
	
	\section{Proofs for Section \ref{sec:LimitDistributions}}
	\label{app:ProofSectionExamples}

	\begin{proof}[Proof of Lemma \ref{lem:WeakConvergence}]
		First note that assertion $(iii)$ follows from the fact that subsets and unions of Donsker classes are again Donsker \citep[Theorem 2.10.1 and Example 2.10.7]{van1996weak}.
	
		For the proof of $(i)$ we consider the partition $\XC = \bigcup_{n \in \NN}\{x_n\}$ and note that the restricted function class $\FC_1|_{\{x_n\}}$ fulfills 
		\begin{align*}
			\mathbb{E}\bigg[\sup_{\tilde f\in \FC_1|_{\{x_n\}}}\!\!\!\!\!\!\!\!\langle \tilde f, \sqrt{n}(\hat \rb_ n -\rb)\rangle\bigg]&= \EV{\Big|\Big\langle f(x_n)\Indicator{x_n}, \sqrt{n}(\hat \rb_ n -\rb)\Big\rangle\Big|}= f(x_n) \EV{\sqrt{n}|\hat \rb_{n,x} - \rb_{x}|}  \\
			&\leq f(x_n) \EV{n|\hat \rb_{n,x} - \rb_{x}|^2}^{1/2} \leq f(x_n)\sqrt{r_x(1-r_x)} \leq f(x)\sqrt{r_x}. 
		\end{align*}
		Hence, by \citet[Theorem 2.10.24]{van1996weak}, $\FC_1$ is $\rb$-Donsker if the weighted Borisov-Dudley-Durst condition $\sum_{x\in \XC}f(x) \sqrt{r_x}<\infty$ is met; since $r_x\leq \sqrt{r_x}$ this also implies that $f\in \lr$. The converse follows by \citet[Lemma~2.2]{van1996new}.
	
		For $(ii)$ note that the class $\tilde \FC_2\coloneqq \{f\}$ is $\rb$-Donsker if and only if $\sum_{x\in \XC}(f(x))^2r_x<\infty$. Both function classes $\tilde \FC_2$ and $\overline \FC_2 \coloneqq \{\Indicator{\XC}, \Indicator{\XC\backslash\{x_1, \dots, x_n\}} \colon n \in \NN\}$ have bounded uniform entropy integrals, since $\tilde \FC_2$ is a singleton and $\overline \FC_2$ is a VC-class of index $2$ \citep[Chapter 9.1]{kosorok2008introduction}. Hence, by \citet[Theorem 9.15]{kosorok2008introduction}, their element-wise product $\FC_2 = \tilde \FC_2\cdot \overline \FC_2$ also has bounded uniform entropy integral and admits envelope function $f$. Thus, by \citet[Theorem 2.5.2]{van1996weak} it is also $\rb$-Donsker if $\sum_{x\in \XC}(f(x))^2r_x<\infty$. For the converse, recall that $\sum_{x\in \XC}(f(x))^2r_x<\infty$ is necessary for $\tilde \FC_2$ to be $\rb$-Donsker; this implies that it is also necessary for $\FC_2$ to be $\rb$-Donsker. 
	\end{proof}

	\begin{proof}[Proof of Remark \ref{rem:ComparisonEROTplan}(ii)]%
	\label{subsec:ConsistencyOtherContributions}
	
	We show the evaluation of suitably bounded functions with respect to the empirical EROT plan yields a distributional limit which is consistent with results by  \cite{gonzalez2022weak} and thus also \cite{Harchaoui2020}. 	%
	
	Assume for dominating functions $\cxp, \cxm, \cym, \cyp$ for $c$ that $\norm{\cxp - \cxm}_{\lInfX}<\infty$ and consider the two-sample case from Theorem \ref{them:LimitLawEmpEROT_Plan}; the arguments for the one-sample case are analogous. Then, under suitable weighted Borisov-Dudley-Durst summability constraints and the explicit representation from Theorem \ref{them:TransportPlanIsHadamardDiff} it follows for a function $f\colon \XC\times \YC\rightarrow \RR$ with $\norm{f}_{\lInfXYSub{\cx\oplus\cy}}<\infty$ by the continuous mapping theorem \citep[Theorem 1.3.9]{van1996weak}  for $\min(n,m)\rightarrow \infty$ with $\frac{m}{n+m}\rightarrow \delta\in (0,1)$  that
		\begin{align*}
			&\sqrt{\frac{nm}{n+m}}\left(\langle f,\pib^\lambda(\hat \rb_n, \hat \sb_n)\rangle  - \langle f,\pib^\lambda(\rb, \sb)\rangle\right) \konvD  \left\langle f, \frac{\pib^\lambda}{ \rb \otimes  \sb }\odot \Big[   \sqrt{\delta}\rb \otimes \Gb_\sb + \sqrt{1-\delta}\Gb_\rb \otimes  \sb \Big]\right\rangle\\
			&+\left\langle f, \pi^\lambda\odot A_*^T\begin{pmatrix}-
				(\Id_{\XC} -\ACX\ACY)^{-1}&(\Id_{\XC} -\ACX\ACY)^{-1}\ACX\\
				\ACY(\Id_{\XC} -\ACX\ACY)^{-1}&-(\Id_{\YCC} -\ACY\ACX)^{-1}\\
			\end{pmatrix}\begin{pmatrix}
				\sqrt{1-\delta}\BCX\Gb_\sb\\
				\sqrt{\delta}\BCY\Gb_\rb
			\end{pmatrix}\right\rangle.
		\end{align*}
		We now decompose the limit into three components, $A_\infty, B_\infty, C_\infty$ given by, 
		\begin{align*}
			A_\infty\coloneqq\left\langle f, \pi^\lambda\odot A_*^T\begin{pmatrix}-
				(\Id_{\XC} -\ACX\ACY)^{-1}&(\Id_{\XC} -\ACX\ACY)^{-1}\ACX\\
				\ACY(\Id_{\XC} -\ACX\ACY)^{-1}&-(\Id_{\YCC} -\ACY\ACX)^{-1}\\
			\end{pmatrix}\begin{pmatrix}
				\sqrt{1-\delta}\BCX\Gb_\sb\\
				\sqrt{\delta}\BCY\Gb_\rb
			\end{pmatrix}\right\rangle,\\
			B_\infty\coloneqq\left\langle f, \frac{\pib^\lambda}{ \rb \otimes  \sb }\odot \Big[   \sqrt{1-\delta}\Gb_\rb \otimes \sb\Big]\right\rangle,\quad \quad 
			C_\infty\coloneqq \left\langle f, \frac{\pib^\lambda}{ \rb \otimes  \sb }\odot \Big[   \sqrt{\delta}\rb \otimes \Gb_\sb \Big]\right\rangle.\quad \quad \quad \quad 
		\end{align*}
		Upon defining the empirical processes $\Gb_\rb^{n,m} = \sqrt{\frac{nm}{n+m}}(\hat \rb_n - \rb)$ and $\Gb_\sb^{n,m} = \sqrt{\frac{nm}{n+m}}(\hat \sb_m - \sb)$ \quad  we introduce the empirical quantities
		\begin{align*}
			A_n\coloneqq\left\langle f, \pi^\lambda\odot A_*^T\begin{pmatrix}-
				(\Id_{\XC} -\ACX\ACY)^{-1}&(\Id_{\XC} -\ACX\ACY)^{-1}\ACX\\
				\ACY(\Id_{\XC} -\ACX\ACY)^{-1}&-(\Id_{\YCC} -\ACY\ACX)^{-1}\\
			\end{pmatrix}\begin{pmatrix}
				\BCX\Gb_\sb^{n,m}\\
				\BCY\Gb_\rb^{n,m}
			\end{pmatrix}\right\rangle,\\
			B_n\coloneqq\left\langle f, \frac{\pib^\lambda}{ \rb \otimes  \sb }\odot \Big[   \Gb_\rb^{n,m} \otimes \sb\Big]\right\rangle,\quad \quad 
			C_n\coloneqq \left\langle f, \frac{\pib^\lambda}{ \rb \otimes  \sb }\odot \Big[ \rb \otimes \Gb_\sb^{n,m}\Big]\right\rangle.\quad \quad \quad \quad 
		\end{align*}
		In particular, the terms $A_n, B_n, C_n$ all converge for $\min(n,m)\rightarrow \infty$ with $\frac{m}{n+m}\rightarrow \delta\in (0,1)$  in distribution to $A_\infty, B_\infty, C_\infty$, respectively. Herein, $B_n$ and $C_n$ are exactly the same quantities as the terms $B_n$ and $C_n$ from \citet[p. 15]{gonzalez2022weak} with $P$ representing $\rb$ and $Q$ representing $\sb$. 
		Also note that the random variables $(\BCX\Gb_\sb^{n,m},\BCY\Gb_\rb^{n,m})$ exactly coincide with $\mathbb{G}^{n}_{P,s}$ and $\mathbb{G}^{m}_{Q,s}$ in their notation. 
		Hence, the expression in $A_n$ matches the integral expression from Equation (25) in that reference. 
		
		Since the distributional limit of \cite{gonzalez2022weak} decomposes into these three terms, we conclude that our distributional limit is consistent  with the limit law by \cite{gonzalez2022weak}. In particular, their limits are in line with results by \cite{Harchaoui2020}, which also asserts that our results coincide with their limit. 
	\end{proof}
	
	\begin{proof}[Proof of Theorem \ref{them:BootstrapEROT}]
		First recall the function classes  $\FC_\XC, \FC_\YC$ from \eqref{eq:FunctionClassFC} and $\HC_\XC, \HC_\YC$ from \eqref{eq:FunctionClassHC}. 
		For the first assertion note by the summability constraint from Theorem \ref{them:LimitLawEmpEROT_Value}(i) using Lemma \ref{lem:WeakConvergence} that $\FC_\XC$ is $\rb$-Donsker. Hence, we infer for the bootstrap empirical process from \citet[Theorem 3.6.13]{van1996weak} that  
		\begin{equation}\label{eq:ConditionsEmpiricalBootstrapProcess}
			\sup_{h \in \BL{\ell^\infty(\FC_\XC)} }\left| \EV{h\big(\sqrt{n}(\hat \rb_n^* - \hat \rb_n)\big) \Big| X_1, \dots, X_n } -  \EV{h\big(\sqrt{n}(\hat \rb_n - \rb)\big)} \right|\konvP 0.
		  \end{equation}
		  Further, note by assumption that $\sb\in \ell^\infty(\FC_\YC)$. Hence, since any bounded Lipschitz function on $\ell^\infty(\FC_\XC)$ whose modulus is bounded by one can be extended to a bounded Lipschitz function on $\ell^\infty(\FC_\XC)\times \ell^\infty(\FC_\YC)$ \cite[Theorem 1]{mcshane34} we conclude for $B = \ell^\infty(\FC_\XC)\times \ell^\infty(\FC_\YC)$ that
		  \begin{align*}%
			\sup_{h \in \BL{B} }\Big| \EV{h\big(\sqrt{n}((\hat \rb_n^*,\sb) - (\hat \rb_n, \sb))\big) \Big| X_1, \dots, X_n }
			-  \EV{h\big(\sqrt{n}((\hat \rb_n, \sb) - (\rb, \sb))\big)} \Big|&\konvP 0.
		  \end{align*}
		  Since the EROT value is Hadamard differentiable at $(\rb, \sb)$ tangentially to $\big(\probset{\XC}\cap \ell^\infty(\FC_\XC)\big)\times \big( \probset{\YC}\cap \ell^\infty(\FC_\YC)\big)$ for $\FC_\XC, \FC_\YC$ defined in \eqref{eq:FunctionClassFC}, the assertion follows from the functional delta method for the bootstrap \citep[Theorem 3.9.11]{van1996weak}.
		  
		  For the second setting note that the summability constraint from Theorem \ref{them:LimitLawEmpEROT_Plan}(i) guarantees Lemma \ref{lem:WeakConvergence} that $\HC_\XC$ is $\rb$-Donsker, while $\sb \in \ell^\infty(\HC_\YC)$. Invoking a similar extension argument of bounded Lipschitz functions as above, we obtain for $B = \lXSub{\cx}\times \lYSub{\cy\eySup{4}}$ that 
		  \begin{align*}%
			\sup_{h \in \BL{B} }\Big| \EV{h\big(\sqrt{n}((\hat \rb_n^*,\sb) - (\hat \rb_n, \sb))\big) \Big| X_1, \dots, X_n }
			-  \EV{h\big(\sqrt{n}((\hat \rb_n, \sb) - (\rb, \sb))\big)} \Big|&\konvP 0.
		  \end{align*}
		  The functional delta method for the bootstrap and Hadamard differentiability of the EROT plan at $(\rb, \sb)$ tangentially to $\big(\probset{\XC}\cap \lXSub{\cx}\big)\times \big( \probset{\YC}\cap \lYSub{\cy\eySup{4}}\big)$ yield the second claim.
	\end{proof}
	
	\section{Proofs for Section \ref{sec:Examples}}\label{app:ProofsExamples}
	\begin{proof}[Proof of Proposition \ref{prop:SemiBounded_Sufficient}]
		First note that $\eySup{2}\in \ls$ since $$\langle \eySup{2}, \sb\rangle \lesssim \sum_{y\in \YC}\exp\left( \frac{\overline \gamma}{5} d^{p-1+\overline \epsilon}(y,z)\right) \!s_y\leq \sum_{y\in \YC}\exp\left( \overline \gamma d^{p-1+\overline \epsilon}(y,z)\right) \!s_y<\infty.$$
	Moreover, note by the growth condition on $f_\YC$ that 
	\begin{align}
	&\sum_{y\in \YC} (1+d^p(y,z))\exp\left(-\frac{\overline \gamma}{10} d^{p-1+ \epsilon}(y,z)\right)\label{eq:NiceSumForNicerConditions}\\
	\leq &\sum_{n = 1}^\infty f_\YC(n)(1+n^p)\exp\left(-\frac{\overline \gamma}{10}(n-1)^{p-1+ \epsilon}\right) \notag\\
	\lesssim &\sum_{n = 1}^\infty (1+n^p)\exp\left(\delta n^{p-1+\epsilon/2}  -\frac{\overline \gamma}{10}(n-1)^{p-1+ \epsilon}\right)<\infty. \notag
	\end{align}
	Further, condition \eqref{eq:MeasureConditionForSimplicationOfAssumptions} implies $\sqrt{s_y}\lesssim \exp(- \frac{\overline \gamma}{2} d^{p-1+\overline \epsilon}(y,z))\leq \exp(- \frac{\overline \gamma}{10} d^{p-1+\overline \epsilon}(y,z))$. Hence, from \eqref{eq:NiceSumForNicerConditions} we infer that
		$$\sum_{y\in\YC}\cy(y)\sqrt{s_y}  \lesssim\sum_{y\in\YC}(1+d^p(y,z)) \exp\left(- \frac{\overline \gamma}{10} d^{p-1+\overline \epsilon}(y,z)\right) <\infty, $$
		 and conclude that \eqref{eq:EROT_Value_BoundedUnbounded} is fulfilled. Further, we find that \eqref{eq:EROT_Plan_BoundedUnbounded} is satisfied since
		\begin{equation*}
			\sum_{y\in\YC}\cy(y)\eySup{4}(y)\sqrt{s_y} \lesssim \sum_{y\in\YC}(1+d^p(y,z)) \exp\left( \left(\frac{4}{10} - \frac{1}{2}\right)\overline \gamma d^{p-1+\overline \epsilon}(y,z)\right)<\infty.  \qedhere
		\end{equation*}
	\end{proof}
	
	\begin{proof}[Proof of Proposition \ref{prop:EROT_EuclideanNorm}]
		All assertions follow via a suitable selection of dominating functions to be detailed below by invoking Theorems \ref{them:LimitLawEmpEROT_Value}(ii) for the empirical EROT value and \ref{them:LimitLawEmpEROT_Plan}(iii) for the empirical EROT plan.  
	
		For assertion $(i)$ we consider the dominating functions according to Section \ref{subsec:Expl:Bounded}, which results in standard Borisov-Dudley-Durst conditions for the measures. 
		To show assertion $(ii)$, we take the dominating functions from \eqref{eq:choicesDominatingFunctionsSemiBounded} for $z = 0$ and employ Proposition \ref{prop:SemiBounded_Sufficient} to simplify the summability constraint on $\sb$ to guarantee distributional limits for the EROT value and plan. For assertion $(iii)$, we choose the dominating function according to \eqref{eq:choicesDominatingFunctionsUnbounded1} and \eqref{eq:choicesDominatingFunctionsUnbounded2} for $z = 0$, and employ Proposition \ref{prop:SemiBounded_Sufficient} to simplify the summability constraints for the distributional limits of the EROT value. Finally, for assertion $(iv)$ we note that the separability condition \eqref{eq:ConditionCostfunctionWithMetric} is met since 
		for any $(x,y)\in \XC\times \YC\subseteq (-\infty, a]\times [b, \infty)$ it holds that
		\begin{align*}
			\norm{x-y}_1 =& \norm{x - a + a-b + b -y}_1 \geq \norm{x - a+b -y}_1 - \norm{ a-b }_1 \\
			=& \norm{x - a}_1 + \norm{b -y}_1 - \norm{ a-b }_1 \geq \norm{x}_1 + \norm{y}_1 - \norm{a}_1 - \norm{b}_1 - \norm{ a-b }_1.
		\end{align*}
		Hence, using the dominating functions from \eqref{eq:SeparabilityDominatingFunctions} with $z = 0$ it follows that the summability constraints are given by $\sum_{x\in \XC}(1+ \norm{x}_1)\sqrt{r_x} <\infty$ and $\sum_{y\in \YC} (1+\norm{y}_1)\sqrt{s_y} <\infty$. 
	\end{proof}

	\section{Proofs for Section \ref{sec:SensitivityAnalysis}}\label{sec:ProofsSensitivityAnalysis}

	\subsection{Continuity of Entropic Optimal Transport Quantities}\label{subsec:ProofContinuity}
	
	\begin{proof}[Proof of Lemma \ref{lem:LipschitzEROTValue}]
		The proof of this lemma is strongly inspired by \citet[Proposition 2]{mena2019}. 	
		First note by assumptions on $\rb, \tilde \rb, \sb$ from Proposition \ref{prop:DualEROT} that the EROT potentials $(\alpha^\lambda, \betab^\lambda)$, $(\tilde \alpha^\lambda, \tilde \betab^\lambda)$ are contained in $(\lr\cap \ell^1_{\tilde \rb}(\XC))\times \ls$.  
		Moreover, based on the dual representation of the EROT value as a supremum it follows by plugging in $(\alpha^\lambda, \betab^\lambda)$ for the EROT problem with $(\tilde \rb, \sb)$ that 
		\begin{align*}
			& \EROT(\rb, \sb) - \EROT(\tilde \rb, \sb) \\
			\leq \; & \langle \alpha^\lambda, \rb\rangle +   \langle \betab^\lambda, \sb\rangle - \lambda \bigg[ \sum_{\substack{x \in \XF \\ y \in \YC}}{{\exp\left( \frac{\alpha_x^\lambda + \beta_y^\lambda - c(x,y)}{\lambda} \right){r}_x {s}_y - {r}_x{s}_y} } \bigg]\\
			-&  \langle \alpha^\lambda, \tilde \rb\rangle -  \langle \betab^\lambda, \sb\rangle - \lambda \bigg[ \sum_{\substack{x \in \XF \\ y \in \YC}}{{\exp\left( \frac{\alpha_x^\lambda + \beta_y^\lambda - c(x,y)}{\lambda} \right)\tilde{r}_x {s}_y - \tilde {r}_x{s}_y} } \bigg]\\
			= \; & \langle \alphab^\lambda, \rb - \tilde \rb\rangle -  \lambda \bigg[ \sum_{\substack{x \in \XF \\ y \in \YC}}{{\exp\left( \frac{\alpha_x^\lambda + \beta_y^\lambda - c(x,y)}{\lambda} \right)({r}_x - \tilde r_x) {s}_y} } \bigg].
		\end{align*}
		Due to the convention of Remark \ref{rem:NonConstantDualPotentials} we obtain for any $x\in \XC$ that
		\begin{align*}
			& \sum_{\substack{y\in\YC}}{{\exp\left( \frac{\alpha_x^\lambda + \beta_y^\lambda - c(x,y)}{\lambda} \right)({r}_x - \tilde r_x) {s}_y} }  \\
			= \;& ({r}_x - \tilde r_x) \exp\bigg(\frac{ \alpha^\lambda_{x}}{\lambda}\bigg)\sum_{\substack{y\in\YC}}\exp\bigg(\frac{ \beta^\lambda_{y} - c(x,y)}{\lambda}\bigg)\sb_y = ({r}_x - \tilde r_x).%
		\end{align*}
		Hence, by Tonelli's theorem we infer that 
		\begin{align*}
			\lambda \bigg[ \sum_{\substack{x \in \XF \\ y \in \YC}}{{\exp\left( \frac{\alpha_x^\lambda + \beta_y^\lambda - c(x,y)}{\lambda} \right)({r}_x - \tilde r_x) {s}_y} } \bigg] = \lambda \sum_{x\in \XC}(r_x - \tilde r_x) = 0, 
		\end{align*}
		which asserts the upper bound of the first assertion. The lower bound follows analogously by repeating the argument for $(\tilde \alpha^\lambda, \tilde \betab^\lambda)$.
	
		For the second assertion we employ Proposition \ref{prop:BoundsOptimalPotentials} to infer for any $x\in \XC$ that 
		\begin{align*}
			 |\alpha_x^\lambda|&\leq |\cxpb(x)| + |\cxmb(x)| + \langle \cx, \rb\rangle + \langle \cy, \sb\rangle + \lambda \log \langle \eySup{}, \sb \rangle\\
			 & \leq \cx(x)\left( \langle \cx, \rb\rangle + \langle \cy, \sb\rangle + \lambda \log \langle \eySup{}, \sb \rangle \right).
		\end{align*}
		Consequently,  we obtain by H\"older's inequality that 
		\begin{align*}
			\langle \alphab^\lambda, \rb - \tilde \rb\rangle \leq \left(\langle \cx, \rb\rangle + \langle \cy, \sb\rangle + \lambda\log\langle \eySup{}, \sb\rangle \right)\norm{\rb - \tilde \rb}_{\lXSub{\cx}}.
		\end{align*}
		In conjunction with the analogous bound for $	\langle \alphab^\lambda, \rb - \tilde \rb\rangle$ the second assertion follows. 
	\end{proof}
	
	\begin{proof}[Proof of Proposition \ref{prop:ConvergenceOptimalDualSolutions}]	
			
		The proof is divided into three steps, where we show each respective assertion. The continuity of the EROT value is based on Lemma \ref{lem:LipschitzEROTValue}. For the pointwise convergence of EROT potentials we follow an approach by \cite{mena2019}, who were inspired by \cite{Feydy2019Interpolating}, and afterwards exploit for the convergence of the EROT plan the relation between primal and dual optimizers (Proposition \ref{prop:DualEROT}). 
	
		\emph{Step 1 - Continuity of EROT value.} By assumption of the convergence it follows that \begin{equation*}
			K_{\XC} \coloneqq \sup_{\ind \in \N}\norm{\rb_\ind}_{\ell^\infty(\FC_\XC)}  < \infty \quad \text{ and } \quad  K_{\YC} \coloneqq \sup_{\ind \in \N}\norm{\sb_\ind}_{\ell^\infty(\FC_\YC)} < \infty.
		  \end{equation*}	
		Moreover, invoking Lemma \ref{lem:LipschitzEROTValue} in conjunction with triangle inequality asserts that 
		\begin{align*}
			& \left|\EROT(\rb_k, \sb_k) - \EROT(\rb, \sb)\right|\\
			 \leq \;& \left|\EROT(\rb_k, \sb_k) - \EROT(\rb_k, \sb)\right| + \left|\EROT(\rb_k, \sb) - \EROT(\rb, \sb)\right|\\
			 \leq \; & \left(\langle \cy, \sb_k+\sb\rangle +2\langle \cx, \rb_k\rangle + 2\lambda\log\langle \exSup{}, \rb_k\rangle \right)\norm{\sb_k - \sb}_{\lYSub{\cy}}\\
			 &+  \left(\langle \cx, \rb_k+ \rb\rangle +2\langle \cy, \sb\rangle + 2\lambda\log\langle \eySup{}, \sb\rangle \right)\norm{\rb_k - \rb}_{\lXSub{\cx}}\\
			 \leq \;&\left(2K_{\YC} +2(1+\lambda)K_{\XC} \right)\norm{\sb_k - \sb}_{\ell^\infty(\FC_\YC)}+  \left(2K_{\XC} +2(1+\lambda)K_{\YC} \right)\norm{\rb_k - \rb}_{\ell^\infty(\FC_\XC)}.
		\end{align*}
		Hence, by assumption on the convergence of $(\rb_\ind, \sb_\ind)_{\ind \in \NN}$ to $(\rb, \sb)$ it follows that the right-hand side tends to zero and assertion $(i)$ follows.

		\emph{Step 2 - Pointwise continuity of EROT potentials.} %
		Denote by $(\tilde \alphab^{\lambda}_\ind, \tilde \betab^\lambda_\ind)_{\ind \in \N}$ the EROT potentials for probability measures $(\rb_\ind, \sb_\ind)_{\ind\in \N}$ which satisfy $\langle \tilde\alphab^{\lambda}_\ind, \rb_\ind \rangle = \langle \tilde \betab^\lambda_\ind, \sb_\ind \rangle$ and are defined according to the convention in Remark \ref{rem:NonConstantDualPotentials} outside the support of the underlying measures. By Proposition \ref{prop:BoundsOptimalPotentials} we then infer for all $\ind\in \N$ and $x\in\XC, y\in \YC$ that
		 \begin{align*} \cxm(x) - K_{\XC} - (1+\lambda) K_{\YC}  &\leq \tilde\alphab^{\lambda}_{\ind,x}\leq \cxp(x)  + K_{\YC},\\
		\tcym(y) - (1+\lambda)K_{\XC} - K_{\YC}  &\leq \tilde\beta^{\lambda}_{\ind,y}\leq \tcyp(y) + K_{\XC}.
	 \end{align*}
	 Hence, the quantity $K_\beta\coloneqq \sup_{\ind \in \N}|\tilde\beta^\lambda_{n, y_1}|< \infty$ is finite. Further, note that the pair $(\tilde\alphab^{\lambda}_{\ind} + \tilde\beta^\lambda_{\ind,y_1}, \tilde\betab^{\lambda}_{\ind} -\tilde\beta^\lambda_{\ind,y_1})$ 
	 exactly coincides on $\XC$ and $\YC$ with the EROT potentials $(\alphab^{\lambda}_{\ind}, \betab^{\lambda}_{\ind})$ from the assertion. This yields for all 
	 $\ind \in \N$ and $x\in\XC, y\in \YC$ the following bounds
	  \begin{align} \begin{aligned}\cxm(x) - K_{\XC} - (1+\lambda)K_{\YC} -K_\beta &\leq \alpha^{\lambda}_{\ind,x}\leq \cxp(x) + K_{\YC} +K_\beta,\\
		\tcym(y) -  (1+\lambda)K_{\XC} - K_{\YC} -K_\beta &\leq \beta^{\lambda}_{\ind,y}\,\leq \tcyp(y) + K_{\XC} +K_\beta.\end{aligned} \label{eq:LowerUpperBounds}
	 \end{align}
	 Invoking a diagonalization argument asserts existence of a subsequence $(\alphab_{\ind_m}^\lambda, \betab_{\ind_m}^\lambda)_{m \in \N}$ converging pointwise for each $x\in\XC$ and $y\in\YC$ to a limit $(\alphab_\infty^\lambda, \betab_\infty^\lambda)\in \R^{\XC}\times\R^\YC$. It remains to show that $(\alphab_\infty^\lambda, \betab_\infty^\lambda)$ is a pair of EROT potentials for probability measures $\rb$ and $\sb$ that follows the convention from Remark \ref{rem:NonConstantDualPotentials} and satisfies $\beta^\lambda_{\infty,y_1}=0$. By uniqueness, pointwise continuity of the EROT potentials then follows. %
	  Upon relabelling, we may assume that $(\alphab^{\lambda}_\ind, \betab^\lambda_\ind)$ converges pointwise, i.e.,  for each $x \in \XC, y \in \YC$ we have
		\begin{align*}
			\exp\bigg(\frac{ -\alpha^{\lambda}_{\infty,x}}{\lambda} \bigg) &= \lim_{\ind \rightarrow \infty} \exp\bigg( \frac{-\alpha^{\lambda}_{\ind,x}}{\lambda} \bigg) = \lim_{\ind \rightarrow \infty} \sum_{y \in \YC}\exp\bigg( \frac{\beta^\lambda_{\ind,y} -c(x,y)}{\lambda} \bigg) s_{\ind,y},\\
			\exp\bigg( \frac{-\beta^\lambda_{\infty,y}}{\lambda} \bigg) &= \lim_{\ind \rightarrow \infty} \exp\bigg(\frac{ -\beta^\lambda_{\ind,y}}{\lambda} \bigg) = \lim_{\ind \rightarrow \infty} \sum_{x \in \XF}\exp\bigg( \frac{\alpha^{\lambda}_{\ind,y} -c(x,y)}{\lambda} \bigg) r_{\ind,x}.
		\end{align*}
		In the following, we show that the limit expression and the sum on the right-hand side can be interchanged for each $x\in \XC$ and $y \in \YC$. As noted in Remark \ref{rem:NonConstantDualPotentials} is this condition necessary and sufficient for EROT potentials.

		 Based on bound \eqref{eq:LowerUpperBounds} it follows for all  $y \in \YC$ and $\ind \in \N$ that
		\begin{equation}
		\begin{aligned}
	  \exp\bigg( \frac{\beta^\lambda_{\ind,y} -c(x,y)}{\lambda} \bigg) &\leq \exp\bigg(\frac{\cyp(y) - \cym(y)}{\lambda} + \frac{K_\XC + K_\beta}{\lambda}\bigg).
	  \end{aligned}
	  \label{eq:UpperboundExpBeta}
	\end{equation}
		Since $\eySup{}\in \ls$ there exists some $N_1\in \NN$ such that 
		$$\sum_{y \not\in \{y_1, \dots, y_{N_1}\}}\exp\bigg( \frac{\beta^\lambda_{\infty,y} -c(x,y)}{\lambda} \bigg) s_{\ind,y} \leq   \frac{\epsilon}{8}.$$
		By convergence of $(\sb_\ind)_{\ind\in\N}$ in $\ell^\infty(\FC_\XC)$ and due to \eqref{eq:UpperboundExpBeta} there is  $N_2 \in \NN$ such that for $\ind \geq N_2$, 
		$$ \sum_{y \not\in \{y_1, \dots, y_{N_1}\}}\exp\bigg( \frac{\beta^\lambda_{\ind,y} -c(x,y)}{\lambda} \bigg) s_{\ind,y} \leq   \frac{\epsilon}{4}.$$
		By pointwise convergence of $\beta^\lambda_\ind$ there also exists $N_3\in \N$ such that it holds for all $\ind \geq N_3$ and $y \in \{y_1, \dots, y_{N_1}\}$ that $$ \left|\exp\bigg(\frac{\beta^\lambda_{\ind,y} -c(x,y)}{\lambda} \bigg) - \exp\bigg(\frac{\beta^\lambda_{\infty,y} -c(x,y)}{\lambda} \bigg) \right| \leq \frac{\epsilon}{4}. $$
		Moreover, by \eqref{eq:UpperboundExpBeta} 
		there exists $N_4\in \N$ such that for all $\ind \geq N_4$ follows
		$$ \left| \sum_{y \in \YC}\exp\bigg(\frac{\beta^\lambda_{\infty,y} -c(x,y)}{\lambda} \bigg)( s_{\ind,y} - %
		s_y )\right| \leq \frac{\epsilon}{4}.$$
		These previous four inequalities yield for $\ind \geq \max\{N_1, N_2, N_3, N_4\}$ that \begin{eqnarray*}
			&&\left| \sum_{y \in \YC}\exp\bigg(\frac{\beta^\lambda_{\ind,y} -c(x,y)}{\lambda} \bigg)s_{\ind,y} -  \sum_{y \in \YC}\exp\bigg(\frac{\beta^\lambda_{\infty,y} -c(x,y)}{\lambda} \bigg)s_y \right|
			\leq \frac{7}{8}\epsilon \leq  \epsilon.
		\end{eqnarray*}
		As $\epsilon>0$ can be chosen arbitrarily small, it holds for all $x\in \XC$  that \begin{equation*}
			\exp\bigg( \frac{-\alpha^{\lambda}_{\infty,x} }{\lambda}\bigg) = \lim_{\ind \rightarrow \infty} \sum_{y \in \YC}\exp\bigg(\frac{\beta^\lambda_{\ind,y} -c(x,y)}{\lambda} \bigg) s_{\ind,y} = \sum_{y \in \YC}\exp\bigg(\frac{\beta^\lambda_{\infty,y} -c(x,y)}{\lambda} \bigg) s_{y}.
			\end{equation*}
			Likewise, it follows by an analogous argument for all $y\in \YC$ that  \begin{equation*}\exp\bigg( \frac{-\beta^\lambda_{\infty,y}}{\lambda} \bigg) = \lim_{\ind \rightarrow \infty} \sum_{x \in \XF}\exp\bigg(\frac{\alpha^{\lambda}_{\ind,x} -c(x,y)}{\lambda} \bigg) r_{\ind,x} = \sum_{x \in \XF}\exp\bigg(\frac{\alpha^{\lambda}_{\infty,x} -c(x,y)}{\lambda} \bigg) r_{x}.\end{equation*}
			This shows that $(\alpha^\lambda_\infty, \betab^\lambda_\infty)$ is a pair of EROT potentials which fulfills the correspondence from Remark \ref{rem:NonConstantDualPotentials} on spaces $\XC$ and $\YC$ while also satisfying $\beta_{\infty, y_1}^\lambda = 0$. Hence, $(\alpha^\lambda_\infty, \betab^\lambda_\infty)= (\alpha^\lambda, \betab^\lambda)$, which verifies assertion $(ii)$.

		\emph{Step 3 - Continuity of EROT plan in $\lXYSub{\cx\oplus\cy}$.} To prove assertion $(iii)$ we first choose the dominating function for the cost function as $\cxm = \tcxm, \cxp = \tcxp, \cym = \tcym$, $\cyp= \tcyp$. Then it follows by Lemma \ref{lem:InclusionMeasures} that $\norm{\cdot}_{\ell^\infty(\FC_\XC)}\leq \norm{\cdot}_{\lXSub{\kxSup{1}}}$ on  $\PC(\XC)$ and $\norm{\cdot}_{\ell^\infty(\FC_\YC)}\leq \norm{\cdot}_{\lYSub{\kySup{1}}}$ on  $\PC(\YC)$. Hence, the convergence of the probability measures $(\rb_\ind)_{\ind\in \NN}$ to $\rb$ in $\lXSub{\kxSup{1}}$ and $(\sb_{k})_{\ind\in \NN}$  to $\sb$ in $\lYSub{\kySup{1}}$ yields uniform bounds and pointwise convergence of the EROT potentials. 
		We thus obtain from the correspondence between EROT plan and potentials (Proposition \ref{prop:BoundsOptimalPotentials}) that 
		\begin{align}
			&\quad \norm{\pi^\lambda(\rb_\ind, \sb_\ind)-\pi^\lambda(\rb, \sb)}_{\lXYSub{\cx\oplus \cy}} \notag \\
			&\leq \sum_{\substack{x\in \XC\\y\in \YC}}(\cx(x) + \cy(y))\exp\left(\frac{\alpha^\lambda_{k,x}\!+\!\beta^\lambda_{k,y}\! - \!c(x,y)}{\lambda}\right) \left(r_{k,x}\left| s_{k,y}\!-\!s_{y}\right| + \left|r_{k,x}\! - \!r_x\right|s_{y}\right)\notag\\
			&\; +\sum_{\substack{x\in \XC\\y\in \YC}} (\cx(x) + \cy(y))\left|\exp\left(\frac{\alpha^\lambda_{k,x}\!+\!\beta^\lambda_{k,y}\! -\! c(x,y)}{\lambda}\right) -  \exp\left(\frac{\alpha^\lambda_{x}\!+\!\beta^\lambda_{y}\! - \!c(x,y)}{\lambda}\right) \right|r_{x}s_{y}.\notag
		\end{align}
		Our bounds from \eqref{eq:LowerUpperBounds} for the EROT potentials assert that 
		\begin{align*}
			& (\cx(x) + \cy(y))\exp\left(\frac{\alpha^\lambda_{k,x}+\beta^\lambda_{k,y} - c(x,y)}{\lambda}\right) \\
			\leq \;&(\cx(x)+\cy(y)) \exSup{}(x)\eySup{}(y)\exp\left(\frac{K_\XC+ K_\YC + 2K_\beta}{\lambda}\right)\\
			\leq \;& 2\kxSup{1}(x)\kySup{1}(y)\exp\left(\frac{K_\XC+ K_\YC + 2K_\beta}{\lambda}\right).
		\end{align*}
		Hence, the first term in the bound for $\|\pi^\lambda(\rb_\ind, \sb_\ind)\!-\!\pi^\lambda(\rb, \sb)\|_{\lXYSub{\cx\oplus \cy}}$ is dominated by 
		\begin{align*}
			2 \exp\left(\frac{K_\XC+ K_\YC + 2K_\beta}{\lambda}\right) \left(\norm{r_k}_{\lXSub{\kxSup{1}}}\norm{s_k-s}_{\lYSub{\kySup{1}}} + \norm{r_k-r}_{\lXSub{\kxSup{1}}}\norm{s}_{\lYSub{\kySup{1}}}\right),
		\end{align*}
		which converges by assumption to zero. For the second term in the upper bound note that the integrand with respect to $\rb\otimes \sb$ pointwise converges to zero and is dominated by $$4\kxSup{1}(x)\kySup{1}(y)\exp\left(\frac{K_\XC+ K_\YC + 2K_\beta}{\lambda}\right).$$ Hence, by dominated convergence it also tends to zero, thus finishing the proof for $(iii)$. \qedhere
	\end{proof}
	
	\begin{proof}[Proof of Proposition \ref{prop:ContinuityEROTVanishingRegularization}]
	
		The proof is divided into two~steps. First we prove the continuity of the EROT value, then we proceed with pointwise continuity of EROT potentials. %
	
		\emph{Step 1 - Continuity of EROT value. } We first perform the decomposition 
		\begin{align*}
			\EROTLambda{\lambda_\ind} (\rb_\ind, \sb_\ind) -  \OT(\rb, \sb)				 = \;& \EROTLambda{\lambda_\ind} (\rb_\ind, \sb_\ind) -   \EROTLambda{\lambda_\ind} (\rb_\ind, \sb)\\
			 & +  \EROTLambda{\lambda_\ind} (\rb_\ind, \sb) -   \EROTLambda{\lambda_\ind} (\rb, \sb)\\
			 & +  \EROTLambda{\lambda_\ind} (\rb, \sb) -   \OT (\rb, \sb).
		 \end{align*}
		 Define $C \coloneqq \norm{c}_{\ell^\infty(\XC\times \YC)}$ and select the dominating functions $\cxpb(x) = \cypb(y) = C/2$  and $\cxmb(x) = \cymb(y) = -C/2$ for the cost function. Then, by Lemma \ref{lem:LipschitzEROTValue} it follows that  
		 \begin{align*}
			& \;|\EROTLambda{\lambda_\ind} (\rb_\ind, \sb_\ind) -   \EROTLambda{\lambda_\ind} (\rb_\ind, \sb)| + |\EROTLambda{\lambda_\ind} (\rb_\ind, \sb) -   \EROTLambda{\lambda_\ind} (\rb, \sb)|\\
			\leq & \; \Big(4(C+1) +
			2 \lambda \log \exp(C/\lambda) \Big)\cdot \left(\norm{\rb_\ind- \rb}_{\lXSub{\cx}}+\norm{\sb_\ind- \sb}_{\lXSub{\cy}}\right)\\
			\leq &\; 6(C+1)C \left(\norm{\rb_\ind- \rb}_{\lXSub{}} + \norm{\sb_\ind- \sb}_{\lXSub{}} \right).
		 \end{align*}
		For the third term if follows from \eqref{eq:StabilityInTermsOfRegularization} and since $H(\rb,\sb)<\infty$ that 
		\begin{align*}
			|\EROTLambda{\lambda_\ind} (\rb, \sb) -   \OT (\rb, \sb)| \leq \lambda_\ind H(\rb,\sb) = \landau(\lamdba_\ind).
		\end{align*}
		Combining these inequalities with the assumption the claim proves Assertion $(i)$. 
		
		\emph{Step 2 - Pointwise continuity of EROT potentials.}
		To show Assertion $(ii)$, note by uniform boundedness of the cost function from Proposition \ref{prop:BoundsOptimalPotentials} that there exists a constant $K(c)>0$ such that the sequence of EROT potentials $(\alphab^{\lambda_\ind}_\ind,\betab^{\lambda_\ind}_\ind)$ is  uniformly bounded by $K(c)$.
	 
		Using a diagonal argument a pointwise converging subsequence  $(\alphab^{\lambda_{\ind_l}}_{\ind_l},\betab^{\lambda_{\ind_l}}_{\ind_l})_{l\in\NN}$ exists with limiting element $(\tilde \alphab,\tilde \betab)$. Note that $\tilde \betab_{y_1} = 0$ necessarily.
		
		Once we prove that the limit $(\tilde \alphab,\tilde \betab)$ coincides on $\supp(\rb)\times \supp(\sb)$ with an OT potential for $(\rb,\sb)$, Assertion $(ii)$ follows by the assumed uniqueness (up to a constant shift) of OT potentials. 
		To this end, first note by Assertion $(i)$ and dominated convergence that 
		\begin{align*}
			\left|\OT(\rb, \sb) - \langle \tilde \alphab, \rb \rangle + \langle \tilde \betab, \sb \rangle \right| =&\left|\lim_{l\rightarrow\infty} \left(\EROTLambda{\lambda_{\ind_l}}(\rb_{\ind_l}, \sb_{\ind_l}) -  \langle \alphab^{\lambda_{\ind_l}}_{\ind_l}, \rb \rangle - \langle \betab^{\lambda_{\ind_l}}_{\ind_l}, \sb \rangle \right)\right|\\
			\leq  & \left| \lim_{l\rightarrow\infty} \left(\EROTLambda{\lambda_{\ind_l}}(\rb_{\ind_l}, \sb_{\ind_l}) -  \langle \alphab^{\lambda_{\ind_l}}_{\ind_l}, \rb_{\ind_l} \rangle - \langle \betab^{\lambda_{\ind_l}}_{\ind_l}, \sb_{\ind_l} \rangle \right)\right| \\
			& + K(c)  \lim_{l\rightarrow \infty}\left(\norm{ \rb-\rb_{\ind_l}}_{\lX}+\norm{ \sb-\sb_{\ind_l}}_{\lY}\right) = 0.
		\end{align*}
	The remainder of this step is concerned with showing that  
	\begin{align}
		\tilde \alpha(x) +\tilde \beta(y)&\leq c(x,y) \quad \text{for each } (x,y) \in \supp(\rb)\times \supp(\sb).\label{eq:ineqOnSupp}
	\intertext{It indeed suffices to verify this inequality only on  $\supp(\rb)\times \supp(\sb)$ since its validity implies by \citet[Lemma 2.6]{nutz2022entropic} the existence of functions $\overline \alpha\colon \XC\rightarrow \RR, \overline \beta\colon \YC\rightarrow \RR$ which coincide on $\supp(\rb), \supp(\sb)$ with $\tilde \alpha, \tilde \betab$, respectively, and fulfill }
		\overline \alpha(x) +\overline \beta(y)&\leq c(x,y) \quad \text{for each } (x,y)\in\XC\times \YC.\notag
	\end{align} 
		To prove \eqref{eq:ineqOnSupp}, suppose there existed a pair $(x',y') \in \supp(\rb)\times \supp(\sb)$ and $\delta>0$ such that $\tilde \alpha_{x'} +\tilde \beta_{y'}> c(x',y') + \delta$. Then, by pointwise convergence, for sufficiently large $l$ it would follow that $\alpha^{\lambda_{\ind_l}}_{x'} +\beta^{\lambda_{\ind_l}}_{\ind_l, y'}> c(x',y') + \delta/2$. In consequence, a contradiction would ensue,
		\begin{align*}
			&-\norm{c}_{\lInfXY}\leq OT(\rb, \sb) = \liminf_{k\rightarrow \infty} \EROTLambda{\lambda_{\ind}}(\rb_{\ind}, \sb_{\ind}) \\
			= \,& \liminf_{k\rightarrow \infty} \langle \alphab^{\lambda_{\ind}}_{\ind}, \rb_{\ind} \rangle + \langle \betab^{\lambda_{\ind}}_{\ind}, \sb_{\ind} \rangle - \lambda \bigg[\sum_{\substack{x\in \XC\\y\in \YC}}\exp\bigg(\frac{\alpha^{\lambda_{\ind}}_{\ind,x} + \betab^{\lambda_{\ind}}_{\ind,y} - c(x,y)}{\lambda_k}\bigg)\rb_{\ind,x} \sb_{\ind,y} -\rb_{\ind,x} \sb_{\ind,y}\bigg]\\
			\leq \,& \liminf_{k\rightarrow \infty} 2 K(c) -\lambda_k\Big( \exp\Big(\frac{\delta}{2\lambda_k}\Big)\rb_{\ind,x'}\sb_{\ind,y'} - 1\Big)  =  - \infty.
		\end{align*}
		Hence, we conclude the validity of \eqref{eq:ineqOnSupp} and thus finish the proof of Assertion $(ii)$. \qedhere

	\end{proof}
	
	\subsection{Hadamard Differentiability of Entropic Optimal Transport Value}\label{subsec:ProofsSensitivityEROTV}

	\begin{proof}[Proof of Theorem \ref{them:EntropicOTCostUnboundedIsHadamardDifferentiable}]\label{prf:HadamardDifferentiabilityOfEntropicOptimalCosts}
	The proof is divided into three steps. After providing some relevant notation and conventions in Step 1, we show suitable lower and upper bound in Step 2 and prove in Step 3 that they converge to the asserted Hadamard derivative. 
	
		\emph{Step 1 - Preliminary conventions.}
		To show Hadamard differentiability, consider a positive sequence $(t_n)_{n\in \N}$ with $t_n \searrow 0$ and a converging sequence $(\hb_{n}^{ \XF}, \hb_{n}^{\YC})_{n \in \N} \in \ell^\infty(\FC_\XC) \times \ell^\infty(\FC_\YC)$  with limit $(\hb^{\XF}, \hb^{\YC})$ such that  $(\rb + t_n \hb_{n}^{ \XF}, \sb + t_n \hb_{n}^{\YC}) \in \big(\probset{\XC}\times \ell^\infty(\FC_\XC)\big)\times \big(\probset{\YC}\times \ell^\infty(\FC_\YC)\big)$ for all $n \in \N$. By definition of $\FC_\XC$ and $\FC_\YC$ and Lemma \ref{lem:InclusionMeasures} the sequence $(\hb_{n}^{ \XC},\hb_n^\YC)$ may be interpreted as a signed measure in $\lXSub{\cx}\times \lYSub{\cy}$ with $\sum_{x\in \XC}\hb_{n,x}^{ \XC} = 0$ and $\sum_{y\in \YC}\hb_{n,y}^{ \YC} = 0$ and where convergence to $(\hb^\XC,\hb^\YC)$ is even stronger than induced by that weighted $\ell^1$-spaces. 
		Further, let $(\alphab^\lambda_n, \betab^\lambda_n)$, $(\tilde\alphab^\lambda_n, \tilde \betab^\lambda_n)$, $(\alphab^\lambda, \betab^\lambda)$ be the EROT potentials for $(\rb + t_n \hb_{n}^{\XF}, \sb + t_n \hb_{n}^{\YC})$, $(\rb, \sb + t_n \hb_{n}^{\YC})$,  $(\rb, \sb)$, respectively,  according to the convention in Remark~\ref{rem:NonConstantDualPotentials} with $\beta^\lambda_{n,y_1}= \tilde \beta^\lambda_{n,y_1}=\beta^\lambda_{y_1} = 0$ for all $n \in \NN$. %
	
		\emph{Step 2 - Bound for difference quotient.}
		Adding a non-trivial zero to the difference quotient for the EROT value yields 
		\begin{align*}
			& \;\frac{1}{t_n}\left(EROT(\rb + t_n\hb_{n}^{ \XF}, \sb + t_n\hb_{n}^{ \YC}) - \EROT(\rb, \sb) \right)\\
			= &\; \frac{1}{t_n}\left(\EROT(\rb + t_n\hb_{n}^{ \XF}, \sb + t_n\hb_{n}^{ \YC}) - \EROT(\rb , \sb+ t_n\hb_{n}^{ \YC})\right)\\
			& +\frac{1}{t_n}\left(\EROT(\rb , \sb+ t_n\hb_{n}^{ \YC})- \EROT(\rb, \sb)\right).
		\end{align*}
		This asserts by Lemma \ref{lem:LipschitzEROTValue} that 
		\begin{align*}
			 & \langle \tilde \alphab^\lambda_n, \hb_{n}^\XC\rangle +  \langle\betab^\lambda,  h_n^\YC\rangle \\ \leq \;& \frac{1}{t_n}\left(EROT(\rb + t_n\hb_{n}^{ \XF}, \sb + t_n\hb_{n}^{ \YC}) - \EROT(\rb, \sb)\right) \\ \leq \;&  \langle \alphab^\lambda_n,\hb_{n}^{ \XF}\rangle + \langle \tilde \betab^\lambda_n,\hb_{n}^{ \YC}\rangle.
	   \end{align*}
	   Hence, once we show that the lower and upper bound tend for $n\rightarrow \infty$ to $\langle \alphab^\lambda,\hb^{ \XF}\rangle + \langle \betab^\lambda,\hb^{ \YC}\rangle$, Hadamard differentiability tangentially to $\PC(\XC)\times \PC(\YC)$ follows.

	   \emph{Step 3 - Convergence of lower and upper bound. }
	   We only state the proof for the upper bound, the argument for the lower bound is analogous. 
	   Adding a non-trivial zero implies
	   \begin{align}
		\label{eq:UpperBoundConvergence1}
		&\left|\langle  \alphab^\lambda_n, \hb_{n}^\XC\rangle +  \langle \tilde \betab_n^\lambda,  h_n^\YC\rangle - \langle \alphab^\lambda,\hb^{ \XF}\rangle - \langle \betab^\lambda,\hb^{ \YC}\rangle\right|\\
		\leq \;& \left|\langle  \alphab^\lambda_n, \hb_{n}^\XC - \hb^\XC\rangle\right| + \left|\langle \alphab^\lambda_n - \alphab^\lambda, \hb^\XC\rangle\right| + \left|\langle \tilde \beta^\lambda_n, \hb_{n}^\YC - \hb^\YC\rangle\right| + \left|\langle \tilde \beta^\lambda_n - \beta^\lambda, \hb^\YC\rangle\right|.\label{eq:UpperBoundConvergence2}
	   \end{align}
		 Since $\lim_{n\rightarrow\infty} t_n(\|\hb^\XC_n\|_{\ell^\infty(\FC_\XC)} + \|\hb^\YC_n\|_{\ell^\infty(\FC_\YC)})= 0$ it follows that $(\alphab^\lambda_n,\tilde \beta^\lambda_n)$ converges pointwise on $\XC$ and $\YC$ to $(\alphab^\lambda,\beta^\lambda)$. Following along the proof of Proposition \ref{prop:ConvergenceOptimalDualSolutions}, we infer that $$\|\alpha^\lambda\|_{\lInfXSub{\cx}} +\|\betab^\lambda\|_{\lInfYSub{\cy}}
		+\limsup_{n\rightarrow \infty}\left(\|\alpha_n^\lambda\|_{\lInfXSub{\cx}} +\|\tilde \betab_n^\lambda\|_{\lInfYSub{\cy}}\right)<\infty.$$
		Moreover, since $\hb^\XC$ and $\hb^\YC$ can be viewed as elements of $\lXSub{\cx}$ and $\lYSub{\cy}$ respectively (Lemma \ref{lem:InclusionMeasures}), we obtain from dominated convergence that the second and fourth term in \eqref{eq:UpperBoundConvergence2} converges to zero. Further, in conjunction with $$\norm{\hb_{n}^\XC - \hb^\XC}_{\lXSub{\cx}} +\norm{\hb_{n}^\YC - \hb^\YC}_{\lYSub{\cy}} \leq \norm{\hb_{n}^\XC - \hb^\XC}_{\ell^\infty(\FC_\XC)}+ \norm{\hb_{n}^\YC - \hb^\YC}_{\ell^\infty(\FC_\YC)}\rightarrow 0$$  we thus conclude using the previous two displays that the first and third term of \eqref{eq:UpperBoundConvergence2} also converge to zero. This shows that the difference in  \eqref{eq:UpperBoundConvergence1} converges to zero, which finishes the proof of Hadamard differentiability. \qedhere

		 \end{proof}
		 
		 \begin{proof}[Proof of Corollary \ref{cor:HadDiffDebiasedEROTvalue}]
			As a sum of Hadamard differentiable functionals (Theorem \ref{them:EntropicOTCostUnboundedIsHadamardDifferentiable}) it follows that $\overline{E\!ROT}^\lambda$ is also Hadamard differentiable with derivative
				\begin{align*}
			\DH_{|(\rb, \sb)} \overline {E\!ROT}^\lambda(h^\XC, h^\YC) &=\DH_{|(\rb, \sb)} {E\!ROT}^\lambda(h^\XC, h^\YC) \\
			&\quad -  \frac{1}{2}\left( \DH_{|(\rb, \rb)} {E\!ROT}^\lambda(h^\XC, h^\XC) +  \DH_{|(\sb, \sb)} {E\!ROT}^\lambda(h^\YC, h^\YC)\right).
		\end{align*}
		Since the cost function is symmetric it follows for any $\tilde \rb \in \PC(\XC)$ that  $\alphab^\lambda(\tilde \rb,\tilde \rb)= \betab^\lambda(\tilde \rb,\tilde \rb)$. The claim now follows at once since 
		\begin{equation*}
			\DH_{|(\rb, \rb)} {E\!ROT}^\lambda(h^\XC\!, h^\XC) = 2\langle \alphab^\lambda(\rb,\rb), h^\XC\rangle, \;\;\;
			\DH_{|(\sb, \sb)} {E\!ROT}^\lambda(h^\YC\!, h^\YC) = 2\langle \betab^\lambda(\sb,\sb), h^\YC\rangle. \qedhere
		\end{equation*}
		\end{proof}

	\subsection{Hadamard Differentiability of Entropic Optimal Transport Plan}\label{subsec:ProofsSensitivityEROTP}
	The proof of Theorem \ref{them:TransportPlanIsHadamardDiff} for the Hadamard differentiability of the EROT plan requires three key results: (1) well-definedness and boundedness of the proposed derivative, (2) the convergence of the difference quotient of $\pib^\lambda$ for finitely supported perturbations to the proposed derivative and (3) a local Lipschitz property for the EROT plan $\pib^\lambda$ as a map of its marginal probability measures. These three statements are provided in Propositions \ref{prop:ProposedDerivativeIsBounded}, \ref{prop:DerivativeFiniteSupportPerturbation}, and \ref{prop:DualOptimizerAreLipschitz}, respectively, whose proofs are deferred to the end of this chapter. We stress, for all of these results we assume $$\norm{\cxpb -\cxmb}_{\lInfX}< \infty.$$
	
		\begin{proposition}[Proposal for derivative of EROT plan]\label{prop:ProposedDerivativeIsBounded}
			Let $ \rb\in \PC(\XC)\cap\lXSub{\cx}$ and $\sb\in \PC(\YC)\cap\lYSub{\kySup{4}}$ be probability measures with full support. Then the linear map given by \begin{align*}
				&-\left[\DCW_{\pib, \alphab, \betab_*|( \varthetab( \rb, \sb_*),  \rb, \sb_* )}\FC\right]^{-1} \circ  \DCW_{\rb, \sb_*|( \varthetab( \rb, \sb_*),  \rb, \sb_*) }\FC\colon \; \ell^1_{\cx}(\XF)\times \ell^1_{\kySup{4}}(\YCC) \rightarrow \\ 
				&\quad \quad \quad \quad \quad \quad \quad \quad \quad \quad \quad \quad \quad \quad \ell^1_{\cx\oplus \cy}(\XF\times \YC) \times \ell^\infty_{\exSup{}}(\XF)\times \ell^\infty_{\eySup{}}(\YCC)
			\end{align*}
			is well-defined and bounded. 
			In particular, the $\pi$-component is given by 
			\begin{align*}
				&-\left[\DCW_{\pib, \alphab, \betab_*|( \varthetab( \rb, \sb_*),  \rb, \sb_* )}\FC\right]^{-1}_\pi \circ  \DCW_{\rb, \sb_*|( \varthetab( \rb, \sb_*),  \rb, \sb_*) }\FC(\hb^\XC, \hb^\YC) = \frac{\pib^\lambda}{ \rb \otimes  \sb }\odot \Big[   \rb \otimes \hb^\YC + \hb^\XF \otimes  \sb \Big]\\
				&-\pi^\lambda\odot A_*^T\begin{pmatrix}
					(\Id_{\XC} -\ACX\ACY)^{-1}&-(\Id_{\XC} -\ACX\ACY)^{-1}\ACX\\
					-\ACY(\Id_{\XC} -\ACX\ACY)^{-1}&(\Id_{\YCC} -\ACY\ACX)^{-1}\\
				\end{pmatrix}\begin{pmatrix}
					0 & \BCX\\
					\BCY&0
				\end{pmatrix}\begin{pmatrix}
					h^{\XC}\\
					h^{\YC}
				\end{pmatrix}.
			\end{align*}
		\end{proposition}
	
		Based on the correspondence between EROT plan and potentials (Proposition \ref{prop:DualEROT}) it follows that the support of the EROT plan coincides with the support of the product measure of the underlying marginal distributions. Hence, as the support of perturbed measures may be affected by the perturbation, additional mathematical challenges in the analysis of the EROT plan come into play. To circumvent these issues, we define for an element $\hb^\XF\in \lX$ its \emph{finite support approximation} of order $l\geq 2$, denoted by $ \hat \hb^\XF_l$, for $x\in \XF = \{x_1, x_2, \dots \}$ as 
		\begin{align}\label{eq:FiniteSupportApproximation}
			\hat h^\XF_{l,x} \coloneqq \begin{cases}
				h^\XF_{x_1} + \sum_{i = l+1}^{\infty}h^\XF_{x_i} & \text{ if } x = x_1, \\
				h^\XF_{x} & \text{ if } x \in \{x_2, \dots, x_l\},\\
				0 & \text{ else.}
			\end{cases} 	
		\end{align}
		Analogously, we denote the finite support approximation of an element $\hb^\YC\in \lY$ by $\hat \hb^\YC_{l}$. Useful technical properties of this type of approximation are derived in Appendix \ref{subsec:App:FSAProperties} (see Lemmata \ref{lem:FiniteSupportApproximationIsUniformlyGoodForLargeL} and \ref{lem:UpperboundNorms}) and will play in important role in the proofs of the subsequent results. 
		
		An appealing consequence of the finite support approximation in the context of Hadamard differentiability is	that it allows us to effectively bypass the challenges of varying supports.
	
		\begin{lemma}\label{lem:GivenLChooseNBigSoAllProbMeasures}
			Consider a strictly positive function $f\colon \XC\rightarrow [1, \infty)$  and let  $\rb \in \PC(\XC)\cap \lXSub{f}$ with full support $\supp(\rb) = \XC$. Take a positive sequence $(t_n)_{n \in \N}$ such that $t_n \searrow 0$ and $(\hb_n^\XF)_{n \in \N}\subseteq \lXSub{f}$ with limit $\hb^\XF$ such that $(\rb + t_n \hb_n^\XF)\in \probset{\XC}\cap \lXSub{f}$ for all $n \in \NN$. 
			Then, for any $l\in \N$ there exists an integer $N\in \N$ such that $(\rb + t_n\hat\hb^\XF_{l}), (\rb + t_n\hat\hb^\XF_{n,l}) \in \probset{\XF}\cap \lXSub{f}$ and $\supp(\rb + t_n\hat\hb^\XF_{l}) = \supp(\rb + t_n\hat\hb^\XF_{n,l}) = \XF$ for all $n\geq N$.
		\end{lemma}
		
		The proof of Lemma \ref{lem:GivenLChooseNBigSoAllProbMeasures} is given in Appendix \ref{subsec:App:FSAProperties}. 
		With Lemma \ref{lem:GivenLChooseNBigSoAllProbMeasures} at our disposal, we can state a directional  differentiability result for the EROT plan, which represents the second key result for the proof of Theorem \ref{them:TransportPlanIsHadamardDiff}.
	
		\begin{proposition}[Differentiability for finitely supported perturbations]\label{prop:DerivativeFiniteSupportPerturbation}
			Let $ \rb\in \PC(\XC)\cap\lXSub{\cx}$ and $\sb\in \PC(\YC)\cap\lYSub{\kySup{4}}$ be probability measures with full support. Given $l\in \N$ consider $N\in \N$ according to Lemma \ref{lem:GivenLChooseNBigSoAllProbMeasures} such that for all $n\geq N$ it holds that  $$(\rb + t_n\hat{\hb}^\XF_l,{\sb} + t_n\hat{\hb}^\YC_{l})\in \big(\probset{\XF}\times \lXSub{\cx}\big) \times \big(\probset{\YC}\cap \lYSub{\kySup{4}}\big)$$ with $\supp(\rb + t_n\hat{\hb}^\XF_l) = \XF$ and  $\supp(\sb + t_n\hat{\hb}^\YC_l) = \YC$. Then,  for $n\rightarrow \infty$, it follows  that \begin{align*}\bigg\|& \frac{\pib^\lambda(\rb + t_n\hat{\hb}^\XF_l,{\sb} + t_n\hat{\hb}^\YC_{l}) - \pib^\lambda(\rb,{\sb})}{t_n}  - \DH_{| \rb, {\sb} }\pib^\lambda(\hat\hb^{\XF}_l, \hat \hb^{\YC}_{l})\bigg\|_{\lXYSub{\cx\oplus\cy}}\rightarrow 0. \end{align*}
			\end{proposition}

		Finally, we also show that the EROT plan fulfills a local Lipschitz property with respect to a suitable weighted $\ell^1$-norm of the underlying measures. 
	
		\begin{proposition}[Local Lipschitzianity of EROT plan]\label{prop:TransportPlanIsLipschitz}\label{prop:DualOptimizerAreLipschitz}
			Let $ \rb\in \PC(\XC)\cap\lXSub{\cx}$ and $\sb\in \PC(\YC)\cap\lYSub{\kySup{4}}$ be probability measures with full support. Denote for $\rho>0$ the set$$\begin{aligned}B_\rho( \rb, \sb) \coloneqq \bigg\{ (\tilde\rb, \tilde\sb) \in \Big(\probset{\XC}\cap\lXSub{\cx}\Big)\times\Big( \probset{\YC}\cap\lYSub{\kySup{4}}\Big) \colon \\ \norm{ \rb - \tilde\rb}_{\lXSub{\cx}}+\norm{ \sb - \tilde\sb}_{\lYSub{\kySup{4}}} \leq \rho  \bigg\}.\end{aligned}$$ 
			Then there exist $\rho_0, \Lambda, \Lambda'>0$ such that for any $(\tilde\rb, \tilde\sb), (\tilde \rb', \tilde \sb')\in B_{\rho_0}( \rb,  \sb)$ 
			 with
			 $\supp(\tilde \rb') \subseteq \supp(\tilde\rb)$ and $\supp(\tilde \sb')\subseteq\supp(\tilde\sb)$ it follows that
			\begin{align}\label{eq:LocalLipschitzForEROTP}
		  \norm{\pib^\lambda(\tilde\rb, \tilde\sb) - \pib^\lambda(\tilde\rb', \tilde\sb')}_{\lXYSub{\cx \oplus \cy}} &\leq \Lambda \norm{(\tilde\rb, \tilde\sb) - (\tilde\rb', \tilde\sb') }_{\lXSub{\cx} \times\lYSub{\kySup{4}}},\\
		  \norm{\big(\alphab^\lambda(\tilde\rb, \tilde\sb) - \alphab^\lambda(\tilde \rb', \tilde\sb')\big) \Indicator{\supp(\tilde \rb')}}_{\lInfXSub{}}& \leq \Lambda' \norm{(\tilde\rb, \tilde\sb) - (\tilde\rb', \tilde\sb') }_{\lXSub{\cx} \times\lYSub{\kySup{4}}},\\
		  \norm{\big(\betab^\lambda(\tilde\rb, \tilde\sb) - \betab^\lambda(\tilde \rb', \tilde\sb')\big)\Indicator{\supp(\tilde \sb')}}_{\lInfYSub{\eySup{}}}\,&\leq \Lambda' \norm{(\tilde\rb, \tilde\sb) - (\tilde\rb', \tilde\sb') }_{\lXSub{\cx} \times\lYSub{\kySup{4}}},
		\label{eq:LocalLipschitzForDualOptimizer}
		\end{align}
		where $(\alphab^\lambda, \betab^\lambda)$ represent EROT potentials as in Proposition \ref{prop:ConvergenceOptimalDualSolutions}, i.e., $\beta^\lambda_{y_1} =0$.
		\end{proposition}
		
		\noindent 
		Combining the previous three propositions enables us to state the proof for Theorem \ref{them:TransportPlanIsHadamardDiff}.

	\begin{proof}[Proof of Theorem \ref{them:TransportPlanIsHadamardDiff}]
		The characterization of the contingent cone follows by \citet[Proposition 4.2.1]{aubin2009setValued} in conjunction with $\rb$ and $\sb$ having full support. 
	To show Hadamard differentiability, take a positive sequence $(t_n)_{n \in \N}$ with $t_n \searrow 0$ 
	 and $(\hb_n^\XF,\hb_{n}^\YC)\subseteq \lXSub{\cx} \times \lYSub{\kySup{4}}$ converging to $(\hb^\XF, \hb^\YC)$ such that$$(\rb + t_n \hb_n^\XF, \sb + t_n \hb^\YC_{n})\in \big(\probset{\XC}\cap \lXSub{\cx}\big)\times \big( \probset{\YC}\cap \lYSub{\kySup{4}} \big)$$
	 for all $n \in \NN$. The quantity of interest is $$\norm{\frac{\pib^\lambda(\rb + t_n\hb^\XF_n,\sb + t_n\hb^\YC_{n}) - \pib^\lambda(\rb,\sb)}{t_n} - \DH_{|(\rb, \sb)}{\pib^\lambda}(\hb^\XF, \hb^\YC)}_\lXYSub{\cx\oplus \cy}  $$
	for which we need to prove that it converges to zero as $n$ tends to infinity. 
	Let $\epsilon>0$, then there exists by Lemma \ref{lem:FiniteSupportApproximationIsUniformlyGoodForLargeL} and by boundedness of $\DH_{|( \rb,  \sb)}{\pib^\lambda}$ (Proposition \ref{prop:ProposedDerivativeIsBounded}) an integer $l\in \N$ such that for the finite support approximations $\hat\hb^\XC_l$ of $\hb^\XC$ and $\hat\hb^\YC_l $ of $\hb^\YC$ it follows that  $$\norm{\hb^\XC - \hat\hb^\XC_l}_{\lXSub{\cx}} + \norm{\hb^\YC - \hat\hb^\YC_l}_{\lYSub{\kySup{4}}}<\frac{\epsilon}{4}\max\left(1, \norm{\DH_{|( \rb,\sb)}{\pib^\lambda}}_{OP}\right)^{-1}.$$
	Further, denote by $\rho_0>0$ the radius from Proposition \ref{prop:DualOptimizerAreLipschitz} such that $\pib^\lambda$ is Lipschitz with modulus $\Lambda> 0$ on the set $$\begin{aligned}B_{\rho_0}( \rb, \sb) \coloneqq \bigg\{ (\tilde\rb, \tilde\sb) \in \Big(\probset{\XC}\cap\lXSub{\cx}\Big)\times\Big( \probset{\YC}\cap\lYSub{\kySup{4}}\Big) \colon \\ \norm{ \rb - \tilde\rb}_{\lXSub{\cx}}+\norm{ \sb - \tilde\sb}_{\lYSub{\kySup{4}}} \leq \rho_0  \bigg\}.\end{aligned}$$
	Moreover, by Lemma \ref{lem:GivenLChooseNBigSoAllProbMeasures} there exists $N_1\in\N$ such that it follows for all $n\geq N_1$ that $\rb + t_n\hat\hb^\XF_{l}, \rb + t_n\hat\hb^\XF_{n,l} \in \probset{\XF}$ with $\supp(\rb + t_n\hat\hb^\XF_{l}) = \supp(\rb + t_n\hat\hb^\XF_{n,l}) = \XF$, and likewise $\sb + t_n\hat\hb^\YC_{l}, \sb + t_n\hat\hb^\YC_{n,l} \in \probset{\YC}$ with $\supp(\sb + t_n\hat\hb^\YC_{l}) = \supp(\sb + t_n\hat\hb^\YC_{n,l}) = \YC$. Additionally, Lemma \ref{lem:UpperboundNorms} asserts existence of $N_2 \in \N$ such that it holds for all $n\geq N_2$ that 
	$$( \rb + t_n \hb^\XC_n,  \sb + t_n \hb^\YC_n), 	( \rb + t_n \hat \hb^\XC_{n,l},  \sb + t_n \hat \hb^\YC_{n,l}), ( \rb + t_n \hat \hb^\XC_{l},  \sb + t_n \hat \hb^\YC_l) \in B_{\rho_0}( \rb,  \sb).$$
	Using Lemma \ref{lem:FiniteSupportApproximationIsUniformlyGoodForLargeL}  there also exists $N_3 \in \N$ such that it follows for all $n \geq N_3$ 
	$$\begin{aligned}
	 \norm{\hb^\XC_n - \hat\hb^\XC_{n,l} }_{\lXSub{\cx}} + \norm{\hb^\YC_n - \hat\hb^\YC_{n,l}}_{\lYSub{\kySup{4}}}&<\frac{\epsilon}{4 \Lambda} \quad \text{ and }\\
	 \norm{\hat\hb^\XC_{n,l} - \hat\hb^\XC_l}_{\lXSub{\cx}} + \norm{\hat\hb^\YC_{n,l} - \hat\hb^\YC_l}_{\lYSub{\kySup{4}}}&<\frac{\epsilon}{4 \Lambda}.
	 \end{aligned}$$
	Finally, by Proposition \ref{prop:DerivativeFiniteSupportPerturbation} there exists $N_4 \in \N$ such that it follows for $n \geq N_4$ 
	$$\norm{\frac{\pib^\lambda(\rb + t_n\hat\hb^\XF_{l},\sb + t_n\hat\hb^\YC_{l}) - \pib^\lambda(\rb ,\sb)}{t_n} - \DH_{|( \rb,  \sb)}{\pib^\lambda}(\hat\hb^\XF_l,\hat\hb^\YC_{l})}_\lXYSub{\cx\oplus \cy}<\frac{\epsilon}{4}.$$
	Summarizing, for all $n \geq \max\{N_1, N_2, N_3,N_4\}$ we obtain that 
				\begin{align*}
			& \,\,\,\,\,\,\norm{\frac{\pib^\lambda(\rb + t_n\hb^\XF_n,\sb + t_n\hb^\YC_{n}) - \pib^\lambda(\rb,\sb)}{t_n} - \DH_{|( \rb,  \sb)}{\pib^\lambda}(\hb^\XF, \hb^\YC)}_\lXYSub{\cx\oplus \cy} \notag\\
			 &\leq\norm{\frac{\pib^\lambda(\rb + t_n\hb^\XF_n,\sb + t_n\hb^\YC_{n}) - \pib^\lambda(\rb + t_n\hat\hb^\XF_{n,l},\sb + t_n\hat\hb^\YC_{n,l})}{t_n}}_\lXYSub{\cx\oplus \cy}\\
			&+  \norm{\frac{\pib^\lambda(\rb + t_n\hat\hb^\XF_{n,l},\sb + t_n\hat\hb^\YC_{n,l}) - \pib^\lambda(\rb + t_n\hat\hb^\XF_{l},\sb + t_n\hat\hb^\YC_{l})}{t_n}}_\lXYSub{\cx\oplus \cy}\\
			&+  \norm{\frac{\pib^\lambda(\rb + t_n\hat\hb^\XF_{l},\sb + t_n\hat\hb^\YC_{l}) - \pib^\lambda(\rb ,\sb)}{t_n} - \DH_{|( \rb,  \sb)}{\pib^\lambda}(\hat\hb^\XF_l,\hat\hb^\YC_{l})}_\lXYSub{\cx\oplus \cy} \\
			&+  \norm{\DH_{|( \rb, \sb)}{\pib^\lambda}(\hat\hb^\XF_l, \hat\hb^\YC_{l}) - \DH_{|( \rb,  \sb)}{\pib^\lambda}(\hb^\XF, \hb^\YC)}_\lXYSub{\cx\oplus \cy} \\
			&\leq \;\;\frac{\epsilon}{4} + \frac{\epsilon}{4} + \frac{\epsilon}{4}+ \frac{\epsilon}{4} \;\;\,\,= \;\;\,\,\epsilon,\notag
		\end{align*}
		which proves the claim. Moreover, the assertion on the representation of the derivative is shown in Proposition \ref{prop:ProposedDerivativeIsBounded}.
	\end{proof}

	\label{app:ProofsSensitivityEROTP}

	\begin{proof}[Proof of Proposition \ref{prop:ProposedDerivativeIsBounded}]
	
		The proof is divided into several steps. In Step 1 we reduce the analysis to the perturbation for EROT potentials. This reformulation is expressed in  Step 2 as a countable system of equations which can be written in terms of two operators. In Steps 3 and 4 we show that one operator is suitably bounded while the other is invertible. In particular, for the invertibility the condition $\norm{\cxpb - \cxmb}_{\lInfY}< \infty$ is exploited which enables the use of von Neumann series. Finally, in Step 5 we express the perturbation for the EROT plan in terms of the perturbation for EROT potentials and prove in Step 6 that the proposed derivative has bounded operator norm. 
		
		\emph{Step 1 - Reduction to perturbation of EROT potentials. }
		Given a pair $(\hb^\XF, \hb^\YC)\in \lXSub{\cx} \times\lYSub{\kySup{4 }}$ we need to show that there exists a unique element $(\hb^{\XF \times \YC},\hb^{\XF, \infty}, \hb_*^{\YC, \infty})\in \lXYSub{\cx\oplus \cy} \times \lInfXSub{} \times \lInfYYSub{\eySup{}}$ such that \begin{equation} \DCW_{\pib, \alphab, \betab_*|(\varthetab( \rb, \sb_*),  \rb, \sb_* )} \FC(\hb^{\XF \times \YC},\hb^{\XF, \infty}, \hb_*^{\YC, \infty}) =  \DCW_{\rb, \sb_*|(\varthetab( \rb, \sb_*),  \rb, \sb_* ) }\FC(\hb^\XF, \hb^\YC_*).\label{eq:DerivativeEquality}
		\end{equation}
		Denoting the EROT plan for $\rb, \sb$  as $\pib^\lambda$ and by the relation between optimizers of  \eqref{eq:EntropicOptimalTransport} and \eqref{eq:DualEntropicOptimalTransportProblem} (Proposition \ref{prop:DualEROT}) Equation \eqref{eq:DerivativeEquality} can be rewritten as \begin{equation}
		  \begin{pmatrix}
			\hb^{\XF\times \YC}- \frac{\pib^\lambda}{\lambda}\odot \Ab_*^T\Big( \hb^{\XF, \infty}, \hb_{*}^{\YC, \infty}\Big)\\[0.15cm]
			\Ab_*(\hb^{\XF\times \YC})
		\end{pmatrix} =  \begin{pmatrix}
			- \frac{\pib^\lambda}{ \rb \otimes  \sb }\odot \Big[   \rb \otimes \hb^\YC + \hb^\XF \otimes  \sb  \Big]\\
			\begin{pmatrix}
				-\hb^\XF\\
			-\hb_{*}^{\YC}	\end{pmatrix}
		\end{pmatrix}.\label{eq:MasterEquation}
		\end{equation}
		In order to solve this system of countably many equations we set 
		\begin{equation} \hb^{\XF\times \YC} \coloneqq \frac{\pib^\lambda}{\lambda}\odot \Ab_*^T\Big( \hb^{\XF, \infty}, \hb_{*}^{\YC, \infty}\Big) - \frac{\pib^\lambda}{ \rb \otimes  \sb }\odot \Big[   \rb \otimes \hb^\YC + \hb^\XF \otimes  \sb \Big],\label{eq:ProofWelldefinedOperatorDefinitionH^XFtimesYC}\end{equation}
		which reduces \eqref{eq:MasterEquation} to 
		\begin{equation}
			A_*\bigg(\frac{\pib^\lambda}{\lambda}\odot \Ab_*^T\Big( \hb^{\XF, \infty}, \hb_{*}^{\YC, \infty}\Big) - \frac{\pib^\lambda}{ \rb \otimes  \sb }\odot \Big[   \rb \otimes \hb^\YC + \hb^\XF \otimes  \sb \Big]\bigg) = \begin{pmatrix}
				-\hb^\XF\\
			-\hb_{*}^{\YC}	\end{pmatrix}.\label{eq:ToBeSolvedEquation}
		\end{equation}
		We now have to show that a unique solution $(\hb^{\XF, \infty}, \hb_*^{\YC, \infty})$ for \eqref{eq:ToBeSolvedEquation} exists. For this purpose, we evaluate for $x \in \XC$ the corresponding component-wise equation of \eqref{eq:ToBeSolvedEquation} and obtain
		\begin{align}
			\frac{\pi^{\lambda}_{xy_1} }{\lambda}h^{\XF, \infty}_x +\left[\sum_{y\in \YCC}{\frac{\pi^\lambda_{xy}}{\lambda} (  h_{x}^{\XF, \infty} + h_{*,y}^{\YC, \infty} ) }\right]  - \left[\sum_{y \in\YC}  \frac{\pi^\lambda_{xy}}{ r_x s_y} [ r_xh^{\YC}_y + h^{\XF}_x s_y]\right] &= -h^{\XF}_x,\notag\\
			\frac{ r_x}{\lambda}h^{\XF, \infty}_x +\left[\sum_{y\in \YCC}{\frac{\pi^\lambda_{xy}}{\lambda} h_{*,y}^{\YC, \infty} }\right]  - \left[\sum_{y \in\YC}  \frac{\pi^\lambda_{xy}}{ s_y}h^{\YC}_y \right] - h_{x}^{\XF}\left[\sum_{y\in \YC}\frac{\pi^\lambda_{xy}}{ r_x}\right] &= -h^{\XF}_x. \notag
		\end{align}
		Note, for all $x \in \XC$ it holds that $\sum_{y\in \YC}\frac{\pi^\lambda_{xy}}{ r_x} = 1$. This leads to 
		\begin{align*}
			 h^{\XF, \infty}_x +\left[\sum_{y\in \YCC}{\frac{\pi^\lambda_{xy}}{ r_x}  h_{*,y}^{\YC, \infty} }\right] &= \lambda\left[\sum_{y\in \YC}{\frac{\pi^\lambda_{xy}}{ r_x s_y}h_y^\YC}\right].
		\end{align*}
		Similarly, since  $\sum_{x \in \XC}\frac{\pi^\lambda_{xy}}{s_y} = 1$ for any $y \in \YCC$ we deduce that 
		\begin{align*}
			 \left[\sum_{x\in \XF}{\frac{\pi^\lambda_{xy}}{ s_y}  h^{\XF, \infty}_x }\right]+h_{*,y}^{\YC, \infty}&= \lambda\left[\sum_{x\in \XF}{\frac{\pi^\lambda_{xy}}{ r_x s_y}h_x^\XF}\right].
		\end{align*}
	
		\emph{Step 2 - Reduction to countable system of linear equations. }
		In summary, we have to solve a system of countably many linear equations. 
		This can be expressed in terms of the $\ACX, \ACY, \BCX, \BCY$, introduced in Section \ref{subsec:HadDiffPlan}, 
		\begin{equation}\label{eq:LinearEquationForInverse}
			\begin{aligned}
				\begin{pmatrix}
					\Id_{\XC} & \ACX\\
					\ACY\;\;\;& \Id_{\YCC}
				\end{pmatrix}
				\begin{pmatrix}
					h^{\XC,\infty}\\
					h^{\YC,\infty}_{*}
				\end{pmatrix} = \lambda 
				\begin{pmatrix}
					0 & \BCX\\
					\BCY& 0
				\end{pmatrix}
				\begin{pmatrix}
					h^{\XC}\\
					h^{\YC}%
				\end{pmatrix},
			\end{aligned}
		\end{equation}
		\noindent where we denote the operator of the l.h.s. by $\AC$ and the operator on the r.h.s. with the factor $\lambda$ by $\BC$. We will show that these operators are defined on the following spaces 
		$$\begin{aligned}
		\AC&\colon \lInfXSub{}\times \lInfYYSub{\eySup{}} \rightarrow \lInfXSub{}\times \lInfYYSub{\eySup{}},\\
		 \BC&\colon \lXSub{\cx} \times \lYYSub{\kySup{4 }}\rightarrow \lInfXSub{}\times \lInfYYSub{\eySup{}}.\end{aligned}$$
		Herein, we equip the product of $\ell^1$-spaces with sum of the norms of each subspace and  for the $\ell^\infty$-spaces we consider the maximum of the norms of the subspaces. In the subsequent two steps we derive a bound for the operator norm of $\BC$ and show that $\AC$ is invertible. 
		
		\emph{Step 3 - Bound on operator $Q$ on r.h.s. of \eqref{eq:LinearEquationForInverse}.}
		For the operator norm of $\BCX$ we employ the quantitative bounds for the EROT plan (Proposition \ref{prop:BoundsOptimalPotentials}). Denote the unit ball of $\lYSub{\kySup{4}}$ by $B_1(\lYSub{\kySup{4}})$, then since $1\leq \exSup{}$ and $\eySup{}\leq \kySup{4}$ we obtain 
		\begin{align*}
			\norm{\BCX}_{OP}& = \sup_{h^\YC\in B_1(\lYSub{\kySup{4}})}\sup_{x\in \XC} \left|\sum_{y \in \YC} \frac{\pi^\lambda_{xy}}{ r_x  s_y} h^\YC_y \right|\\
			& \leq \sup_{h\in B_1(\lYSub{\kySup{4}})} \sup_{x\in \XC} \left| \exSup{}(x)\langle \exSup{}, \rb\rangle \langle \eySup{}, \sb\rangle \norm{h^\YC}_{\lYSub{\eySup{}}}\right|\\
			& \leq \norm{\exSup{}}_{\lInfX}^2 \langle \eySup{}, \sb\rangle. 
		\intertext{Likewise, upon denoting the unit ball of $\lYSub{\cx}$ by $B_1(\lXSub{\cy})$,  it follows since $1\leq \cx$ that }
			\norm{\BCY}_{OP}& = \sup_{h^\XC\in B_1(\lXSub{\cx})}\sup_{y\in \YCC}\eySup{}(y)^{-1}\left|\sum_{x \in \XC}  \frac{\pi^\lambda_{xy}}{ r_x  s_y}h^\XC_x \right|\\
			&\leq  \sup_{h^\XC\in B_1(\lXSub{\cx})}\sup_{y\in \YCC} \norm{\exSup{}}_{\lInfX}^2\langle \eySup{}, \sb\rangle \norm{h^\XC}_{\lX}\\
			&\leq  \norm{\exSup{}}_{\lInfX}^2\langle \eySup{}, \sb\rangle \norm{h^\XC}_{\lXSub{\cx}}. 
		\end{align*}
		\noindent Hence, we obtain for the operator norm of $\BC$ that
		\begin{equation}
		  \norm{\BC}_{OP} \leq \lambda \norm{\exSup{}}_{\lInfX}^2\langle \eySup{}, \sb\rangle.\label{eq:boundOperatorQ}
		\end{equation}
	
		\emph{Step 4 - Invertibility of operator $\AC$ on l.h.s. of \eqref{eq:LinearEquationForInverse}.}
		Next, we show that the operator $\AC$ is invertible for which we apply the Neumann series calculus \citep[Theorem 2.9]{sasane2017FriendlyApproach_FA}. This requires that the Neumann series 
		\begin{align}\sum_{k = 0}^{\infty}
			(\Id - \AC)^k = \sum_{k = 0}^{\infty}
			\begin{pmatrix}
				0 &-\ACX\\
				- \ACY &0\\
			\end{pmatrix}^k
		\end{align}
		converges in operator norm. 
		To this end, we replace the norm of $\lInfYYSub{\eySup{}}$ by an equivalent norm.
		We first label $\YC = \{y_1, y_2, \dots, \}$ and set $\eta \coloneqq \inf_{x \in \XC}\frac{\pi^\lambda_{xy_1}}{ r_x}>0$ which is strictly positive by the lower bounds for $\pib^\lambda$ (Proposition \ref{prop:BoundsOptimalPotentials}) and since $\norm{\cxp - \cxm}_{\lInfX}<\infty$. Using the upper bound for $\pib^\lambda$ (Proposition \ref{prop:BoundsOptimalPotentials}) we obtain for any $x\in \XC$ that \begin{align*}&\eySup{}(y)\frac{\pi^\lambda_{xy}}{ r_x} \leq \norm{\exSup{}}_\lInfX^2\eySup{2}(y)s_y\leq \norm{\exSup{}}_\lInfX^2\kySup{4 }(y)s_y,
		\end{align*}
		which shows that $\eySup{}(y)\frac{\pi^\lambda_{xy}}{ r_x}$ is summable over $y \in \YC$ as $ \sb\in \lYSub{\kySup{4 }}$. Notably, the dominating function is independent of $x\in \XC$. Hence, there exists some $N\in \N$ such that 
		\begin{equation}\label{eq:ImportantConstraintForInvertibilityOfM}
		  \sum_{i = N+1}^{\infty} \big(\eySup{}(y_i)-1\big)\frac{\pi^\lambda_{xy_i}}{ r_x}  \leq \frac{\eta}{2}.\notag
		\end{equation}
		For such $N \in \N$ it follows for all $x \in \XC$ that  
		\begin{equation}
		  \sum_{ i = 2}^{\infty} \frac{\pi^\lambda_{xy_i}}{ r_x} + 
		  \sum_{i = N+1}^{\infty} \big(\eySup{}(y_i)-1\big)\frac{\pi^\lambda_{xy_i}}{ r_x}  \leq 1-\frac{\pi^\lambda_{xy_1}}{ r_x} + \frac{\eta}{2} \leq 1- \frac{\eta}{2}.\label{eq:lowerBoundPiQuotient}
		\end{equation}
		To change the norm of $\lInfYYSub{\exSup{}}$ we define the weight function $$\teySup{}\colon \YC \rightarrow [1, \infty), \quad \teySup{}(y) = \begin{cases}
			1 & \text{ if } y \in \{y_1, \dots, y_N\},\\
			\eySup{}(y) & \text{ else,}
		\end{cases} $$
		where we note that $L_{\YC}^{-1}\eySup{}(y)\leq \teySup{}(y)\leq \eySup{}(y)$ for $L_{\YC} = \max_{i= 1, \dots, N}\eySup{}(y)$. Consequently, it follows that $\norm{\cdot }_{\lInfYSub{\eySup{}}}\leq \norm{\cdot }_{\lInfYSub{\teySup{}}}\leq L_\YC \norm{\cdot }_{\lInfYSub{\eySup{}}}$. Hence, we can consider the operators $\ACX,\ACY$ as mappings $$\begin{aligned} \tACX\colon \lInfYYSub{\teySup{}} \rightarrow \lInfXSub{}, \quad   \tACY\colon \lInfXSub{} \rightarrow \lInfYYSub{\teySup{}},\end{aligned}$$ respectively, and similarly define $\tAC$ as the mapping $\AC$ defined on this product space with norm. 
		Based on \eqref{eq:lowerBoundPiQuotient} we then note that $\|\tACX\|_{OP}\leq 1 - \eta/2$ and since $\sum_{x \in \XC}\frac{\pi^\lambda_{xy}}{ s_{y}}=1\leq \teySup{}(y)$ for all $y \in \YCC$, it follows that $\|\tACY\|_{OP} \leq 1$. Thus, we infer that $$\norm{\begin{pmatrix}
			0 &-\tACX\\
			- \tACY &0\\
		\end{pmatrix}^{\!\!\!2}}_{OP}\leq \max\big\{\norm{\tACX \tACY}_{OP}, \norm{\tACY \tACX}_{OP}\big\} \leq 1-\frac{\eta}{2},$$ which implies that the Neumann series $\sum_{k = 0}^{\infty}
		(\Id - \tAC)^k$ converges in operator norm with   $$ \begin{aligned}\sum_{k = 0}^{\infty}{\norm{(\Id - \tAC)^{k}}_{OP}}&\leq  \Big(\norm{\Id}_{OP} + \norm{(\Id - \tAC)}_{OP} \Big) \sum_{k = 0}^{\infty} \norm{(\Id - \tAC)^{2}}_{OP}^{k}  \\
			& \leq \frac{2}{1 - (1 - \eta/2)} = \frac{4}{\eta} < \infty.\end{aligned}$$ %
		This implies that $\tAC$ has a continuous inverse with operator norm dominated by $\|\tAC^{-1}\|_{OP} \leq 4/\eta$. Returning to the original norm with $\lInfYYSub{\eySup{}}$, we obtain that $\norm{\AC^{-1} }_{OP} \leq 4 L_\YC /\eta$.
		Concluding, there exists a unique pair of elements $(\hb^{\XF,\infty}, \hb_*^{\YC, \infty})\in \lInfXSub{} \times \lInfYYSub{\eySup{}} $ that solves equation \eqref{eq:LinearEquationForInverse} and $\AC^{-1}\BC$ is a bounded operator. 
	
		\emph{Step 5 - Representation of derivative for EROT plan. } 
		Combining Equations \eqref{eq:MasterEquation} and \eqref{eq:LinearEquationForInverse} thus yields an explicit expression for the $\pi$-component of the proposed derivative, 
		\begin{align*}
			h^{\XC\times \YC} &= \left[\DCW_{\pib, \alphab, \betab_*|( \varthetab( \rb, \sb_*),  \rb, \sb_* )}\FC\right]^{-1}_\pi \circ  \DCW_{\rb, \sb_*|( \varthetab( \rb, \sb_*),  \rb, \sb_*) }\FC(\hb^\XC, \hb^\YC)  \\
			&=\frac{\pi^\lambda(\rb, \sb)}{\lambda}\odot A_*^T\AC^{-1}\BC\begin{pmatrix}
				h^{\XC}\\
				h^{\YC}
			\end{pmatrix} - \frac{\pib^\lambda}{ \rb \otimes  \sb }\odot \Big[   \rb \otimes \hb^\YC + \hb^\XF \otimes  \sb \Big]\\
			&=\pi^\lambda\odot A_*^T\begin{pmatrix}
				(\Id_{\XC} -\ACX\ACY)^{-1}&-(\Id_{\XC} -\ACX\ACY)^{-1}\ACX\\
				-\ACY(\Id_{\XC} -\ACX\ACY)^{-1}&(\Id_{\YCC} -\ACY\ACX)^{-1}\\
			\end{pmatrix}\begin{pmatrix}
				0 & \BCX\\
				\BCY&0
			\end{pmatrix}\begin{pmatrix}
				h^{\XC}\\
				h^{\YC}
			\end{pmatrix} \\
			&- \frac{\pib^\lambda}{ \rb \otimes  \sb }\odot \Big[   \rb \otimes \hb^\YC + \hb^\XF \otimes  \sb \Big].
		\end{align*}
		Herein, the expression for $\AC^{-1}$ is well-defined since $\max(\|\tACX\tACY\|, \|\tACY\tACY\|)<1$. 
		
		\emph{Step 6 - Bounds for perturbation of EROT plan.}
		Finally, we prove that $\hb^{\XC\times \YC}$ from  \eqref{eq:ProofWelldefinedOperatorDefinitionH^XFtimesYC} is contained in $\lXYSub{\cx\oplus \cy}$. 
		This follows by the following calculation%
		\begin{align}
		\norm{\hb^{\XF\times \YC}}_\lXYSub{\cx\oplus \cy}&\leq \norm{\frac{\pib^\lambda}{\lambda}\odot \Ab_*^T\Big( \hb^{\XF, \infty}, \hb_{*}^{\YC, \infty}\Big)}_\lXYSub{\cx\oplus \cy}\notag\\
		& +\norm{\frac{\pib^\lambda}{ \rb \otimes  \sb }\odot \Big[   \rb \otimes \hb^\YC + \hb^\XF \otimes  \sb \Big]}_\lXYSub{\cx\oplus \cy}. \notag%
		\end{align}
		Invoking H\"older's inequality asserts 
		\begin{align}
			\norm{\hb^{\XF\times \YC}}_\lXYSub{\cx\oplus \cy}\leq\;\; 2 \norm{\frac{\pib^\lambda}{\lambda}}_\lXYSub{\cx\otimes\kySup{}} \Big( \norm{\hb^{\XF, \infty}}_{\lInfXSub{}} + \norm{\hb^{\YC, \infty}_*}_{\lInfYYSub{\eySup{}}} \Big) \notag\\
		 \norm{\frac{\pib^\lambda}{ \rb  \otimes  \sb }}_{\lInfXYSub{1_{\XC}\otimes\eySup{}}}2\bigg( \norm{ \rb }_{\lXSub{\cx}}\norm{\hb^\YC}_{\lYSub{\kySup{}}}  + \norm{\hb^\XF}_{\lXSub{\cx}}\norm{ \sb }_{\lYSub{\kySup{}}} \bigg), \label{eq:UpperboundNormPiComponent}
		\end{align}
		where we used in the second line for $(x,y) \in \XC\times \YC$ the inequality  $$\big(\cx(x) + \cy(x)\big) \eySup{}(y)  =  \cx(x)\eySup{}(y) + \kySup{}(y) \leq 2 \cx(x) \kySup{}(y).$$
		In particular, it holds by the upper bound for $\pib^\lambda$ (Proposition \ref{prop:BoundsOptimalPotentials}) that 
		\begin{align*}
			\|\pib^\lambda\|_\lXYSub{\cx\otimes\kySup{}} &= \sum_{(x,y) \in \XC\times \YC}\cx(x)\kySup{}(y) \pi^\lambda_{xy} \leq \norm{\exSup{}}^2_{\lInfX}\norm{\rb}_{\lXSub{\cx}} \norm{\sb}_{\kySup{4 }},
		\end{align*}
		which is finite since $(\rb,\sb)\in \lXSub{\cx}\times \lYSub{\kySup{4 }}$. By the same upper bound we see that the term $\norm{\pib^\lambda/ \, \rb  \otimes  \sb}_{\lInfXYSub{1_{\XC}\otimes\eySup{}}}$ is bounded by
		$$%
		\sup_{(x,y) \in \XC\times \YC} \frac{\pib^\lambda_{xy}/(r_xs_y)}{\eySup{}(y)} \leq \norm{\exSup{}}^2_\lInfX \langle \eySup{}, \sb\rangle. 
		$$
		This also shows that $\hb^{\XF\times \YC}$ continuously depends on $(\hb^{\XF, \infty}, \hb^{\YC, \infty}_*) \in \lInfXSub{}\times \lInfYSub{\eySup{}}$ and $(\hb^{\XF}, \hb^{\YC}_*)\in \lXSub{\cx}\times \lYSub{\kySup{4 }}$ and thus concludes the proof on well-definedness and boundedness of the operator of the claim.
		\end{proof}
		\begin{remark}[On costs with unbounded variation in both arguments]\label{rem:IssuesUnboundedOperator}
		Crucial for the well-definedness of the proposed derivative for the EROT plan (Proposition \ref{prop:ProposedDerivativeIsBounded}) for the setting $\norm{\cxpb- \cxmb}_{\lInfX}< \infty$ is that the operator $\AC \colon \lInfX\times \lInfYYSub{\eySup{}} \rightarrow \lInfX\times \lInfYYSub{\eySup{}}$ from the l.h.s. of \eqref{eq:LinearEquationForInverse} has a continuous inverse for which we employ the Neumann series calculus. To this end, we verify that there exists $\epsilon >0$ and construct a function $\teySup{} \colon \YC \rightarrow [1,\infty)$ with $\teySup{} \geq k \eySup{}$ for some  $k>0$ such that  \begin{align*}
			\sum_{y \in \YCC} \teySup{}(y)\frac{\pi^\lambda_{xy}}{ r_x} &\leq 1- \epsilon && \!\!\!\!\!\!\!\!\!\!\!\!\!\!\!\!\!\!\!\!\!\!\!\!\!\!\!\!\!\!\text{ for all }x \in \XC,  \\
			\frac{1}{\teySup{}(y)}\sum_{x \in \XC}\frac{\pi^\lambda_{xy}}{ s_y} &\leq  1&& \!\!\!\!\!\!\!\!\!\!\!\!\!\!\!\!\!\!\!\!\!\!\!\!\!\!\!\!\!\!\text{ for all }y \in \YCC.
		\intertext{	Generalizing this approach to ground costs with unbounded variation in both components $\norm{\cxpb- \cxmb}_{\lInfX}= \infty$ and $\norm{\cypb- \cymb}_{\lInfY}= \infty$ would require existence of $\epsilon>0$ and suitable functions $\psi_\XC\colon \rightarrow [1,\infty), \psi_\YC\colon \rightarrow [1,\infty)$ with $\psi_\XC \geq k \exSup{}$ and $\psi_\YC \geq k \eySup{}$ for some $k >0$ such that}
			\frac{1}{\psi_\XC(x)}\sum_{y \in \YCC} \psi_\YC(y)\frac{\pi^\lambda_{xy}}{ r_x} &\leq 1- \epsilon &&\!\!\!\!\!\!\!\!\!\!\!\!\!\!\!\!\!\!\!\!\!\!\!\!\!\!\!\!\!\!\text{ for all }x \in \XC,  \\
			\frac{1}{\psi_\YC(y)}\sum_{x \in \XC}\psi_\XC(x)\frac{\pi^\lambda_{xy}}{ s_y} &\leq 1&&\!\!\!\!\!\!\!\!\!\!\!\!\!\!\!\!\!\!\!\!\!\!\!\!\!\!\!\!\!\!\text{ for all }y \in \YCC.
			\end{align*}
		\end{remark}

		\begin{proof}[Proof of Proposition \ref{prop:DerivativeFiniteSupportPerturbation}]
			For the proof we rewrite in Step 1 the optimality criterion from Lemma \ref{lem:RewrittenOptimalityCriterionWithFunctionF} and obtain an equality consisting of the proposed derivative and additional error terms. In Step 2, we conclude using Appendix \ref{subsec:App:TechnicalLemmas} that these error terms tend to zero.

			\emph{Step 1 - Rewriting of optimality criterion (Lemma \ref{lem:RewrittenOptimalityCriterionWithFunctionF}).}
		Define the mapping from probability measures to optimizers of \eqref{eq:EntropicOptimalTransport} and \eqref{eq:DualEntropicOptimalTransportProblem}
		\begin{equation*}
		  \begin{aligned}\varthetab\colon \big(\probset{\XF} &\cap \lXSub{\cx}\big) \times \big(\probset{\YC}_*\cap \lYYSub{\kySup{4}}\big) \rightarrow \probset{\XF\times \YC} \times\R^\XC\times\R^{\YCC}, \\
		&\quad \quad \quad (\rb, \sb_*) \mapsto \Big(\pib^\lambda(\rb, \sb), \alphab^\lambda(\rb, \sb), \betab_*^\lambda(\rb, \sb)\Big),
		\end{aligned}%
		\end{equation*}
		where we select $(\alpha^\lambda, \beta^\lambda_*)$ according to Proposition \ref{prop:ConvergenceOptimalDualSolutions}, i.e., such that the element $(0,\beta^\lambda_*)\in \R^\YC$ represents a dual optimizer. %
		Recalling the function $\FC$ from Section \ref{sec:SensitivityAnalysis}, it holds by Lemma \ref{lem:RewrittenOptimalityCriterionWithFunctionF} for each $n\geq N$ that 
		$$
				0=\FC\big(\varthetab({\rb}, {\sb}_*), {\rb}, {\sb}_*\big) =\FC\big(  \varthetab({\rb}+t_n\hat{\hb}^\XF_l, {\sb}_*+t_n\hat{\hb}^\YC_{*,l}), {\rb}+t_n\hat{\hb}^\XF_l, {\sb}_*+t_n\hat{\hb}^\YC_{*,l}\big),
				$$
			which yields
			\begin{align*}
				0 =& \;\;\;\; \bigg(\FC\big(\varthetab({\rb}, {\sb}_*), {\rb}, {\sb}_*\big)-
				\FC\big(  \varthetab({\rb}, {\sb}_*), {\rb}+t_n\hat{\hb}^\XF_l, {\sb}_*+t_n\hat{\hb}^\YC_{*,l}\big)\bigg) \\
				 &-\bigg(\FC\big(  \varthetab({\rb}+t_n\hat{\hb}^\XF_l, {\sb}_*+t_n\hat{\hb}^\YC_{*,l}), {\rb}+t_n\hat{\hb}^\XF_l, {\sb}_*+t_n\hat{\hb}^\YC_{*,l}\big) \\
				 & \quad -   
				\FC\big(  \varthetab({\rb}, {\sb}_*), {\rb}+t_n\hat{\hb}^\XF_l, {\sb}_*+t_n\hat{\hb}^\YC_{*,l}\big)\bigg).
				\end{align*}
			Adding another three terms of non-trivial zeros leads to the following equation \begin{align}
				& \left[\DCW_{\pib, \alphab, \betab_*|(\varthetab( \rb, {\sb}_*),  \rb, {\sb}_* )}\FC\right]\left( \varthetab(\rb+t_n\hat{\hb   }^\XF_l, \sb_*+t_n\hat{\hb   }^\YC_{*,l}) - \varthetab(\rb, \sb_*) \right)\notag\\ 
				&=- t_n\left[\DCW_{\rb, \sb_*|(\varthetab( \rb, {\sb}_*),  \rb, {\sb}_*)}\FC\right] (\hat\hb^{\XF}_l, \hat \hb^{\YC}_{*,l})\label{eq:FinitePerturbationProof_1}
				\\ 
				&+  \bigg(\FC\big(\varthetab({\rb}, {\sb}_*), {\rb}, {\sb}_*\big)- \FC\big(  \varthetab({\rb}, {\sb}_*), {\rb}+t_n\hat{\hb}^\XF_l, {\sb}_*+t_n\hat{\hb}^\YC_{*,l}\big) \label{eq:FinitePerturbationProof_2}\\
				&\quad + t_n\left[\DCW_{\rb, \sb_*|(\varthetab( \rb, {\sb}_*),  \rb, {\sb}_*)} \FC\right](\hat\hb^{\XF}_l, \hat \hb^{\YC}_{*,l}) \bigg) \notag\\
				 &+ \bigg(   
				\FC\big(  \varthetab({\rb}, {\sb}_*), {\rb}+t_n\hat{\hb}^\XF_l, {\sb}_*+t_n\hat{\hb}^\YC_{*,l}\big)\label{eq:FinitePerturbationProof_3} \\
				& \quad - \FC\big(  \varthetab({\rb}+t_n\hat{\hb}^\XF_l, {\sb}_*+t_n\hat{\hb}^\YC_{*,l}), {\rb}+t_n\hat{\hb}^\XF_l, {\sb}_*+t_n\hat{\hb}^\YC_{*,l}\big)\notag\\
				& \quad + \left[\DCW_{\pib, \alphab, \betab_*|(\varthetab( \rb, {\sb}_*),  \rb + t_n \hat{\hb   }_l^\XF, {\sb}_* + t_n \hat{\hb   }_{*,l}^\YC)}\FC\right]\left( \varthetab(\rb+t_n\hat{\hb}^\XF_l, \sb_*+t_n\hat{\hb}^\YC_{*,l}) - \varthetab(\rb, \sb_*) \right)\bigg) \notag \\
				& + \bigg( \left[\DCW_{\pib, \alphab, \betab_*|(\varthetab( \rb, {\sb}_*),  \rb , {\sb}_*)}\FC\right]\left( \varthetab(\rb+t_n\hat{\hb}^\XF_l, \sb_*+t_n\hat{\hb}^\YC_{*,l}) - \varthetab(\rb, \sb_*) \right)\label{eq:FinitePerturbationProof_4}\\
				&\quad - \left[\DCW_{\pib, \alphab, \betab_*|(\varthetab( \rb, {\sb}_*),  \rb + t_n \hat \hb_l^\XF, {\sb}_* + t_n \hat{\hb   }_{*,l}^\YC)}\FC\right]\left( \varthetab(\rb+t_n\hat{\hb}^\XF_l, \sb_*+t_n\hat{\hb   }^\YC_{*,l}) - \varthetab(\rb, \sb_*) \right)\bigg),\notag
				\end{align}
			where $\D_{\rb, \sb_*}\FC$, $\D_{\pib, \alphab, \betab_*}\FC$ represent the na\" ive component-wise derivatives as employed in Section \ref{sec:SensitivityAnalysis}. 
			 For the term in \eqref{eq:FinitePerturbationProof_1} we already know by Proposition \ref{prop:ProposedDerivativeIsBounded} that applying $\big[\DCW_{\pib, \alphab, \betab_*|(\varthetab( \rb, {\sb}_*),  \rb, {\sb}_* )}\FC\big]^{-1}$ is well-defined.
			  
			 \emph{Step 2 - Convergence of difference quotient to asserted derivative.}
			  We need to show that applying $\big[\DCW_{\pib, \alphab, \betab_*|(\varthetab( \rb, {\sb}_*),  \rb, {\sb}_* )}\FC\big]^{-1}$ on each of the summands \eqref{eq:FinitePerturbationProof_2}, \eqref{eq:FinitePerturbationProof_3},  and \eqref{eq:FinitePerturbationProof_4} is also well-defined and that the $\lXYSub{\cx\oplus\cy}$-norm of the resulting component in $\lXYSub{\cx\oplus\cy}$, i.e., the $\pib$-component decreases with order $o(t_n)$ for $n\rightarrow \infty$. This part of the proof is technical and deferred to Lemma \ref{lem:FinishingLemma} in Appendix \ref{subsec:App:TechnicalLemmas}. Most notably, for this purpose we require $\sb\in \lXSub{\kxSup{4}}$.
			As a consequence, we obtain for $n\rightarrow \infty$ that \begin{align*}&\;\bigg\| \frac{\pib^\lambda(\rb + t_n\hat{\hb}^\XF_l,{\sb} + t_n\hat{\hb}^\YC_{l}) - \pib^\lambda(\rb,{\sb})}{t_n}  - \DH_{| \rb, {\sb} }\pib^\lambda(\hat\hb^{\XF}_l, \hat \hb^{\YC}_{l})\bigg\|_{\lXYSub{\cx\oplus\cy}}\\ 
			 =&\;\Bigg\| \frac{\varthetab_{\pib}(\rb + t_n\hat{\hb}^\XF_l,{\sb}_* + t_n\hat{\hb}^\YC_{*,l}) - \varthetab_{\pib}(\rb,{\sb}_*)}{t_n} \\
				&\;+ \bigg(\left[\DCW_{\pib, \alphab, \betab_*|(\varthetab( \rb, {\sb}_*),  \rb, {\sb}_* )}\FC\right]^{-1}_{\pib}\circ \DCW_{\rb, \sb_*|(\varthetab( \rb, {\sb}_*),  \rb, {\sb}_*)} \FC\bigg)(\hat\hb^{\XF}_l, \hat \hb^{\YC}_{*,l}) \Bigg\|_{\lXYSub{\cx\oplus\cy}} \\
				=& \; o(1),\end{align*}
				which proves the assertion.
			\end{proof}

		\begin{proof}[Proof of Proposition \ref{prop:TransportPlanIsLipschitz}]	
			The proof consists of three steps. After introducing in Step~1 some notation for finitely supported measures, we show in Step 2 and 3 local Lipschitz properties for the EROT potentials and plan, respectively, for finitely supported probability measures. In Step 4 we lift the bounds to countably supported measures. Overall, we prove that the operator norm of the proposed derivative from Proposition \ref{prop:ProposedDerivativeIsBounded} can be uniformly bounded for pairs of probability measures in a sufficiently small neighborhood of $(\rb,\sb)$. %

			\emph{Step 1 - Conventions and notation.}
			We first set $\tau \coloneqq \sup_{(\tilde\rb, \tilde\sb) \in B_{1}}\langle \eySup{}, \tilde \sb \rangle\in (0,\infty) $ and define \begin{align*}\kappa&\coloneqq 1+\tau\norm{\exSup{}}^2_\lInfX , \quad 	
		  \eta \coloneqq \left(\tau^{2}\norm{\exSup{}}^{3}_{\lInfX}\right)^{-1},\quad \rho_0\coloneqq \frac{\eta}{4\kappa}.%
		\end{align*}
			Further, consider $(\hat \rb, \hat \sb)\in B_{\rho_0}$ with finite support and let $\hat \XC \coloneqq \supp(\hat\rb)$, $\hat \YC \coloneqq \supp(\hat\sb)$. Note that the inequality $\hat s_{y_1}\geq s_{y_1}/2>0$ holds, which asserts $y_1\in \hat \YC$. %
			For a positive function $f \colon \XC \rightarrow (0,\infty)$ define the Banach  spaces $\ell^1_f(\tXC)\coloneqq \ell^1_{f\vert \tXC}(\tXC)$ and $\ell^\infty_f(\tXC)\coloneqq \ell^\infty_{f\vert \tXC}(\tXC)$. With this notation, we define the operator $\hat \Ab_*$ as $\Ab_*$ from Section~\ref{sec:SensitivityAnalysis} restricted to $\ell^1_{\cx}(\tXC)\times \ell^1_{\kySup{4 }}(\tYCC)$ and, similarly, introduce $\hat \FC$ as $\FC$ from Section~\ref{sec:SensitivityAnalysis} with a modified domain and range space
			\begin{align*}\hat{\FC} \!\colon \!\! \!\Big( \!\ell^1_{\cx \oplus \cy}(\tXC\! \times \!\tYC)\!\times\! \ell^\infty_{}(\tXC)\!\times\! \ell^\infty_{\eySup{}}(\tYCC) \!\Big)\!\! \times \!\!\Big( \!\ell^1_{\cx}(\tXC)\!\times \!\ell^1_{\kySup{4 }}(\tYCC)\! \Big) \!\!
		\rightarrow  \!\R^{|\tXC\times \tYC| + |\tXC| + |\tYC| - 1}.
		\end{align*}

		\emph{Step 2 - Lipschitz modulus of EROT potentials for finitely supported $(\hat \rb,\hat \sb)\in B_{\rho_0}$.}
		By Lemma \ref{lem:RewrittenOptimalityCriterionWithFunctionF} the triplet $(\hat\pib^\lambda, \hat\alphab^\lambda, \hat\betab^\lambda)\in  \Big( \ell^1_{\cx \oplus \cy}(\tXC \times \tYC)\times \ell^\infty_{}(\tXC)\times \ell^\infty_{\eySup{}}(\hat \YC) \Big)$ are optimizers of \eqref{eq:EntropicOptimalTransport} and \eqref{eq:DualEntropicOptimalTransportProblem} for the probability measures $(\hat\rb, \hat\sb)$ if and only if $\hat{\FC}\big((\hat\pib^\lambda, \hat\alphab^\lambda, \hat\betab^\lambda_*),(\hat\rb, \hat\sb)\big) = 0$.
			Furthermore, the function $\hat{\FC}$ is Fr\'echet differentiable\footnote{A map $f\colon U\rightarrow V$ between normed vector spaces is Fr\'echet differentiable if for any $u\in U$ there is a bounded linear map $A\colon U \rightarrow V$, the Fr\'echet derivative, such that $\lim_{\norm{h}_U\!\!\!\searrow\, 0 }\norm{f(u+h) - f(u) - A(h)}_{V}/\norm{h}_U = 0$. It is a stronger notion than Hadamard differentiability, see \citet[Chapter~3]{Averbukh1967}.}, the derivative in this notion will be denoted by $\DF$. Following the arguments by \cite{klatt2018empirical} the partial derivative of $\hat\FC$ with respect to $(\hat\pib, \hat\alphab, \hat\betab_*)$ at optimizers $(\hat\pib^\lambda, \hat\alphab^\lambda, \hat\betab^\lambda_*)$ for $(\hat \rb, \hat \sb)$ in matrix representation is then given by 
			$$\begin{aligned}
			 & \;\left[\DF_{\hat\pib, \hat\alphab, \hat\betab_*|( \hat\pib^\lambda, \hat\alphab^\lambda, \hat\betab_*^\lambda, \hat \rb, \hat\sb_* )}\hat\FC\right] \\
			= &\; \begin{pmatrix}
				\Id_{\ell^1_{\cx \oplus \cy}(\tXC \times \tYC)} &\quad  \frac{1}{\lambda}\exp\left( \frac{1}{\lambda}\left[\hat\Ab_*^T\big(\hat\alphab^\lambda, \hat\betab^\lambda_*\big) - \cb \right]\right)\odot (\hat\rb\otimes \hat\sb)\odot \hat\Ab_*^T\\
				\hat\Ab_*& 0
			\end{pmatrix},\end{aligned} $$ 
			which is an invertible operator since the identity operator is invertible in conjunction with $\Ab_*$ having full rank of order $|\hat\XC|+|\hat\YC| - 1$, and because $\frac{1}{\lambda}\exp\left( \frac{1}{\lambda}\left[\Ab_*^T\big(\hat\alphab^\lambda, \hat\betab^\lambda_*\big) - \cb \right]\right)\odot (\hat\rb\otimes \hat\sb)$ is component-wise strictly positive. By the implicit function theorem this induces a mapping on an open set $\mathcal{U} \subseteq \ell^1_{\cx}(\tXC)\times \ell^1_{\kySup{4 }}(\tYCC)$ with $(\hat \rb, \hat \sb)\in \mathcal{U}$ 
			$$\hat\vartheta \colon \mathcal{U} \rightarrow \Big( \ell^1_{\cx \oplus \cy}(\tXC \times \tYC)\times \ell^\infty_{\exSup{}}(\tXC)\times \ell^\infty_{\eySup{}}(\tYCC) \Big)$$
			such that for any $(\overline \rb, \overline \sb_*) \in \mathcal{U}$ the relation $\hat\FC(\hat\vartheta(\overline \rb, \overline \sb), (\overline \rb, \overline \sb)) = 0$ holds. In particular, for $(\overline \rb, \overline \sb_*)\in \mathcal{U}\cap \mathcal{P}(\hat \XC)\times \mathcal{P}(\hat\YC)_*$ it follows that $\hat\vartheta(\overline \rb, \overline \sb)$ coincides with  the triplet of EROT optimizers $(\hat\pib^\lambda, \hat\alphab^\lambda, \hat\betab^\lambda_*)$ for these respective probability measures. 
		  Further, the implicit function theorem yields that $\hat\vartheta$ is Fr\'echet differentiable at $(\hat{\rb},\hat{\sb})$ with derivative 
			$$\DF_{|(\hat{\rb},\hat{\sb}_*)} \vartheta = -\left[\DF_{\pib, \alphab, \betab_*|( \varthetab(\hat \rb, \hat\sb_*), \hat \rb, \hat\sb_* )}\hat \FC\right]^{-1} \circ  \DF_{\rb, \sb_*|( \varthetab(\hat \rb, \hat\sb_*), \hat \rb, \hat\sb_*) }\hat\FC.$$
			Hence, it remains to bound this operator. Adapting the notation of the proof for Proposition \ref{prop:ProposedDerivativeIsBounded} we know that there exist suitable operators $\hAC,\hBC$ such that the derivative for the component of $\hat \vartheta$ in  $\ell^\infty_{}(\tXC)\times \ell^\infty_{\eySup{}}(\hat \YC)$, i.e., the component for EROT potentials, is given by $\hAC^{-1}\hBC$. For the operator $\hat Q$ we know that $$\begin{aligned}\norm{\hat Q}_{OP} &\leq \lambda(1 + \kySup{4 }(y_1))  \norm{\exSup{}}_{\lInfX}^2\langle \eySup{}, \sb\rangle \leq \lambda (1+ \kySup{4 }(y_1) )\kappa  \eqqcolon \Lambda_{1} < \infty.
		\end{aligned}
		$$
			For a bound on the operator norm of $\hAC$ we obtain by 
			Proposition \ref{prop:BoundsOptimalPotentials} the lower bound  $\min_{x \in \tXC}\pi^\lambda_{xy_1}(\hat \rb, \hat \sb)/\hat r_x \geq \eta$. Moreover, we choose $N\in \N$ such that %
			\begin{align*}
				\sum_{i = N+1}^{\infty} \kySup{4 }(y_i) \kappa s_{y_i}\leq \frac{\eta}{4}.
			\end{align*}
			By definition of $\rho_0$ we obtain for all $(\overline\rb, \overline\sb) \in B_{\rho_0}$ that 
			\begin{align*}
				\norm{\overline s - s}_{\lXSub{\kxSup{2}}} =
				\sum_{i = 1}^{\infty} \kySup{4 }(y_i)   | \overline s_{y_i} -s_{y_i}|\leq \rho_0 = \frac{\eta}{4 \kappa},%
			\end{align*}
			which  by our choice on $N$ yields for each $(\overline\rb, \overline\sb) \in B_{\rho_0}$ and  $x \in \XC$ that 
			\begin{align*}
			\sum_{i = N+1}^{\infty} \big(\eySup{}(y_i) - 1\big)\frac{\pi^\lambda_{xy_i}(\overline\rb, \overline\sb)}{\overline r_x}
				\leq &\;\!\!\!\sum_{i = N+1}^{\infty} \big(\eySup{}(y_i) - 1\big)\eySup{}(y_i)\norm{\exSup{}}_\lInfX^2\langle \eySup{}, \overline s\rangle\overline s_{y_i}\\
				\leq &\; \!\!\!\sum_{i = N+1}^{\infty} \kySup{4 }(y_i) \kappa \overline s_{y_i} \leq\!\!\! \sum_{i = N+1}^{\infty} \kySup{4 }(y_i) \kappa (| \overline s_{y_i} - s_{y_i}|  + s_{y_i}) \leq  \frac{\eta}{2}.
			\end{align*}
			In particular, it follows that $\sum_{i = N+1}^{\infty} \big(\eySup{}(y_i) - 1\big)\pi^\lambda_{xy_i}(\hat \rb, \hat \sb)/\hat r_x\leq \frac{\eta}{2}$ for all $x \in \XC$ and the quantity $L_\YC = \sup_{i = 1, \dots N}\eySup{}(y)$ is finite. Hence, by the Neumann series calculus we obtain that the operator norm of $\hAC^{-1}$ can be bounded by $$\norm{\hAC^{-1}}_{OP} \leq \frac{4L_\YC}{\eta} \eqqcolon \Lambda_2,$$ 
			which yields $\|\hAC^{-1} \hBC\|_{OP} \leq \Lambda_1 \Lambda_2 \eqqcolon \tilde \Lambda'$. By definition, this bound is independent of $(\hat \rb, \hat \sb)$. This means for any two pairs of probability measures $(\tilde\rb, \tilde\sb), (\tilde \rb', \tilde \sb')\in B_{\rho_0}$ with finite, coinciding support it follows that \eqref{eq:LocalLipschitzForDualOptimizer} is valid for the Lipschitz-constant $\tilde \Lambda'$. Moreover, by Proposition \ref{prop:ConvergenceOptimalDualSolutions} we note that \eqref{eq:LocalLipschitzForDualOptimizer} generalizes to the setting of $\supp(\tilde \rb') \subseteq \supp(\tilde \rb)$ and $\supp(\tilde \sb') \subseteq \supp(\tilde\sb)$.
			
			\emph{Step 3 - Lipschitz modulus of EROT plan for finitely supported $(\hat \rb,\hat \sb)\in B_{\rho_0}$.}
			Next, we derive the Lipschitz property for the EROT plan $\pib^\lambda$ in case of finitely supported probability measures. To this end, we again consider the pair $(\hat \rb, \hat \sb)$ and note by \eqref{eq:ProofWelldefinedOperatorDefinitionH^XFtimesYC} from the proof of Proposition \ref{prop:ProposedDerivativeIsBounded} that the derivative for the component of $\hat\vartheta$ in $\ell^1_{\cx\oplus\cy}(\tXC \times \tYC)$, denoted by $\vartheta_{\hat \pib}$, is given~by  
			$$\begin{aligned}\DF_{|\hat \rb, \hat \sb_*}\hat \vartheta_{\hat \pib} \colon &\ell^1_{\cx}(\tXC)\times\ell^1_{\kySup{4 }}(\tYCC) \rightarrow  \ell^1_{\cx\oplus\cy}(\tXC \times \tYC),\\&(\hat h^\XC, \hat h^\YC_*) \mapsto \frac{\pib^\lambda(\hat \rb, \hat \sb)}{\lambda}\odot \hat\Ab_*^T \hAC^{-1}\hBC (\hat\hb^\XC, \hat\hb^\YC_*) + \frac{\pib^\lambda(\hat \rb, \hat \sb)}{\hat \rb \otimes \hat \sb }\odot \Big[  \hat \rb \otimes \hat \hb^\YC + \hat \hb^\XF \otimes \hat \sb \Big].\end{aligned}$$
			Similar to the upper bound for \eqref{eq:UpperboundNormPiComponent} we see that 
			$$\begin{aligned}
				&\norm{\DF_{|\hat \rb, \hat \sb_*}\hat \vartheta_{\hat \pib} }_{OP} \leq \sup_{(\overline\rb, \overline\sb)\in B_{\rho_0}} \bigg(2\lambda^{-1}\norm{\pib^\lambda(\overline\rb, \overline\sb)}_\lXYSub{\cx\otimes\kySup{}}\Lambda' \\
			  &\quad\quad\quad\quad\quad\quad\quad\quad +  2\norm{\frac{\pib^\lambda(\overline\rb, \overline\sb)}{\overline\rb  \otimes  \overline\sb }}_{\lInfXYSub{1_\XC\otimes\eySup{}}}\Big( \norm{\overline \rb }_{\lXSub{\cx}} + \norm{\overline\sb }_{\lYSub{\kySup{}}}\Big)\bigg)\eqqcolon \tilde\Lambda < \infty
			\end{aligned}$$
			is finite. 
			Since the upper bound is independent of $(\hat \rb, \hat \sb)$ we obtain for any two pairs of probability measures $(\tilde\rb, \tilde\sb), (\tilde\rb', \tilde\sb')\in B_{\rho_0}$ with coinciding support the inequality 
			$$ \norm{\pib^\lambda(\tilde\rb, \tilde\sb) - \pib^\lambda(\tilde\rb', \tilde\sb')}_{\lXYSub{\cx\oplus\cy}}
		  \leq \tilde\Lambda\norm{(\tilde\rb, \tilde\sb) - (\tilde\rb', \tilde\sb') }_{\lXSub{\cx} \times\lYSub{\kySup{4 }}}.$$
		  Again invoking Proposition \ref{cor:ConvergenceEROTP}, this inequality also holds in case of $\supp(\tilde \rb) \subseteq \supp( \rb)$ and $\supp(\tilde \sb) \subseteq \supp(\sb)$ and thus finishes the first step of the proof.
			  
		  \emph{Step 4 - Extension of Lipschitz bounds to unbounded support.}
			It remains to show that these Lipschitz bounds extend to pairs $(\tilde\rb, \tilde\sb), (\tilde\rb', \tilde\sb')\in B_{\rho_0}$ with $\supp(\tilde \rb') \subseteq \supp( \rb)$, $\supp(\tilde \sb') \subseteq \supp(\sb)$ where at least one (potentially all) probability measure has infinite support. Concerning the Lipschitz property for the EROT plan we consider finite support approximations, introduced in \eqref{eq:FiniteSupportApproximation}, for probability measures. By Proposition \ref{cor:ConvergenceEROTP} and Lemma \ref{lem:FiniteSupportApproximationIsUniformlyGoodForLargeL} it follows for given $\epsilon>0$ that there exists $l\in \N$ such that $(\hat{\tilde\rb}_l, \hat{\tilde\sb}_l),(\hat{\tilde\rb}_l', \hat{\tilde\sb}'_l) \in B_{\rho_0}$ with  $$\begin{aligned} \norm{\pib^\lambda(\tilde\rb, \tilde\sb)  - \pib^\lambda(\hat{\tilde\rb}_l, \hat{\tilde\sb}_l)}_{\lXYSub{\cx\oplus\cy}}< \frac{\epsilon}{2}, \quad 
			\norm{\pib^\lambda(\tilde\rb', \tilde\sb')  - \pib^\lambda(\hat{\tilde\rb}_l', \hat{\tilde\sb}'_l)}_{\lXYSub{\cx\oplus\cy}}< \frac{\epsilon}{2}. \end{aligned}$$
			By our Lipschitz bounds from Step 3 it then follows that $$\begin{aligned}
			\norm{\pib^\lambda(\tilde\rb, \tilde\sb) - \pib^\lambda(\tilde\rb', \tilde\sb')}_{\cx\oplus\cy} &\leq \norm{\pib^\lambda(\hat{\tilde\rb}_l, \hat{\tilde\sb}_l) - \pib^\lambda(\hat{\tilde\rb}_l', \hat{\tilde\sb}'_l)}_{\cx\oplus\cy} + \epsilon\\
			&\leq \tilde\Lambda \norm{(\hat{\tilde\rb}_l, \hat{\tilde\sb}_l) - (\hat{\tilde\rb}_l', \hat{\tilde\sb}'_l)}_{\lXSub{\cx} \times\lYSub{\kySup{4 }}} + \epsilon \\
			& \leq  \tilde\Lambda (\cx(x_1) + \kySup{4 }(y_1)) \norm{(\tilde\rb, \tilde\sb) - (\tilde\rb', \tilde\sb') }_{\lXSub{\cx} \times\lYSub{\kySup{4 }}} + \epsilon.
		\end{aligned}$$	
		As $\epsilon>0$ can be chosen arbitrarily small we deduce the local Lipschitz property with modulus $\Lambda \coloneqq\tilde\Lambda (\cx(x_1) + \kySup{4 }(y_1))$.
		 
		 For the EROT potentials, we only state the proof for $\alphab^\lambda$, the corresponding argument for $\betab^\lambda$ is analogous. 
			Consider $x \in \supp(\tilde \rb')$, then it follows by Proposition \ref{prop:ConvergenceOptimalDualSolutions} for given $\epsilon>0$ that there exists $l\in \N$ such that  $|\alpha^\lambda_{x}(\tilde\rb,\tilde\sb) -\alpha^\lambda_{x}(\hat{\tilde\rb}_l, \hat{\tilde\sb}_l)| < \epsilon/2$ and $|\alpha^\lambda_{x}(\tilde\rb',\tilde\sb') - \alpha^\lambda_{x}(\hat{\tilde\rb}'_l, \hat{\tilde\sb}'_l)| < \epsilon/2$ as well as $(\hat{\tilde\rb}_l, \hat{\tilde\sb}_l),(\hat{\tilde\rb}'_l, \hat{\tilde\sb}'_l) \in B_{\rho_0}$. Applying our Lipschitz bound form Step 2 then yields %
			 \begin{align*}
				|\alpha^\lambda_{x}(\tilde\rb, \tilde\sb) - \alpha^\lambda_{x}(\tilde \rb', \tilde \sb')| &\leq |\alpha^\lambda_{x}(\hat{\tilde \rb}_l, \hat{\tilde\sb}_l) - \alpha^\lambda_{x}(\hat{\tilde \rb}'_l, \hat{\tilde \sb}'_l)| + \epsilon  \\
				&\leq \tilde\Lambda' \norm{(\hat{\tilde\rb}_l, \hat{\tilde\sb}_l) - (\hat{\tilde\rb}'_l, \hat{\tilde\sb}'_l) }_{\lXSub{\cx} \times\lYSub{\kySup{4 }}} + \epsilon \\
			&\leq  \tilde\Lambda' (\cx(x_1) + \kySup{4 }(y_1)) \norm{(\tilde\rb, \tilde\sb) - (\tilde\rb', \tilde\sb') }_{\lXSub{\cx} \times\lYSub{\kySup{4 }}} + \epsilon.
			\end{align*}
			Choosing $\epsilon$ arbitrarily small and setting $\Lambda'\coloneqq \tilde\Lambda' (\cx(x_1) + \kySup{4 }(y_1)) $ gives $$ |\alpha^\lambda_{x}(\tilde\rb, \tilde\sb) - \alpha^\lambda_{x}(\tilde \rb', \tilde \sb')| \leq \Lambda' \norm{(\tilde\rb, \tilde\sb) - (\tilde\rb', \tilde\sb') }_{\lXSub{\cx} \times\lYSub{\kySup{4 }}}.$$
			Taking the supremum over $x \in \supp(\tilde\rb')$ then proves the claim.
		\end{proof}

	\section{Lemmata for Hadamard Differentiability of Entropic \\Optimal Transport Plan}\label{subsec:App:TechnicalLemmas}
		For the proof of Lemma \ref{lem:FinishingLemma} we employ the following result on invertibility of the operator $\big[\DCW_{\pib, \alphab, \betab_*|(\varthetab( \rb, \sb_*),  \rb, \sb_* )}\FC\big]^{-1}_\pib$ and the norm of the resulting element. 
	
	\begin{lemma}\label{lem:InverseOfDerivativeOfFIsWellDefined}
		Let  $ \rb\in \PC(\XC)\cap\lXSub{\cx}$ and $ \sb\in \PC(\YC)\cap\lYSub{\kySup{4}}$ be two probability measures with full support and consider a monotone, possibly unbounded, function $\psi\colon \N \rightarrow [1, \infty)$ such that $\sum_{i = 1}^{\infty}{\psi(i) \kySup{4}(y_i)  s_{y_i}} < \infty.$ Further, define the function %
		\begin{equation} \Gamma \colon \R^{\XC\times \YC} \rightarrow [0, \infty],\quad \xib \mapsto \max\left\{\sup_{x \in \XF} \sum_{y \in \YC}\frac{|\xi_{xy}|}{  r_{x} }, \sup_{y_i\in \YC}\sum_{x \in \XF}\frac{|\xi_{xy_i}|}{ s_{y_i} \psi(i)\eySup{3}(y_i)} \right\}\notag\label{eq:conditionForInverseOfDerivativeOfF}\end{equation}
		and consider an element $\xib\in \lXYSub{\cx\oplus \cy}$ with $\Gamma(\xib) < \infty$.
		Then applying the operator $\big[\DCW_{\pib, \alphab, \betab_*|(\varthetab( \rb, \sb_*),  \rb, \sb_* )}\FC\big]^{-1}_\pib$ onto $(\xib, 0, 0)\in \R^{\XC\times \YC}\times \R^\XC\times \R^{\YCC}$ is well-defined and gives an element in $ \lXYSub{\cx\oplus \cy}$.
		In particular, there is $\kappa>0$ independent of $\xib$ such that \begin{align}& \norm{\left[\DCW_{\pib, \alphab, \betab_*|( \varthetab( \rb, \sb_*),  \rb, \sb_* )}\FC\right]^{-1}_\pi(\xib, 0, 0)}_{\lXYSub{\cx\oplus \cy}}
		\leq  
		\kappa \Gamma(\xib).\notag\end{align}
		\end{lemma}
	
	\begin{proof}
	For the proof we first show that a suitable solution exists and then verify a suitable upper bound for that expression.  
	
		\emph{Step 1 - Derivation of solution.}
		We first show that there exists a unique element $(\zetab^{\XF\times \YC}, \zetab^{\XF,\infty},\zetab^{\YC, \infty}_*)\in \lXYSub{\cx \oplus \cy} \times \lInfXSub{}\times \lInfYYSub{\psi\eySup{3}}$ such that 
		$$\begin{pmatrix}
		\zetab^{\XF\times \YC}- \frac{\pib^\lambda}{\lambda}\odot \Ab_*^T\Big( \zetab^{\XF, \infty}, \zetab_{*}^{\YC, \infty}\Big)\\[0.15cm]
		\Ab_*(\zetab^{\XF\times \YC})
	\end{pmatrix} =  \begin{pmatrix}
		\xib\\
		\begin{pmatrix}
			0\\
			0
		\end{pmatrix}
	\end{pmatrix},$$
	where we exploited on the l.h.s. the relation between primal and dual optimizers of \eqref{eq:EntropicOptimalTransport} (Proposition \ref{prop:DualEROT}). By setting \begin{equation}\zetab^{\XF\times \YC}\coloneqq \xib + \frac{\pib^\lambda}{\lambda}\odot \Ab_*^T\Big( \zetab^{\XF, \infty}, \zetab_{*}^{\YC, \infty}\Big)\label{eq:InverseOfFDefinitinoXiXFtimesYC}\end{equation}
	we reduce the system of countably many equalities. Then it remains to find $(\zetab^{\XF,\infty},\zetab^{\YC, \infty}_*)\in \lInfXSub{} \times \lInfYYSub{\psi\eySup{3}}$ such that $$ \Ab_* \left(\frac{\pib^\lambda}{\lambda}\odot \Ab_*^T\Big( \zetab^{\XF, \infty}, \zetab_{*}^{\YC, \infty}\Big)\right) = -\Ab_*(\xib). $$
	This relation means that for all $x\in \XF$ the following equation is valid \begin{align*}
	\frac{\pi^\lambda_{xy_1}}{\lambda}\zeta^{\XF, \infty}_x + \sum_{y \in \YCC} \frac{\pi^\lambda_{xy}}{\lambda}(\zeta^{\XF, \infty}_x + \zeta^{\YC, \infty}_{*,y}) &= -\sum_{y \in \YC} \xi_{xy}\\
	 \frac{ r_x}{\lambda}\zeta^{\XF, \infty}_x + \sum_{y \in \YCC} \frac{\pi^\lambda_{xy}}{\lambda}\zeta^{\YC, \infty}_{*,y} &= -\sum_{y \in \YC} \xi_{xy},
	\end{align*}
	which is equivalent to \begin{equation} \zeta^{\XF, \infty}_x + \sum_{y \in \YCC} \frac{\pi^\lambda_{xy}}{ r_x}\zeta^{\YC, \infty}_{*,y} = -\frac{\lambda}{ r_x}\sum_{y \in \YC} \xi_{xy}.\label{eq:InverseOfFProofEquation1}\end{equation}
	Similarly, we obtain for each $y\in \YCC$ that \begin{align*}
	\sum_{x \in \XF} \frac{\pi^\lambda_{xy}}{\lambda}(\zeta^{\XF, \infty}_x + \zeta^{\YC, \infty}_{*,y}) &= -\sum_{x \in \XF} \xi_{xy},%
	\end{align*}
	which implies that 
	\begin{align}
	\left(\sum_{x \in \XF}\frac{\pi^{\lambda}}{ s_y} \zeta_{x}^{\XC, \infty}\right) + \zeta^{\YC, \infty}_{*,y}  = -\frac{\lambda}{ s_y}\sum_{x \in \XF}\xi_{xy}.\label{eq:InverseOfFProofEquation2}
	\end{align}
	The equalities \eqref{eq:InverseOfFProofEquation1} and \eqref{eq:InverseOfFProofEquation2} can therefore be represented by \begin{equation}
	  \AC(\zetab^{\XC, \infty}, \zetab^{\YC, \infty}_*) = -\lambda \tilde\xib,  \label{eq:PertubationErrorEquationForSolutions}
	\end{equation}
	where $\AC \colon \lInfXSub{}\times \lInfYYSub{\psi\eySup{3}} \rightarrow \lInfXSub{}\times \lInfYYSub{\psi\eySup{3}}$ denotes the operator from the proof of Proposition \ref{prop:ProposedDerivativeIsBounded} with different domain and range. 
	Notably, by our assumption $\Gamma(\xib)< \infty$ it follows that \begin{equation}
	  \left(\sum_{y \in \YC}\frac{ \xi_{{x_1}y}}{ r_{x_1}}, \sum_{y \in \YC} \frac{\xi_{{x_2}y}}{ r_{x_2}}, \dots, \sum_{x \in \XF}\frac{\xi_{x{y_2}}}{ s_{y_2}}, \sum_{x \in \XF}\frac{\xi_{x{y_3}}}{ s_{y_3}}, \dots \right)\in\lInfXSub{}\times \lInfYYSub{\psi\eySup{3}}.\notag%
	\end{equation}
	Moreover, since $\sum_{i = 1}^{\infty}\psi(i)\eySup{3}(y_i)s_{y_i}<\infty,$ it follows by a similar argument as in the proof of Proposition \ref{prop:ProposedDerivativeIsBounded} that $\AC$ has a bounded inverse operator and thus there exists a unique solution $(\zetab^{\XF,\infty},\zetab^{\YC, \infty}_*)\in \lInfXSub{}\times \lInfYYSub{\psi\eySup{3}}$ for equation \eqref{eq:PertubationErrorEquationForSolutions}.
		Hence, we obtain that$$\;\;\norm{(\zetab^{\XF,\infty},\zetab^{\YC, \infty}_*)}_{\lInfXSub{} \times \lInfYYSub{\psi\eySup{3}}} %
		\leq  \norm{\AC^{-1}}_{OP}\lambda \Gamma(\xib).$$
	
		\emph{Step 2 - Upper bound for norm of $\zetab^{\XF\times \YC}$.}
	 It remains to show that $\zetab^{\XF\times \YC}$, as defined in \eqref{eq:InverseOfFDefinitinoXiXFtimesYC}, is contained in $\lXYSub{\cx\oplus \cy}$. Exploiting the upper bound for $\pib^\lambda$ (Proposition \ref{prop:BoundsOptimalPotentials}) yields \begin{align*}&\norm{\zetab^{\XF\times \YC}}_\lXYSub{\cx\oplus\cy} \leq \norm{\xib^{\XF\times \YC}}_\lXYSub{\cx\oplus\cy} +  \norm{\pib^{\lambda}\odot A_*^T(\zetab^{\XF, \infty}, \zetab^{\YC, \infty}_{*})}_\lXYSub{\cx\oplus\cy} \\%
	&\leq  \sum_{(x,y) \in \XC\times \YC}\cx(x)|\xib^{\XC\times \YC}_{xy}| + \sum_{(x,y) \in \XC\times \YC}\cy(y)|\xib^{\XC\times \YC}_{xy}|
	\\
	& + \norm{\pib^{\lambda}}_{\lXYSub{\cx\otimes \psi\kySup{3}}} \cdot 2\norm{(\zetab^{\XF, \infty}, \zetab^{\YC, \infty}_{*})}_{\lInfXSub{} \times \lInfYYSub{\psi\eySup{3}}}\\
	&\leq  \sum_{(x,y) \in \XC\times \YC}\cx(x) r_x\frac{|\xi^{\XC\times \YC}_{xy}|}{r_x} + \sum_{(x,y_i) \in \XC\times \YC}\psi(i)\cy(y_i)\eySup{3}(y_i) s_{y_i} \frac{|\xi^{\XC\times \YC}_{xy_i}|}{\psi(i)\eySup{3}(y_i)s_{y_i}}\\
	& + \norm{\rb}_{\lXSub{\cx}} \norm{\sb}_{\lYSub{\psi\kySup{4}}}\norm{\exSup{}}_{\lInfX}^2\langle\eySup{},\sb\rangle 2 \norm{M^{-1}}_{OP}\lambda \Gamma(\xib)\\
	&\leq \Gamma(\xib) \bigg(\!\!\norm{\rb}_{\lXSub{\cx}} + \norm{\sb}_{\lYSub{\psi\kySup{3}}} + \norm{\rb}_{\lXSub{\cx}} \norm{\sb}_{\lYSub{\psi\kySup{4}}}
	\!\norm{\exSup{}}_{\lInfX}^2\langle\eySup{},\sb\rangle
	2 \norm{M^{-1}}_{OP}\lambda\!\bigg)\\
	&\eqqcolon \Gamma(\xib) \kappa < \infty,
	\end{align*}
	which proves that the claim. %
	\end{proof}
	
	\begin{remark}\label{rem:UnboundedDomainInverseOperator}
		We like to discuss the domain of the operator $\big[\DCW_{\pib, \alphab, \betab_*|(\varthetab( \rb, \sb_*),  \rb, \sb_* )}\FC\big]^{-1}$.
		For an element $\xib\in \lXYSub{\cx\oplus \cy}$ the previous proof shows that the element $(\xib, 0, 0)^T\in \lXYSub{\cx\oplus\cy} \times\lInfXSub{}\times \lInfYYSub{\psi\eySup{3}}$ has a well-defined image under the mapping $\big[\DCW_{\pib, \alphab, \betab_*|(\varthetab( \rb, \sb_*),  \rb, \sb_* )}\FC\big]^{-1}$ in $\lXYSub{\cx\oplus\cy} \times\lInfXSub{}\times \lInfYYSub{\psi\eySup{3}}$ if and only if
		\begin{equation}
	   \left(\sum_{y \in \YC}\frac{ \xi_{{x_1}y}}{ r_{x_1}}, \sum_{y \in \YC} \frac{\xi_{{x_2}y}}{ r_{x_2}}, \dots, \sum_{x \in \XF}\frac{\xi_{x{y_2}}}{ s_{y_2}}, \sum_{x \in \XF}\frac{\xi_{x{y_3}}}{ s_{y_3}}, \dots \right)\in\lInfXSub{}\times \lInfYYSub{\psi\eySup{3}}.\label{eq:GoodEquation}
	\end{equation}
		In particular, by Lemma \ref{lem:ExistenceUnboundedFunctionPhi} one can construct for any $\epsilon>0$ and $\xib$ satisfying \eqref{eq:GoodEquation} an element $\xib' \in \lXYSub{\cx\oplus \cy}$ such that $\norm{\xib - \xib'}_{\lXYSub{\cx\oplus \cy}}< \epsilon$ but where $\xib'$ does not fulfill \eqref{eq:GoodEquation}. Hence, the domain of $\big[\DCW_{\pib, \alphab, \betab_*|(\varthetab( \rb, \sb_*),  \rb, \sb_* )}\FC\big]^{-1}$ does not even contain an open ball in $\lXYSub{\cx\oplus \cy}$. 
	\end{remark}

	We now proceed with the proof of the error bounds from Proposition \ref{prop:DerivativeFiniteSupportPerturbation}.	
	
	\begin{lemma}\label{lem:FinishingLemma}
		Assume the same setting as in Proposition \ref{prop:DerivativeFiniteSupportPerturbation} and denote by $\vartheta$ the mapping as defined in its proof. %
		Then it follows for $n\rightarrow \infty$ that 
		\begin{align}
			& \bigg\|\left[\DCW_{\pib, \alphab, \betab_*|(\varthetab( \rb, {\sb}_*),  \rb, {\sb}_* )}\FC\right]^{-1}_\pib\bigg(\FC\big(\varthetab({\rb}, {\sb}_*), {\rb}, {\sb}_*\big)\label{eq:FinitePerturbationProof_2Proof}\\
			&- \FC\big(  \varthetab({\rb}, {\sb}_*), {\rb}+t_n\hat{\hb}^\XF_l, {\sb}_*+t_n\hat{\hb}^\YC_{*,l}\big) \notag\\
			&+ t_n\left[\DCW_{\rb, \sb_*|(\varthetab( \rb, {\sb}_*),  \rb, {\sb}_*)} \FC\right](\hat\hb^{\XF}_l, \hat \hb^{\YC}_{*,l}) \bigg)\bigg\|_{\lXYSub{\cx\oplus\cy}}\notag \\
			 +& \bigg\|\left[\DCW_{\pib, \alphab, \betab_*|(\varthetab( \rb, {\sb}_*),  \rb, {\sb}_* )}\FC\right]^{-1}_\pib\bigg(   
			\FC\big(  \varthetab({\rb}, {\sb}_*), {\rb}+t_n\hat{\hb}^\XF_l, {\sb}_*+t_n\hat{\hb}^\YC_{*,l}\big)\label{eq:FinitePerturbationProof_3Proof} \\
			&  - \FC\big(  \varthetab({\rb}+t_n\hat{\hb}^\XF_l, {\sb}_*+t_n\hat{\hb}^\YC_{*,l}), {\rb}+t_n\hat{\hb}^\XF_l, {\sb}_*+t_n\hat{\hb}^\YC_{*,l}\big)\notag\\
			&  + \left[\DCW_{\pib, \alphab, \betab_*|(\varthetab( \rb, {\sb}_*),  \rb + t_n \hat{\hb   }_l^\XF, {\sb}_* + t_n \hat{\hb   }_{*,l}^\YC)}\FC\right]\Big( \varthetab(\rb+t_n\hat{\hb}^\XF_l, \sb_*+t_n\hat{\hb}^\YC_{*,l}) \notag\\
			& - \varthetab(\rb, \sb_*) \Big)\bigg) \bigg\|_{\lXYSub{\cx\oplus\cy}}\notag \\
			+& \bigg\| \left[\DCW_{\pib, \alphab, \betab_*|(\varthetab( \rb, {\sb}_*),  \rb, {\sb}_*)}\FC\right]^{-1}_\pib \bigg(\Big[\left[\D_{\pib, \alphab, \betab_*|(\varthetab( \rb, {\sb}_*),  \rb , {\sb}_*)}\FC\right]\label{eq:FinitePerturbationProof_4Proof}\\
			&- \left[\DCW_{\pib, \alphab, \betab_*|(\varthetab( \rb, {\sb}_*),  \rb + t_n \hat \hb_l^\XF, {\sb}_* + t_n \hat{\hb   }_{*,l}^\YC)}\FC\right]\Big]\notag\\
			& \Big( \varthetab(\rb+t_n\hat{\hb}^\XF_l, \sb_*+t_n\hat{\hb   }^\YC_{*,l})  - \varthetab(\rb, \sb_*) \Big)\bigg)\bigg\|_{\lXYSub{\cx\oplus\cy}}\notag=\;  o(t_n).\notag
		\end{align}
	\end{lemma}
	
	\begin{proof}
		 We prove this assertion using Lemma \ref{lem:InverseOfDerivativeOfFIsWellDefined} by showing 
		 for each term on which we apply $\big[\D_{\pib, \alphab, \betab_*|(\varthetab( \rb, {\sb}_*),  \rb, {\sb}_* )}\FC\big]_{\pib}^{-1}$ 
		 that the corresponding quantity $\Gamma(\cdot)$, given some function $\psi$, decreases with order $o(t_n)$. 
		 The resulting image element then tends to zero in $\lXYSub{\cx \oplus \cy}$ with order $o(t_n)$.
			
		\emph{Step 1 - Convergence of first term.} We start with the term in \eqref{eq:FinitePerturbationProof_2Proof}. A straightforward calculation gives \begin{align}
				&\bigg(\FC\big(\varthetab({\rb}, {\sb}_*), {\rb}, {\sb}_*\big)- \FC\big(  \varthetab({\rb}, {\sb}_*), {\rb}+t_n\hat{\hb}^\XF_l, {\sb}_*+t_n\hat{\hb}^\YC_{*,l}\big) \notag\\
				& \quad + t_n\left[\D_{\rb, \sb_*|(\varthetab( \rb, {\sb}_*),  \rb, {\sb}_*)}\FC \right](\hat\hb^{\XF}_l, \hat \hb^{\YC}_{*,l}) \bigg)\notag = t_n^2 \begin{pmatrix}
					\frac{\pib^\lambda({\rb}, {\sb}_*)}{{\rb} \otimes{\sb}}\odot(\hat{\hb}_l^\XF \otimes \hat{\hb}_{l}^\YC)\\
					0 \\
					0
				\end{pmatrix} \eqqcolon \begin{pmatrix}
					\xib_1(n) \\
					0\\
					0
				\end{pmatrix}.\notag
			\end{align}
		Since $\xib_1(n)$ is only finitely supported, it is contained in $\lXYSub{\cx \oplus \cy}$ and, given $\psi\equiv 1$, we note that $\Gamma(\xib_1(n)) = \mathcal{O}(t_n^2) \subseteq o(t_n)$ which yields by Lemma~\ref{lem:InverseOfDerivativeOfFIsWellDefined}$$ \norm{\left[\DCW_{\pib, \alphab, \betab_*|(\varthetab( \rb, {\sb}_*),  \rb, {\sb}_* )}\FC\right]^{-1}_{\pib} (\xib_1(n), 0, 0)}_{\lXYSub{\cx\oplus \cy}} = o(t_n).$$
		
		\emph{Step 2 - Bounds for second term.} 
		For the term of \eqref{eq:FinitePerturbationProof_3Proof} denote the EROT potentials for $({\rb}, {\sb})$ and $({\rb}+t_n\hat{\hb}^\XF_l, {\sb}+t_n\hat{\hb}^\YC_{l})$ by $(\alphab^\lambda, \betab^\lambda)$ and $(\alphab^\lambda_{n},\betab^\lambda_{n})$, respectively,   with $\beta^\lambda_{y_1}= \beta^\lambda_{n,xy_1}=0$. Further, consider the closed ball $B_{\rho_0}( \rb,  \sb)$ from Proposition \ref{prop:DualOptimizerAreLipschitz} and let $n$ be sufficiently large such that $({\rb}+t_n\hat{\hb}^\XF_l,{\sb}_*+t_n\hat{\hb}^\YC_{*,l})\in B_{\rho_0}( \rb,  \sb)$ (Lemma \ref{lem:UpperboundNorms}). This guarantees that dual optimizers are Lipschitz (Proposition \ref{prop:DualOptimizerAreLipschitz}). %
		A straightforward calculation yields 
		\begin{align*}
			&\;\;\bigg[\; \FC\big(  \varthetab({\rb}, {\sb}_*), {\rb}+t_n\hat{\hb}^\XF_l, {\sb}_*+t_n\hat{\hb}^\YC_{*,l}\big) \\
			& \quad - \FC\big(  \varthetab({\rb}+t_n\hat{\hb}^\XF_l, {\sb}_*+t_n\hat{\hb}^\YC_{*,l}), {\rb}+t_n\hat{\hb}^\XF_l, {\sb}_*+t_n\hat{\hb}^\YC_{*,l}\big)\notag\\
			& \quad + \left[\D_{\pib, \alphab, \betab_*|(\varthetab( \rb, {\sb}_*),  \rb + t_n \hat{\hb   }_l^\XF, {\sb}_* + t_n \hat{\hb   }_{*,l}^\YC)}\FC\right]\left( \varthetab(\rb+t_n\hat{\hb   }^\XF_l, \sb_*+t_n\hat{\hb   }^\YC_{*,l}) - \varthetab(\rb, \sb_*) \right)\bigg] \\
			& =\begin{pmatrix}
				 -\big((\rb+t_n\hat{\hb}_l^\XF)\otimes(\sb+t_n\hat{\hb}_l^\YC)\big)\odot \exp\left(\frac{\Ab_*^T(\alphab^\lambda,\: \betab^\lambda_*)-\cb}{\lambda}\right)\odot \Delta\Big(\frac{\Ab_*^T(\alphab^\lambda_{n} - \alphab^\lambda,\: \betab^\lambda_{*,n}-\betab^\lambda_*)}{\lambda}\Big)\\
				 0\\
				 0
			 \end{pmatrix}\\
			&  \eqqcolon \begin{pmatrix}
				 \xib_2(n)\\
				 0\\
				 0
			 \end{pmatrix},
		\end{align*}
		where $\Delta$ is given by $\Delta\colon \R \rightarrow \R_{\geq 0}, x \mapsto \exp(x) - 1 - x$ and is evaluated component-wise for each entry of $\Ab_*^T(\alphab^\lambda_{n} - \alphab^\lambda, \betab^\lambda_{*,n}-\betab^\lambda_*)$. The function $\Delta$ is smooth, non-negative and its derivative vanishes at zero, hence, for a positive null-sequence $x_n \rightarrow 0$ it follows that $\Delta(x_n) = o(x_n)$. In addition, for each $x\geq 0$ the inequality $\Delta(-x)\leq \Delta(x)$ is satisfied.
		 Before we apply Lemma \ref{lem:InverseOfDerivativeOfFIsWellDefined} let us state a component-wise bound for $\xib_2(n)$. 
			 For this purpose, let  $$
			\kappa \coloneqq \norm{\exSup{}}_\lInfX\sup_{n \in \NN}\langle \eySup{}, s+t_n\hat h^\YC_l\rangle< \infty$$
			and recall from Proposition \ref{prop:BoundsOptimalPotentials} that componentwise we have $\exp\left(\lambda^{-1}[A_*^T(\alphab^\lambda,\betab^\lambda_*) - c ] \right)  \leq \kappa(\exSup{}\otimes \eySup{})$.
		 Hence, for all $n\in \N$ it holds componentwise that %
			\begin{align}
			|\xib_2(n)|=& \;\big|(\rb+t_n\hat{\hb}_l^\XF)\otimes(\sb+t_n\hat{\hb}_l^\YC)\big|\odot \exp\left(\frac{\Ab_*^T(\alphab^\lambda,\: \betab^\lambda_*)-\cb}{\lambda}\right) \notag \\ &\quad \quad \odot \left|\Delta\bigg(\frac{\Ab_*^T(\alphab^\lambda_{n} - \alphab^\lambda,\: \betab^\lambda_{*,n}-\betab^\lambda_*)}{\lambda}\bigg)\right|\notag\\
			 \leq & \;\big|(\rb+t_n\hat{\hb}_l^\XF)\otimes(\sb+t_n\hat{\hb}_l^\YC)\big|\odot 
			\kappa\big(\exSup{}\otimes\eySup{}\big) \label{eq:upperBoundSecondErrorTerm_1}\\ &\quad \quad \odot\Delta\bigg(\frac{A_*^T(|\alphab^\lambda_{n} - \alphab^\lambda|, |\betab^\lambda_{*,n} - \betab^\lambda_*|)}{\lambda }\bigg)\notag.
			\end{align}
			Since $\Delta$ is Lipschitz with modulus $\exp(t)$ on the bounded domain $[-t,t]$ for $t\geq 0$ it follows by our bounds for  $|\alphab^\lambda_{n} - \alphab^\lambda|$, $|\betab^\lambda_{*,n} - \betab^\lambda_*|$ from Proposition \ref{prop:BoundsOptimalPotentials} for some constant $\kappa'>0$ and  any 
			$(x,y) \in \XC\times (\YCC)$ that \begin{equation}
				\left|\Delta\bigg(\frac{|\alphab^\lambda_{n,x} - \alphab^\lambda_{x}|+ |\betab^\lambda_{*,n,y} - \betab^\lambda_{*,y}|}{\lambda }\bigg) - \Delta\bigg(\frac{|\alphab^\lambda_{n,x} - \alphab^\lambda_{x}|}{\lambda }\bigg)\right| \leq \kappa' \eySup{}|\betab^\lambda_{*,n,y} - \betab^\lambda_{*,y}|.\label{eq:BoundDeltaFunction}
			  \end{equation}
	This implies componentwise that 
	\begin{align}
		|\xib_2(n)|\leq & \; \big|(\rb+t_n\hat{\hb}_l^\XF)\otimes(\sb+t_n\hat{\hb}_l^\YC)\big|\odot 
		\kappa \big(\exSup{}\otimes\eySup{}\big)\label{eq:upperBoundSecondErrorTerm_2}\\ &\quad \quad \odot A_*^T\bigg(\Delta \bigg(\frac{|\alphab^\lambda_{n} - \alphab^\lambda|}{\lambda }\bigg),\kappa' \eySup{} \frac{|\betab^\lambda_{*,n} - \betab^\lambda_*|}{\lambda}\bigg).\notag %
	\end{align}
		Based on \eqref{eq:upperBoundSecondErrorTerm_2} we thus infer by $( \rb,  \sb)\in \lXSub{\cx}\otimes \lYSub{\kySup{4}}$ that $\xib_2(n) \in \lXYSub{\cx\oplus\cy}$. 
		
		\emph{Step 3 - Convergence of second term.}
		In order to use Lemma \ref{lem:InverseOfDerivativeOfFIsWellDefined} we consider a monotonous unbounded function $\psi\colon \N \rightarrow [1, \infty)$ with  $\sum_{i = 1}^{\infty}{\psi(i) \kySup{4}(y_i)  s_{y_i}} < \infty$. Such a function exists by Lemma \ref{lem:ExistenceUnboundedFunctionPhi}.
		 For this function $\psi$ we now show that $\Gamma(\xib_2(n))= o(t_n)$, so let $I\in \N$ and consider $x\in\XF$. Then it holds that %
		 \begin{align}
			&\;\sum_{y \in \YC}\frac{| r_x + t_n \hat h_{l,x}^\XF| \cdot| s_y + t_n \hat h_{l,y}^\YC |}{{r}_x} \exp\left(\frac{\alpha^\lambda_x+\beta^\lambda_y-c(x,y)}{\lambda}\right)\left|\Delta\left(\frac{\alpha^\lambda_{n,x} - \alpha^\lambda_{x} + \beta^\lambda_{n, y} - \beta^\lambda_{y}}{\lambda}\right)\right|\notag\\
			\leq &\;\sum_{i = 1}^{I}\frac{| r_x + t_n \hat h_{l,x}^\XF| \cdot| s_{y_i} + t_n \hat h_{l,y_i}^\YC |}{{r}_x} \kappa \exSup{}(x) \eySup{}(y_i) \Delta\left(\frac{\left|\alpha^\lambda_{n,x} - \alpha^\lambda_{x}\right| + \left|\beta^\lambda_{n, y_i} - \beta^\lambda_{y_i}\right|}{\lambda}\right)\notag\\
			+ &\;\sum_{i = I+1}^{\infty}\frac{| r_x + t_n \hat h_{l,x}^\XF| \cdot| s_{y_i} + t_n \hat h_{l,y_i}^\YC |}{{r}_x} \kappa \kappa' \exSup{}(x) \psi(y_i)\eySup{3}(y_i)\notag \\
			 &\;\quad \quad \quad\cdot \bigg(\left|\Delta\left(\frac{\alpha^\lambda_{n,x} - \alpha^\lambda_{x}}{\lambda}\right)\right|+\left|\frac{\beta^\lambda_{n, y_i} - \beta^\lambda_{y_i}}{\psi(y_i)\eySup{}(y_i) \lambda}\right|\bigg)\notag\\
			\leq & \; \left(1 + t_n\max_{x \in \{x_1, \dots x_l\}} \frac{|\hat h^\XF_{l,x}|}{{r_x}}\right)\kappa\norm{\exSup{}}_{\lInfX}\big\| \sb + t_n \hat \hb_{l}^\YC\big\|_\lYSub{\eySup{}}\notag\\
			& \quad \quad \quad\cdot \Delta\left(\frac{\norm{\alphab^\lambda_{n} - \alphab^\lambda}_\lInfX + \max_{i = 1}^{I}|\betab^\lambda_{n,y_i} - \betab^\lambda_{y_i}|}{\lambda}\right)\notag\\
			+ & \; \left(1 + t_n\max_{x \in \{x_1, \dots x_l\}} \frac{|\hat h^\XF_{l,x}|}{{r_x}}\right)\kappa \kappa'\norm{\exSup{}}_{\lInfX}\big\| \sb + t_n \hat \hb_{l}^\YC\big\|_\lYSub{\psi\eySup{2}}\Delta\left(\frac{\norm{\alphab^\lambda_{n} - \alphab^\lambda}_\lInfX}{\lambda}\right)\notag\\
			+  & \; \left(1 + t_n\max_{x \in \{x_1, \dots x_l\}} \frac{|\hat h^\XF_{l,x}|}{{r_x}}\right)\kappa \kappa'\norm{\exSup{}}_{\lInfX}\big\| \sb + t_n \hat \hb_{l}^\YC\big\|_\lYSub{\psi\eySup{3}} \frac{\norm{ \betab^\lambda_{n} - \betab^\lambda}_\lInfYSub{\eySup{}}}{\psi(I)\lambda }\notag\\
			= & \; \kappa''\left[\Delta\left(\lambda^{-1}\left(\norm{\alphab^\lambda_{n} - \alphab^\lambda}_\lInfX +  \textup{max}_{i = 1}^{I}|\betab^\lambda_{n,y_i} - \betab^\lambda_{y_i}|\right)\right)+ \psi(I)^{-1}\norm{ \betab^\lambda_{n} - \betab^\lambda}_\lInfYSub{\eySup{}}\right] \notag%
		\end{align}
		for some  $\kappa''>0$ that is independent of $I$ and $x \in \XC$. By the local Lipschitz property (Proposition \ref{prop:DualOptimizerAreLipschitz}) we know that $$\begin{aligned}
		 \norm{\alphab^\lambda_{n} - \alphab^\lambda}_\lInfX \leq t_n\Lambda' %
		\norm{(\hat \hb^\XC_l, \hat \hb^\YC_l)}_{\lXSub{\cx} \times\lYSub{\kySup{4 }}}, \\
		\norm{\betab^\lambda_{n} - \betab^\lambda}_\lInfYSub{\eySup{}} \leq t_n\Lambda' %
		\norm{(\hat \hb^\XC_l, \hat \hb^\YC_l)}_{\lXSub{\cx} \times\lYSub{\kySup{4 }}}.
			\end{aligned}$$ 
		Hence, for $\epsilon >0$ select $I\in \N$ large enough such that $$ \kappa''\frac{\norm{ \betab^\lambda_{n} - \betab^\lambda}_\lInfYSub{\eySup{}}}{\psi(I) } \leq  \psi(I)^{-1}t_n\kappa''\Lambda'\norm{(\hat \hb^\XC_l, \hat \hb^\YC_l)}_{\lXSub{\cx} \times\lYSub{\kySup{4 }}} < \frac{\epsilon t_n}{2},$$ 
		and choose $N\in \N$ sufficiently large such that for all $n \geq N$ holds $$\begin{aligned}
	& \kappa''\Delta\left(\lambda^{-1}\left(\norm{\alphab^\lambda_{n} - \alphab^\lambda}_\lInfX +  \textup{max}_{i = 1}^{I}|\betab^\lambda_{n,y_i} - \betab^\lambda_{y_i}|\right)\right)\\	\leq \;\;& \kappa''\Delta\left(t_n  \lambda^{-1}\Big(1+  \textup{max}_{i = 1}^{I}\exSup{}(i)\Big)\Lambda' \norm{(\hat \hb^\XC_l, \hat \hb^\YC_l)}_{\lXSub{\cx} \times\lYSub{\kySup{4 }}} \right) < \frac{\epsilon t_n}{2}.
	\end{aligned}
	 $$
	 Consequently, we obtain that $\sup_{x \in \XC} \sum_{y \in \YC} |\xi_2(n)_{xy}|/r_x < \epsilon t_n$, and thus it holds that $\sup_{x \in \XC} \sum_{y \in \YC} |\xi_2(n)_{xy}|/r_x = o(t_n)$. To show that $\Gamma(\xi_2(n))= o(t_n)$ let us again consider some integer $I\in \N$. 
	 For $y_i \not \in  \{y_1, \dots, y_I\}$ we employ the upper bound \eqref{eq:upperBoundSecondErrorTerm_2} and see for some $\kappa'''>0$ independent of $y_i$ and $I$ that 
		\begin{align}
			& \sum_{x \in \XF}\frac{| r_x + t_n \hat h_{l,x}^\XF| \cdot| s_{y_i} + t_n \hat h_{l,y_i}^\YC |}{{s}_{y_i}\psi(i)\eySup{3}(y_i)}\notag\\
			& \quad \quad  \exp\left(\frac{\alpha^\lambda_x + \beta^\lambda_{y_i}-c(x,y_i)}{\lambda}\right)\Delta\left(\frac{|\alpha^\lambda_{n,x} - \alpha^\lambda_{x}| + |\beta^\lambda_{n, y_i} - \beta^\lambda_{y_i}|}{\lambda}\right)\notag\\
			\leq &  \left(1 + t_n\max_{y \in \{y_1, \dots y_l\}} \frac{|\hat h^\YC_{l,y}|}{{s}_y}\right) \big\| \rb + t_n \hat \hb_{l}^\XF\big\|_\lXSub{}\kappa'''  \norm{\exSup{2}}_{\lInfX} \frac{|\betab^\lambda_{n,y_i} - \betab^\lambda_{y_i}|}{{\psi(i) \eySup{}(y_i)}}.\label{eq:UpperBounds_GammaXi2Beta}%
		\end{align}
		Note that the first four factors can be bounded by a constant independent of $I$ and $y\in \YC$. For the last factor in  \eqref{eq:UpperBounds_GammaXi2Beta} we employ the local Lipschitz property of dual optimizers with modulus $\Lambda'$ (Proposition \ref{prop:DualOptimizerAreLipschitz}) and obtain $$ \frac{|\betab^\lambda_{n,y_i} - \betab^\lambda_{y_i}|}{{\psi(i) \eySup{}(y_i)}} \leq   \psi(i)^{-1} t_n  \Lambda' \norm{(\hat \hb^\XC_l, \hat \hb^\YC_l)}_{\lXSub{\cx} \times\lYSub{\kySup{4 }}}.$$
		Hence, for $\epsilon>0$ we may choose $I$ sufficiently large such that \begin{equation}
	 \psi(i)^{-1} t_n  \Lambda' \norm{(\hat \hb^\XC_l, \hat \hb^\YC_l)}_{\lXSub{\cx} \times\lYSub{\kySup{4 }}} < \epsilon t_n \quad \text{ for all } i > I.\label{eq:UpperBoundInSmallOt_n_1}
	\end{equation}
	 For each $y_i\in \{y_1, \dots, y_I\}$ we then use the upper bound from \eqref{eq:upperBoundSecondErrorTerm_1} in conjunction with a similar bound as in \eqref{eq:BoundDeltaFunction} and obtain for some $\kappa''''>0$ independent of $y_i$ and $I$ that 
		\begin{align}
			& \sum_{x \in \XF}\frac{| r_x + t_n \hat h_{l,x}^\XF| \cdot| s_{y_i} + t_n \hat h_{l,y_i}^\YC |}{{s}_{y_i}\psi(i)\eySup{3}(y_i)} \exp\left(\frac{\alpha^\lambda_x + \beta^\lambda_{y_i}-c(x,y_i)}{\lambda}\right)\Delta\left(\frac{\alpha^\lambda_{n,x} - \alpha^\lambda_{x} + \beta^\lambda_{n, y_i} - \beta^\lambda_{y_i}}{\lambda}\right)\notag\\
			\leq &  \left(1 + t_n\max_{y \in \{y_1, \dots y_l\}} \frac{|\hat h^\YC_{l,y}|}{{s}_y}\right) \frac{\big\| \rb + t_n \hat \hb_{l}^\XF\big\|_\lXSub{}}{\psi(i) \eySup{3}(y_i)} \kappa''''\norm{\exSup{2}}_{\lInfX}\eySup{}(y_i) \Delta\left(\frac{|\betab^\lambda_{n,y_i} - \betab^\lambda_{y_i}|}{\lambda}\right)\notag\\
			= & \; \landau\big(\Delta(|\betab^\lambda_{n,y_i} - \betab^\lambda_{y_i}|)\big) \subseteq o(|\betab^\lambda_{n,y_i} - \betab^\lambda_{y_i}|) \subseteq o(t_n).\label{eq:UpperBoundInSmallOt_n_2}
		\end{align}
		Here we exploited the local Lipschitz property of $\beta^\lambda$ for the components $y_i \in \{y_1, \dots, y_I\}$.
		As a result we obtain by \eqref{eq:UpperBoundInSmallOt_n_1} and \eqref{eq:UpperBoundInSmallOt_n_2} that $\Gamma(\xib_2(n)) = o(t_n)$. Hence, Lemma \ref{lem:InverseOfDerivativeOfFIsWellDefined}  asserts $$ \norm{\left[\DCW_{\pib, \alphab, \betab_*|(\varthetab( \rb, {\sb}_*),  \rb, {\sb}_* )} \FC \right]_\pi^{-1}(\xib_2 (n), 0, 0)}_{\cx \oplus \cy} = o(t_n).$$ 

		\emph{Step 4 - Convergence of last term. }
		For the last term, i.e., for \eqref{eq:FinitePerturbationProof_4Proof} we obtain\begin{align*}
			&\bigg[ \left[\DCW_{\pib, \alphab,\betab_*|(\varthetab( \rb, {\sb}_*),  \rb , {\sb}_*)}\FC\right]\left( \varthetab(\rb+t_n\hat{\hb}^\XF_l, \sb_*+t_n\hat{\hb}^\YC_{*,l}) - \varthetab(\rb, \sb_*) \right)\notag\\
			&\quad - \left[\DCW_{\pib, \alphab,\betab|(\varthetab( \rb, {\sb}_*),  \rb + t_n \hat \hb_l^\XF, {\sb}_* + t_n \hat{\hb   }_{*,l}^\YC)}\FC\right]\left( \varthetab(\rb+t_n\hat{\hb}^\XF_l, \sb_*+t_n\hat{\hb   }^\YC_{*,l}) - \varthetab(\rb, \sb_*) \right) \bigg]\\
			= & \begin{pmatrix}
				-(t_n\hat{\hb}_l^\XF\otimes \sb + t_n\rb \otimes  \hat{\hb   }_l^\YC + t_n^2\hat{\hb   }_l^\XF\otimes \hat{\hb   }_l^\YC)\odot \exp\left(\frac{\Ab_*^T(\alphab^\lambda,\:\betab^\lambda_*)-\cb}{\lambda}\right)\odot\frac{\Ab_*^T(\alphab^\lambda_{n} -\alphab^\lambda,\: \betab^\lambda_{*,n} - \betab^\lambda_*)}{\lambda}  \\
				0\\
				0
			\end{pmatrix}\\
			\eqqcolon  & \begin{pmatrix}
				\xib_3(n)\\
				0\\
				0
			\end{pmatrix}.
		\end{align*}
		Since $( \rb,  \sb)\in \lXSub{\cx}\otimes \lYSub{\kySup{4}}$ and as $\hat\hb^\XC_l, \hat \hb^\YC_l$ both have finite support, we obtain that $\xib_3(n) \in \lXYSub{\cx\oplus \cy}$.  
		To~finish the proof, we show that $\Gamma(\xib_3(n))= o(t_n)$, where we select $\psi \equiv 1$. %
		For any $x\in \XF$ it then follows using the Lipschitz property for dual optimizers (Proposition \ref{prop:DualOptimizerAreLipschitz}) that \begin{align}
			& \;\;\frac{1}{ r_x} t_n \bigg[\sum_{y\in \YC}\left| \hat h_{l,x}^\XF s_y+ r_x \hat h_{l,y}^\YC+ t_n \hat h_{l,x}^\XF \hat h_{l,y}^\YC\right|\notag\\
			&\quad \quad \quad \quad \;\;\;\;\;\cdot \exp\bigg(\frac{\alpha^\lambda_{x} +\beta^\lambda_{y} - c(x,y)}{\lambda} \bigg)(|\alpha^\lambda_{n,x} - \alpha^\lambda_{x}| + |\beta^\lambda_{n,y} - \beta^\lambda_{y}|)\bigg]\notag\\
			\leq & \;\;t_n \left(\norm{\alphab^\lambda_{n} - \alphab^\lambda}_{\lInfXSub{}} + \norm{\betab^\lambda_{n} - \betab^\lambda}_{\lInfYSub{\eySup{}}}\right)\exp\bigg(\frac{\norm{\cxpb - \cxmb}_{\lInfX}}{\lambda}\bigg) \notag\\
			& \cdot \kappa \left(1 + \max_{x\in \{x_1, \dots, x_l\}}\frac{|\hat h_{l,x}^\XF|}{ r_x} \right) \left[\sum_{y \in \YC}({s}_y + (1+t_n)|\hat h_{l,y}^\YC|)\eySup{2}(y)\right]\notag \\%
			= & \;\;\landau \left(t_n \norm{\alphab^\lambda_{n} - \alphab^\lambda}_{\lInfXSub{}} + t_n\norm{\betab^\lambda_{n} - \betab^\lambda}_{\lInfYSub{\eySup{}}}\right) \subseteq \mathcal{O}(t_n^2)\subseteq o(t_n).\notag %
		\end{align}
		Likewise, we see for all $y \in \YC$ that
		\begin{align}
			& \frac{1}{ s_y\eySup{3}(y)} t_n\bigg[ \sum_{x\in \XC}\left| \hat h_{l,x}^\XF s_y+ r_x \hat h_{l,y}^\YC+ t_n \hat h_{l,x}^\XF \hat h_{l,y}^\YC\right|\notag\\
			&\quad \quad \quad \;\;\;\;\cdot \exp\left(\frac{\alpha^\lambda_{x} +\beta^\lambda_{y} - c(x,y)}{\lambda} \right)\left(|\alpha^\lambda_{n,x} - \alpha^\lambda_{x}| + |\beta^\lambda_{n,y} - \beta^\lambda_{y}|\right)\bigg]\notag\\
			\leq & \;\;t_n \left(\eySup{2}(y)^{-1}\norm{\alphab^\lambda_{n} - \alphab^\lambda}_{\lInfXSub{}} + \norm{\betab^\lambda_{n} - \betab^\lambda}_{\lInfYSub{\eySup{2}}}\right)\exp\bigg(\frac{\norm{\cxpb - \cxmb}}{\lambda}\bigg)\notag\\
			&\;\;\cdot \kappa\left(1 + \max_{y\in \{y_1, \dots, y_l\}}\frac{|\hat h_{l,y}^\YC|}{ s_y} \right) \left[\sum_{x \in \XC}({r}_x + (1+t_n)|\hat h_{l,x}^\XC|)\right]\notag \\%
			= & \;\;\landau \left(t_n \norm{\alphab^\lambda_{n} - \alphab^\lambda}_{\lInfXSub{}} + t_n\norm{\betab^\lambda_{n} - \betab^\lambda}_{\lInfYSub{\eySup{2}}}\right) \subseteq \mathcal{O}(t_n^2)\subseteq o(t_n).\notag 
		\end{align}
		This yields $\Gamma(\xib_3(n))=o(t_n)$ which concludes by Lemma \ref{lem:InverseOfDerivativeOfFIsWellDefined} the proof since $$ \norm{\left[\DCW_{\pib, \alphab, \betab_*|(\varthetab( \rb, {\sb}_*),  \rb, {\sb}_* )}\FC\right]_\pi^{-1} (\xib_3 (n), 0, 0)}_{\lXYSub{\cx\oplus\cy}}= o(t_n).\qedhere$$
	\end{proof}
	\begin{lemma}\label{lem:ExistenceUnboundedFunctionPhi}
		Let $(a_n)_{n \in \N}$ be a non-negative sequence such that $\sum_{i = 1}^{\infty} a_i< \infty$. Then there exists an unbounded and monotonously increasing function $\psi\colon \N\rightarrow[1, \infty)$ such that $\sum_{i= 1}^{\infty}\psi(i) a_i <\infty$.
	\end{lemma}
	
	\begin{proof}
		We define the function $\psi$ by an iterative scheme. Let $\psi_{0}\equiv 1$ and select for any $k \in \N$ the integer  $N_{(k)}\geq \max(k, N_{(k-1)})$ as small as possible such that $$\sum_{i= 1}^{N(k)}\psi_{k-1}(i) a_i + \sum_{i= N(k)+1}^{\infty}2\psi_{k-1}(i) a_i \leq \left(2 - \frac{1}{k}\right)\sum_{i = 1}^{\infty}a_i.$$ We then define the monotone function$$\psi_{k}(i) \coloneqq \begin{cases}
	 \psi_{k-1}(i) & \text{ if }i \leq N(k),\\
	2\psi_{k-1}(i) & \text{ if }i > N(k).
	 \end{cases}
	$$
	Note that $\psi_{k-1}\leq \psi_k$ for each $k \in \N$ and since $N(k)<\infty$ it holds for all $i>N(k)$ that $\psi_{k-1}(i) < \psi_k(i)$. Further, as $N(k)\geq k$ for all $k\in \N$ the sequence $\psi_{k}$ converges entry-wise to a monotone function $\psi$ which is  unbounded and satisfies by monotone convergence $\sum_{i= 1}^{\infty}\psi(i) a_i \leq  2\sum_{i= 1}^{\infty} a_i$.
	\end{proof}

	\section{Auxiliary Results of Finite Support Approximation}\label{subsec:App:FSAProperties}
	In this section we show a number of characteristics  of finite support approximations. Based on the notation in the proof of Theorem \ref{them:TransportPlanIsHadamardDiff}, we consider a probability measure $ \rb\in \lXSub{\cx}$ on $\XC$ with full support. Further, let $(t_n)_{n \in \N}$ be a sequence such that $t_n \searrow 0$ and $(\hb_n^\XF)_{n \in \N}\subseteq \lXSub{\cx}$ with limit $\hb^\XF$ and that $\rb + t_n \hb_n^\XF\in \probset{\XC}\cap \lXSub{\cx}$ for each $n\in \N$. Additionally, for an integer $l \geq 2$ we denote the finite support approximations for $\hb^\XF$ by $\hat\hb_l^\XF$, i.e., labelling 
	$$ \hat h^\XF_{l,x} \coloneqq \begin{cases}
			h^\XF_{x_1} + \sum_{i = l+1}^{\infty}h^\XF_{x_i} & \text{ if } x = x_1, \\
			h^\XF_{x} & \text{ if } x \in \{x_2, \dots, x_l\},\\
			0 & \text{ else.}
		\end{cases} 	
		$$ 

	\begin{lemma}\label{lem:FiniteSupportApproximationIsUniformlyGoodForLargeL}
		For $\delta >0$ there exist $l\in \N$ such that $$\norm{\hb^\XF - \hat\hb^\XF_{l}}_\lXSub{\cx} \leq \delta.$$
		Additionally, there is $N \in \N$ such that for all $n \geq N$ holds $$\norm{\hat\hb^\XF_l - \hat\hb^\XF_{n,l}}_\lXSub{\cx} \leq \delta, \quad \norm{\hb^\XF_n - \hat\hb^\XF_{n,l}}_\lXSub{\cx} \leq \delta.$$
	\end{lemma}
	
	\begin{proof}
		Choose $l$ large enough such that $\cx(x_1)\sum_{i = l+1}^{\infty}{\cx(x_i)|h^\XF_{x_i}|}\leq \delta/6$. Then we see $$\begin{aligned} \norm{\hb^\XF - \hat\hb^\XF_{l}}_\lXSub{\cx} %
		\leq  \cx(x_1)\sum_{i = l+1}^{\infty}{|h^\XF_{x_i}|} +   \sum_{i = l+1}^{\infty}\cx(x_i) |h^\XF_{x_i}|\leq\frac{\delta}{3}.\end{aligned}$$ Let $N$ be large enough such that $\cx(x_1)\norm{\hb^\XF - \hb^\XF_{n}}_{\lXSub{\cx}}\leq \delta/3$ for all $n \geq N$. This yields 
		\begin{align*}\norm{\hat\hb^\XF_{l} - \hat\hb^\XF_{n,l}}_\lXSub{\cx} &= \cx(x_1) \left| h^\XF_{x_1} - h^\XF_{n, x_1} + \sum_{i = l+1}^{\infty}( h^\XF_{x_i} - h^\XF_{n, x_i})\right|+ \sum_{i = 2}^l\cx(x_i)\left|h^\XF_{x_i} -  h^\XF_{n, x_i}\right|\\
		&\leq\cx(x_1) \sum_{x \in \XF}\cx(x)\left| h^\XF_{x} -  h^\XF_{n, x}\right| = \cx(x_1)\norm{\hb^\XF - \hb^\XF_{n}}_\lXSub{\cx}\leq \frac{\delta}{3} ,
		\end{align*}
		which implies by triangle inequality 
		$$\begin{aligned}&\norm{\hb^\XF_{n} - \hat\hb^\XF_{n,l}}_\lXSub{\cx}\leq \norm{\hb^\XF_{n} - \hb^\XF}_\lXSub{\cx} + \norm{\hb^\XF - \hat\hb^\XF_{l}}_\lXSub{\cx} +\norm{\hat\hb^\XF_{l} - \hat\hb^\XF_{n,l}}_\lXSub{\cx}
		\leq  \delta,
		\end{aligned}$$
		and thus finishes the proof.
	\end{proof}

	\begin{lemma}\label{lem:UpperboundNorms}
	It holds that $$K \coloneqq \norm{\hb^\XC}_{\lXSub{\cx}} +\sup_{n\in \N}\norm{\hb^\XC_{n}}_{\lXSub{\cx}}  + \sup_{n,l\in \N}\norm{\hat\hb^\XC_{n,l}}_{\lXSub{\cx}}< \infty.$$ Further, for $\rho>0$ there is $N\in \N$ such that it follows for all $n \geq N$ and $l\in \N$ that $$ \rb + t_n \hb^\XC_n, 	 \rb + t_n \hat \hb^\XC_{n,l}, \rb + t_n \hat \hb^\XC_{l}\in \left\{ \tilde\rb \in \lXSub{\cx} \colon \norm{\tilde\rb -  \rb}_{\lXSub{\cx} } \leq \rho  \right\}. $$
	\end{lemma}
	
	\begin{proof}
	By convergence of $\hb^\XC_n$ towards $\hb^\XC$ with respect to $\lXSub{\cx}$ we see that $$\norm{\hb^\XC}_{\lXSub{\cx}} +\sup_{n\in \N}\norm{\hb^\XC_{n}}_{\lXSub{\cx}}< \infty.$$
	Further, for given $l,n\in \N$ it follows that $\norm{\hat\hb^\XC_{n,l}}_{\lXSub{\cx}} \leq \cx(x_1) \norm{\hb^\XC_{n}}_{\lXSub{\cx}}$, which yields $K< \infty$. For $N\in \N$ large enough such that $t_n< \rho K^{-1}$ we conclude the second claim. 
	\end{proof}

	With these auxiliary results we can state the proof of Lemma \ref{lem:GivenLChooseNBigSoAllProbMeasures}
	
	\begin{proof}[Proof of Lemma \ref{lem:GivenLChooseNBigSoAllProbMeasures}]
		First, we note that $\sum_{x \in \XC}h^\XC_{n,x} = 0$ for all $n\in \N$ implies $\sum_{x \in \XC}h^\XC_{x} = 0$. By construction of the finite support approximation, it also follows for any $l,n \in \N$ that $\sum_{x \in \XC}\hat h^\XC_{n,l,x} = \sum_{x \in \XC}\hat h^\XC_{l,x} = 0$. Hence, it remains to show for given $l\in \N$ that there exists $N\in \N$ such that the elements $\rb + t_n\hat\hb^\XF_{l},\rb + t_n\hat\hb^\XF_{n,l}$ are strictly positive in each entry. To this end, consider $K>0$ as in Lemma \ref{lem:UpperboundNorms} and choose $N\in \N$ large enough such that it holds for all $n\geq N$ that $t_n < (2K)^{-1}\min_{i = 1, \dots, l} \min({r}_{x_i}, 1 - {r}_{x_i})$. %
		 This yields for any $i\in \{1, \dots, l \}$ that 
		$${r}_{x_i} + t_n \hat{h}_{l, x_i}^\XF = {r}_{x_i} + t_n K \cdot \frac{\hat{h}_{l,x_i}^\XF}{K} \;\;\begin{cases}
			\leq {r}_{x_i} + \frac{1}{2}\min({r}_{x_i}, 1 - {r}_{x_i})\leq \frac{1}{2}{r}_{x_i}+ \frac{1}{2}<1, \\
			\geq {r}_{x_i} - \frac{1}{2}\min({r}_{x_i}, 1 - {r}_{x_i})\geq \frac{1}{2}{r}_{x_i}\;\;\;\;\;\;\,>0.\\
		\end{cases}$$
		Further, since $\hat h^\XF_x=0$ for  all $x\not\in \{x_1, \dots, x_l\}$ it follows for all $n\geq N$ that  ${r}_{x} + t_n \hat{h}_{l,x}^\XC = {r}_{x} \in (0,1)$, which shows that ${\rb} + t_n \hat{\hb   }^\XC_{l}\in \probset{\XC}$ and $\supp(\rb + t_n \hat{\hb   }^\XC_{l}) = \XC$. The same argument implies for all $n\geq N$ that ${\rb} + t_n \hat{\hb}^\XF_{n,l}\in \probset{\XC}$ and $\supp({\rb} + t_n \hat{\hb}^\XC_{n,l}) = \XC$.
	\end{proof}

	\section{Auxiliary Results for Banach spaces}\label{app:BanachSpacesNorms}

	In this section we prove correspondences between the different Banach spaces employed within this work. More precisely, for a positive function $f\colon \XC\rightarrow (0,\infty)$ on the countable space $\XC$ we consider the function class $\FC\coloneqq \{\tilde f\colon \XC\rightarrow \RR \colon |\tilde f|\leq f\}$ and %
	 relate the Banach spaces $\lXSub{f}$ and $\ell^\infty(\FC)$ to each other.

	\begin{lemma}\label{lem:equivalenceNorms}
		 The linear evaluation map $$e\colon \ell^1_f(\XC)\rightarrow \ell^\infty(\FC), \quad a \mapsto \left( e(a)\colon \FC\mapsto \RR,  \tilde f\mapsto \langle \tilde f, a\rangle \right)$$
		is an isometric embedding of $\ell^1_f(\XC)$ into $\ell^\infty(\FC)$, i.e., $ \norm{e(a)}_{\ell^\infty(\FC)}=\norm{a}_{\ell^1_f(\XC)}$ for $a \in \ell^1_f(\XC)$.
	\end{lemma}
	\begin{proof}
		First note for any $a \in \ell^1_f(\XC)$ that $e(a) \in \ell^\infty(\FC)$ since  $$\norm{e(a)}_{\FC}= \sup_{\tilde f\in \FC}\big|e(a)(\tilde f)\big|  = \sup_{\tilde f\in \FC}\big|\langle \tilde f, a\rangle \big|\leq \sup_{\tilde f\in \FC}\big\|\tilde f\big\|_{\lInfXSub{f}}\norm{a}_{\lXSub{f}} = \norm{a}_{\lXSub{f}}.$$
		Hence, the linear mapping $e$ is continuous. Moreover, it holds that 
		$$ \norm{a}_{\lXSub{f}} = \sum_{x\in \XC}f(x)|a_x| = \sum_{x\in \XC} \sign(a_x) f(x) a_x \leq \sup_{\tilde f\in \FC}\big|e(a)(\tilde f)\big| =\norm{e(a)}_{\FC},$$
		where the inequality holds since the function $x\mapsto \sign(a_x) f(x)$ is contained in $\FC$. 
	\end{proof}

		The previous lemma shows that the weighted $\ell^1$-space can be identified as a subset of $\ell^\infty(\FC)$. In particular, this implies that the intersections $\PC(\XC)\cap \ell^1_f(\XC)$ and $\PC(\XC)\cap \ell^\infty(\FC)$ can be identified. 
	Moreover, given another function class $\tilde \FC$ on $\XC$ with $\FC \subseteq \tilde \FC$ the following inclusion from the intersections $\PC(\XC)\cap \ell^\infty(\tilde \FC)$ into $\ell^1_f(\XC)$ can be formalized.

	\begin{lemma} \label{lem:InclusionMeasures}
	The canonical (linear) projection map 
	\begin{align*}
		\iota \colon (\PC(\XC)\cap \ell^\infty(\tilde \FC), \norm{\cdot}_{\ell^\infty(\tilde \FC)})\rightarrow (\PC(\XC)\cap \ell^\infty(\FC), \norm{\cdot}_{\ell^\infty(\FC)}), \quad (\langle f, \rb\rangle)_{f\in \FC_2} \mapsto (\langle f, \rb\rangle)_{f\in \FC_1} 
	\end{align*}
	is continuous and injective. In particular, $\PC(\XC)\cap \ell^\infty(\tilde \FC)$ can be embedded into $\ell^1_f(\XC)$.
	\end{lemma}
	
	\begin{proof}
		First note that any $\rb\in \PC(\XC)\cap \ell^\infty(\tilde \FC)$ fulfills since $\FC\subseteq \tilde \FC$ that $$\norm{\iota(\rb)}_{\ell^\infty(\FC)}=\sup_{f\in \FC}| \langle f, \rb\rangle| \leq \sup_{f\in \tilde \FC}| \langle f, \rb\rangle| = \norm{\rb}_{\ell^\infty(\tilde \FC)},$$
		which shows continuity. Moreover, if $\norm{\iota(\rb)-\iota(\sb)}_{\ell^\infty(\FC)} = \norm{\iota(\rb)-\iota(\sb)}_{\ell^1_f(\XC)} = 0$, and hence $\rb = \sb$ which shows injectivity. The second assertion then follows from the identification $\PC(\XC)\cap \ell^\infty(\FC)=\PC(\XC)\cap \ell^1_f(\XC)$ where latter is clearly contained in $\ell^1_f(\XC)$. 
	\end{proof}

	\end{appendix}

\end{document}